\documentclass[11pt,english]{amsart}
\usepackage[T1]{fontenc}
\usepackage[latin9]{inputenc}
\usepackage{prettyref}
\usepackage{mathrsfs}
\usepackage{mathtools}
\usepackage{amstext}
\usepackage{amsthm}
\usepackage{amssymb}
\usepackage{stmaryrd}
\usepackage{esint}

\makeatletter
\numberwithin{equation}{section}
\numberwithin{figure}{section}
\usepackage{enumitem}		
\theoremstyle{plain}
\newtheorem{thm}{\protect\theoremname}[section]
  \theoremstyle{definition}
  \newtheorem{defn}[thm]{\protect\definitionname}
  \theoremstyle{plain}
  \newtheorem{prop}[thm]{\protect\propositionname}
  \theoremstyle{plain}
  \newtheorem{lem}[thm]{\protect\lemmaname}
  \theoremstyle{plain}
  \newtheorem{cor}[thm]{\protect\corollaryname}

\makeatletter
\@addtoreset{thm}{section}
\makeatother
\usepackage{prettyref}
\usepackage{refstyle}
\usepackage{hyperref}
\usepackage{url}

\makeatother

\usepackage{babel}
  \providecommand{\corollaryname}{Corollary}
  \providecommand{\definitionname}{Definition}
  \providecommand{\lemmaname}{Lemma}
  \providecommand{\propositionname}{Proposition}
\providecommand{\theoremname}{Theorem}

\begin{document}

\title[magnetic Dirac operator]{Koszul complexes, Birkhoff normal form and the magnetic Dirac operator
 }

\author{Nikhil Savale}
\begin{abstract}
We consider the semi-classical Dirac operator coupled to a magnetic
potential on a large class of manifolds including all metric contact
manifolds. We prove a sharp local Weyl law and a bound on its eta
invariant. In the absence of a Fourier integral parametrix, the method
relies on the use of almost analytic continuations combined with the
Birkhoff normal form and local index theory.
\end{abstract}

\address{128 Hayes Healy, Department of Mathematics, University of Notre Dame,
Notre Dame 46556 IN.}

\email{nsavale@nd.edu}


\maketitle
$ $

\section{Introduction}

Semi-classical analysis concerns the study of the spectrum of semi-classical
($h$-)pseudodifferential operators $A_{h}:C^{\infty}\left(X\right)\rightarrow C^{\infty}\left(X\right)$,
$h\in\left(0,1\right]$, in the limit $h\rightarrow0$ and is now
the subject of several texts \cite{Dimassi-Sjostrand,GuilleminSternberg-Semiclassical,Ivrii2013-newbook,Ivrii98-oldbook,Fedoriuk-Maslov,Robert-book,Zworski}.
Standard examples of such operators include the Schroedinger operator
$A_{h}=-h^{2}\Delta_{X}+V$ on a compact $n$-dimensional Riemannian
manifold $X$ with potential $V\in C^{\infty}\left(X\right)$. The
clearest asymptotic result is given by the celebrated local Weyl law
(cf. eg. \cite{Dimassi-Sjostrand} Ch. 10): assuming $0$ is not a
critical value of the symbol $\sigma\left(A\right)=a\left(x,\xi\right)\in C^{\infty}\left(T^{*}X\right)$,
the number of eigenvalues $N\left(-ch,ch\right)$ of $A_{h}$ in the
interval $\left(-ch,ch\right)$ satisfies 
\begin{equation}
N\left(-ch,ch\right)=O\left(h^{-n+1}\right)\label{eq:local Weyl law}
\end{equation}
as $h\rightarrow0$, $\forall c>0$. Similar results also exist in
the case where $0$ is a Morse-Bott critical level for the symbol
(cf. \cite{Brummelhuis-Paul-Uribe}). In the critical case, the exponent
in the local Weyl law may drop depending on the co-dimension of zero
energy level $\Sigma_{0}^{A}\coloneqq\left\{ a\left(x,\xi\right)=0\right\} $
and the signature of the normal Hessian. The local Weyl laws thus
obtained are sharp and are proved using a parametrix construction
for the evolution operator $e^{\frac{it}{h}A_{h}}$ as a Fourier integral
operator. 

In the context of non-scalar operators $A_{h}:C^{\infty}\left(X;E\right)\rightarrow C^{\infty}\left(X;E\right)$
acting on sections of a vector bundle $E$, fewer result are known.
The simplest case is when the non-scalar symbol $a\left(x,\xi\right)\in C^{\infty}\left(T^{*}X;E\right)$
is smoothly diagonalizable near the zero energy level $\Sigma_{0}^{A}=\left\{ \det\left(a\left(x,\xi\right)\right)=0\right\} $.
In this case similar Fourier integral methods apply (cf. \cite{Emmrich-Weinstein,Fedoriuk-Maslov}
or \cite{Guillemin,Sandoval} for an exposition in the microlocal/classical
setting). For non-scalar operators another method is provided under
the microhyperbolicity condition of Ivrii (cf. \cite{Ivrii98-oldbook}
Ch. 2,3 or \cite{Dimassi-Sjostrand} Ch. 12). In this paper, we study
the particular case of the magnetic Dirac operator where neither diagonalizability
nor the microhyperbolicity condition is satisfied.

More precisely, let $\left(X,g^{TX}\right)$ be an oriented Riemannian
manifold of odd dimension $n=2m+1$ equipped with a spin structure.
Let $S$ be the corresponding spin bundle and let $L$ be an auxiliary
Hermitian line bundle. Fix a unitary connection $A_{0}$ on $L$ and
let $a\in\Omega^{1}\left(X;\mathbb{R}\right)$ be a one form. This
gives a family of unitary connections on $L$ via $\nabla^{h}=A_{0}+\frac{i}{h}a$
and a corresponding family of coupled magnetic Dirac operators 
\begin{equation}
D_{h}\coloneqq hD_{A_{0}}+ic\left(a\right)\label{eq:Semiclassical Magnetic Dirac}
\end{equation}
for $h\in\left(0,1\right]$. 

In order to derive sharp spectral asymptotics, we shall make a couple
of restrictive assumptions on the one form $a$ and the metric $g^{TX}$.
First, the one form $a$ will be assumed to be a contact one form
(i.e. one satisfying $a\wedge\left(da\right)^{m}>0$). This gives
rise to the contact hyperplane $H=\textrm{ker}\left(a\right)\subset TX$
as well as the Reeb vector field $R$ defined via $i_{R}da=0$, $i_{R}a=1$. 

To state the assumption on the metric, consider the contracted endomorphism
$\mathfrak{J}:T_{x}X\rightarrow T_{x}X$ defined at each point $x\in X$
via 
\[
da\left(v_{1},v_{2}\right)=g^{TX}\left(v_{1},\mathfrak{J}v_{2}\right),\quad\forall v_{1},v_{2}\in T_{x}X.
\]
From the contact assumption, $\mathfrak{J}$ has a one dimensional
kernel spanned by the Reeb vector field $R$. The endomorphism $\mathfrak{J}$
is clearly anti-symmetric with respect to the metric 
\[
g^{TX}\left(v_{1},\mathfrak{J}v_{2}\right)=-g^{TX}\left(\mathfrak{J}v_{1},v_{2}\right)
\]
and hence its non-zero eigenvalues come in purely imaginary pairs
$\pm i\mu$ ; $\mu>0$. The assumption on the metric $g^{TX}$is then
as follows. 
\begin{defn}
\label{def: Diagonalizability assumption}We say that the metric $g^{TX}$
is $suitable$ to the contact form $a$ if there exist positive constants
$0<\mu_{1}\leq\mu_{2}\leq\ldots\leq\mu_{m}$ (independent of $x\in X$)
and a positive real function $\nu\left(x\right)>0$ such that 
\begin{equation}
\textrm{Spec}\left(\mathfrak{J}_{x}\right)=\left\{ 0,\pm i\mu_{1}\nu\left(x\right),\pm i\mu_{2}\nu\left(x\right),\ldots,\pm i\mu_{m}\nu\left(x\right)\right\} \label{eq:Diagonalizability assumption-1}
\end{equation}
$\forall x\in X$. 
\end{defn}
Before proceeding further, we give two examples of suitable metrics.
\begin{enumerate}
\item The dimension of the manifold $\textrm{dim }X=3$. In this case any
metric $g^{TX}$ is suitable as $\textrm{Spec}\left(\mathfrak{J}_{x}\right)=\left\{ 0,\pm i\left|da\right|\right\} $
has only two non-zero eigenvalues.
\item There is a smooth endomorphism $J:TX\rightarrow TX$, such that\\
 $\left(X^{2m+1},a,g^{TX},J\right)$ is a metric contact manifold.
That is, we have 
\begin{eqnarray}
J^{2}v_{1} & = & -v_{1}+a\left(v_{1}\right)R,\nonumber \\
g^{TX}\left(v_{1},Jv_{2}\right) & = & da\left(v_{1},v_{2}\right),\quad\forall v_{1},v_{2}\in T_{x}X.\label{eq: metric contact structure}
\end{eqnarray}
 In this case the nonzero eigenvalues of $\mathfrak{J}_{x}=J_{x}$
are $\pm i$ (each with multiplicity $m$). For any given contact
form $a$ there exists an infinite dimensional space of $\left(g^{TX},J\right)$
satisfying \prettyref{eq: metric contact structure}. This case in
particular includes all strictly pseudo-convex CR manifolds.
\end{enumerate}
In addition to the local Weyl law we shall also be interested in the
asymptotics of the eta invariant $\eta_{h}=\eta\left(D_{h}\right)$
of the Dirac operator, formally its signature (see \prettyref{sub:Spectral-invariants-of}
for a definition) . The main result is now stated as follows.
\begin{thm}
\label{thm: asmptotics spectral invariants}Under the contact and
suitability assumptions on $a,g^{TX}$, the local Weyl counting function
and eta invariant of $D_{h}$ satisfy the sharp asymptotics
\begin{eqnarray}
N\left(-ch,ch\right) & = & O\left(h^{-m}\right)\label{eq:local Weyl counting function est}\\
\eta_{h} & = & O\left(h^{-m}\right)\label{eq: eta estimate}
\end{eqnarray}
as $h\rightarrow0$.
\end{thm}
We note that the exponents above are significantly lower than \prettyref{eq:local Weyl law}.
This is again partly attributed to the high co-dimension of $\Sigma_{0}^{D}$. 

The proof of the asymptotic result \prettyref{thm: asmptotics spectral invariants}
above will be based on a functional trace expansion. To state the
trace expansion involved, set $\nu_{0}\coloneqq\mu_{1}\left[\min_{x\in X}\nu\left(x\right)\right]$
and choose $f\in C_{c}^{\infty}\left(-\sqrt{2\nu_{0}},\sqrt{2\nu_{0}}\right)$.
Pick real numbers $0<T'<T$ and let $\theta\in C_{c}^{\infty}\left(\left(-T,T\right);\left[0,1\right]\right)$
such that $\theta\left(x\right)=1$ on $\left(-T',T'\right)$. Let
\begin{eqnarray*}
\mathcal{F}^{-1}\theta\left(x\right) & \coloneqq & \check{\theta}\left(x\right)=\frac{1}{2\pi}\int e^{ix\xi}\theta\left(\xi\right)d\xi\\
\mathcal{F}_{h}^{-1}\theta\left(x\right) & \coloneqq & \frac{1}{h}\check{\theta}\left(\frac{x}{h}\right)=\frac{1}{2\pi h}\int e^{\frac{i}{h}x\xi}\theta\left(\xi\right)d\xi
\end{eqnarray*}
be its classical and semi-classical inverse Fourier transforms respectively.
We shall then prove. 
\begin{thm}
\label{thm:main trace expansion}Let $a,g^{TX}$ be a contact form
and suitable metric respectively. There exist smooth functions $u_{j}\in C^{\infty}\left(\mathbb{R}\right)$
such that there is a trace expansion 
\begin{align}
\textrm{tr}\left[f\left(\frac{D}{\sqrt{h}}\right)\left(\mathcal{F}_{h}^{-1}\theta\right)\left(\lambda\sqrt{h}-D\right)\right] & =\label{eq: Main trace expansion}\\
\textrm{tr}\left[f\left(\frac{D}{\sqrt{h}}\right)\frac{1}{h}\check{\theta}\left(\frac{\lambda\sqrt{h}-D}{h}\right)\right] & =h^{-m-1}\left(\sum_{j=0}^{N-1}f\left(\lambda\right)u_{j}\left(\lambda\right)h^{j/2}+O\left(h^{N/2}\right)\right)
\end{align}
 for $T$ sufficiently small and for each $N\in\mathbb{N}$,$\lambda\in\mathbb{R}$. 
\end{thm}
Again, the trace \prettyref{eq: Main trace expansion} should be compared
with the wave trace expansions for scalar and microhyperbolic operators
(\cite{Dimassi-Sjostrand} ch. 10, 12) although a different scale
of size $\sqrt{h}$ is being used. In the absence of a Fourier integral
parametrix or microhyperbolicity our strategy is to combine the use
of almost analytic continuations with local index theory expansions.
We first show that the trace is $O\left(h^{\infty}\right)$ in the
region $\textrm{spt}\left(\theta\right)\subset\left\{ T>\left|x\right|\geq h^{\epsilon}\right\} $,
$\epsilon\in\left(0,\frac{1}{2}\right)$ (see \prettyref{lem: O(h infty) LEMMA}).
Here the the lack of microhyperbolicity for the symbol poses a difficulty
in the use of almost analytic continuations (cf. \cite{Dimassi-Sjostrand}
ch. 12, see also \cite{Dimassi-Sjostrand--article}). We however show
that this can be overcome with a closer understanding of the total
symbol of $D$ via its Birkhoff normal form. It is in deriving the
Birkhoff normal form then that Koszul complexes are used and the assumptions
on $a,g^{TX}$ required. The local index theory method (cf. \cite{Bismut,Ma-Marinescu})
finally provides the expansion in the region $\textrm{spt}\left(\theta\right)\subset\left\{ \left|x\right|<h^{\epsilon}\right\} $
(see \prettyref{lem: Easy trace expansion lemma}).

There is a large recent literature for semi-classical problems in
the presence of magnetic fields (see \cite{Helffer-Kordyukov} for
a survey). In particular the extensive book of Ivrii \cite{Ivrii2013-newbook}
specifically considers the case of the magnetic Dirac operator in
ch. 17. The Birkhoff normal form here \prettyref{eq: normal form d1}
generalizes proposition 17.2.1 therein. Our use of normal forms should
also be compared to its use in scalar cases from \cite{Charles-Ngoc,Helffer-Kordyukov-Raymond-Ngoc,Ngoc-Raymond}. 

The asymptotic problem of the eta invariant \prettyref{eq: eta estimate}
was earlier considered by the author in \cite{Savale-Asmptotics}
where a non-sharp estimate was proved, under no assumptions on $a,g^{TX}$,
via the use of the heat trace. This asymptotic problem was first considered
and applied in \cite{Taubes-Weinstein} in the proof of the three-dimensional
Weinstein conjecture using Seiberg-Witten theory. The three-dimensional
case has been further explored in \cite{Tsai-thesis-paper}. 

The paper is organized as follows. In \prettyref{sec:Preliminaries}
begin with preliminary notions used throughout the paper including
basic facts about Clifford representations, Dirac operators and the
semi-classical calculus. In \prettyref{sub:Magnetic Dirac operator Rm}
we we compute the spectrum of a model magnetic Dirac operator on $\mathbb{R}^{m}$
using Clifford representations and the harmonic oscillator. In \prettyref{sec:First Reductions}
we perform certain reductions towards proving \prettyref{thm:main trace expansion}
including a time scale breakup of the trace into \prettyref{lem: O(h infty) LEMMA}
and \prettyref{lem: Easy trace expansion lemma}. These reductions
are then used in \prettyref{sec:Reduction to R^n} to further reduce
\prettyref{lem: O(h infty) LEMMA} to the case of a Euclidean magnetic
Dirac operator on $\mathbb{R}^{n}$. In \prettyref{sec: Birkhoff normal form}
we obtain the Birkhoff normal form for the Euclidean magnetic Dirac
operator on $\mathbb{R}^{n}$ from \prettyref{sec:Reduction to R^n}.
It is here in \prettyref{sub: Weyl product and Koszul} that Koszul
complexes are employed for the normal form. In \prettyref{sec: Extension of a resolvent}
we show how the normal form is used in proving \prettyref{lem: O(h infty) LEMMA}
via the use of almost analytic continuations. In \prettyref{sec:Local trace expansion}
we prove \prettyref{lem: Easy trace expansion lemma} using the methods
of local index theory. In \prettyref{sec:Asymptotics-of-spectral invariants}
we show how to prove the spectral estimates of \prettyref{thm: asmptotics spectral invariants}
via the trace expansion \prettyref{thm:main trace expansion}. Finally,
in \prettyref{sec:Appendix A} we prove some spectral estimates useful
in \prettyref{sec:Reduction to R^n} and \prettyref{sec: Birkhoff normal form}.

\section{\label{sec:Preliminaries}Preliminaries}

\subsection{\label{sub:Spectral-invariants-of}Spectral invariants of the Dirac
operator}

\noindent Here we review the basic facts about Dirac operators used
throughout the paper with \cite{Berline-Getzler-Vergne} providing
a standard reference. Consider a compact, oriented, Riemannian manifold
$\left(X,g^{TX}\right)$ of odd dimension $n=2m+1$. Let $X$ be equipped
with spin structure, i.e. a principal $\textrm{Spin}\left(n\right)$
bundle $\textrm{Spin}\left(TX\right)\rightarrow SO\left(TX\right)$
with an equivariant double covering of the principal $SO\left(n\right)$-bundle
of orthonormal frames $SO\left(TX\right)$. The corresponding spin
bundle $S=\textrm{Spin}\left(TX\right)\times_{\textrm{Spin}\left(n\right)}S_{2m}$
is associated to the unique irreducible representation of $\textrm{Spin}\left(n\right)$.
Let $\nabla^{TX}$ denote the Levi-Civita connection on $TX$. This
lifts to the spin connection $\nabla^{S}$ on the spin bundle $S$.
The Clifford multiplication endomorphism $c:T^{*}X\rightarrow S\otimes S^{*}$
may be defined (see \prettyref{sub:Clifford algebra}) satisfying
\begin{align*}
c(a)^{2}=-|a|^{2}, & \quad\forall a\in T^{*}Y.
\end{align*}
Let $L$ be a Hermitian line bundle on $Y$. Let $A_{0}$ be a fixed
unitary connection on $L$ and let $a\in\Omega^{1}(Y;\mathbb{R})$
be a 1-form on $Y$. This gives a family $\nabla^{h}=A_{0}+\frac{i}{h}a$
of unitary connections on $L$. We denote by $\nabla^{S\otimes L}=\nabla^{S}\otimes1+1\otimes\nabla^{h}$
the tensor product connection on $S\otimes L.$ Each such connection
defines a coupled Dirac operator 
\begin{align*}
D_{h}\coloneqq hD_{A_{0}}+ic\left(a\right)=hc\circ\left(\nabla^{S\otimes L}\right):C^{\infty}(Y;S\otimes L)\rightarrow C^{\infty}(Y;S\otimes L)
\end{align*}
for $h\in\left(0,1\right]$. Each Dirac operator $D_{h}$ is elliptic
and self-adjoint. It hence possesses a discrete spectrum of eigenvalues. 

We define the eta function of $D_{h}$ by the formula
\begin{align}
\eta\left(D_{h},s\right)\coloneqq & \sum_{\begin{subarray}{l}
\quad\:\lambda\neq0\\
\lambda\in\textrm{Spec}\left(D_{h}\right)
\end{subarray}}\textrm{sign}(\lambda)|\lambda|^{-s}=\frac{1}{\Gamma\left(\frac{s+1}{2}\right)}\int_{0}^{\infty}t^{\frac{s-1}{2}}\textrm{tr}\left(D_{h}e^{-tD_{h}^{2}}\right)dt.\label{eq:eta invariant definition}
\end{align}
Here, and in the remainder of the paper, we use the convention that
$\textrm{Spec}(D_{h})$ denotes a multiset with each eigenvalue of
$D_{h}$ being counted with its multiplicity. The above series converges
for $\textrm{Re}(s)>n.$ It was shown in \cite{APSI,APSIII} that
the eta function possesses a meromorphic continuation to the entire
complex $s$-plane and has no pole at zero. Its value at zero is defined
to be the eta invariant of the Dirac operator
\[
\eta_{h}\coloneqq\eta\left(D_{h},0\right).
\]
By including the zero eigenvalue in \prettyref{eq:eta invariant definition},
with an appropriate convention, we may define a variant known as the
reduced eta invariant by 
\begin{align*}
\bar{\eta}_{h}\coloneqq & \frac{1}{2}\left\{ k_{h}+\eta_{h}\right\} .
\end{align*}

The eta invariant is unchanged under positive scaling
\begin{equation}
\eta\left(D_{h},0\right)=\eta\left(cD_{h},0\right);\quad\forall c>0.\label{eq: eta scale invariant}
\end{equation}
Let $L_{t,h}$ denote the Schwartz kernel of the operator $D_{h}e^{-tD_{h}^{2}}$
on the product $X\times X$. Throughout the paper all Schwartz kernels
will be defined with respect to the Riemannian volume density. Denote
by $\textrm{tr}\left(L_{t,h}\left(x,x\right)\right)$ the point-wise
trace of $L_{t,h}$ along the diagonal. We may now analogously define
the function 
\begin{align}
\eta\left(D_{h},s,x\right)= & \frac{1}{\Gamma\left(\frac{s+1}{2}\right)}\int_{0}^{\infty}t^{\frac{s-1}{2}}\textrm{tr}\left(L_{t,h}\left(x,x\right)\right)dt.\label{eq:eta function diagonal}
\end{align}
In \cite{Bismut-Freed-II} theorem 2.6, it was shown that for $\textrm{Re}(s)>-2$,
the function $\eta\left(D_{h},s,x\right)$ is holomorphic in $s$
and smooth in $x$. From \prettyref{eq:eta function diagonal} it
is clear that this is equivalent to 
\begin{align}
\textrm{tr}\left(L_{t,h}\right)= & O\left(t^{\frac{1}{2}}\right),\quad\textrm{as}\:t\rightarrow0.\label{eq:pointwise trace asymp as t->0}
\end{align}
The eta invariant is then given by the convergent integral 
\begin{equation}
\eta_{h}=\int_{0}^{\infty}\frac{1}{\sqrt{\pi t}}\textrm{tr}\left(D_{h}e^{-tD_{h}^{2}}\right)dt.\label{eq: eta integral}
\end{equation}

\subsection{\label{sub:Clifford algebra}Clifford algebra and and its representations}

Here we review the construction of the spin representation of the
Clifford algebra. The following being standard, is merely used to
setup our conventions and subsequently compute the spectrum of the
model magnetic Dirac operator on $\mathbb{R}^{m}$ in \prettyref{sub:Magnetic Dirac operator Rm}. 

Consider a real vector space $V$ of even dimension $2m$ with metric
$\left\langle ,\right\rangle $. Recall that its Clifford algebra
$Cl\left(V\right)$ is defined as the quotient of the tensor algebra
$T\left(V\right):=\oplus_{j=0}^{\infty}V^{\otimes j}$ by the ideal
generated from the relations $v\otimes v+\left|v\right|^{2}=0$. Fix
a compatible almost complex structure $J$ and split $V\otimes\mathbb{C}=V^{1,0}\oplus V^{0,1}$
into the $\pm i$ eigenspaces of $J$. The complexification $V\otimes\mathbb{C}$
carries an induced $\mathbb{C}$-bilinear inner product $\left\langle ,\right\rangle _{\mathbb{C}}$
as well as an induced Hermitian inner product $h^{\mathbb{C}}\left(,\right)$.
Next, define $S_{2m}=\Lambda^{*}V^{1,0}$. Clearly $S_{2m}$ is a
complex vector space of dimension $2^{m}$ on which the unique irreducible
(spin)-representation of the Clifford algebra $Cl\left(V\right)\otimes\mathbb{C}$
is defined by the rule
\[
c_{2m}\left(v\right)\omega=\sqrt{2}\left(v^{1,0}\wedge\omega-\iota_{v^{0,1}}\omega\right),\quad v\in V,\omega\in S_{2m}.
\]
The contraction above is taken with respect to $\left\langle ,\right\rangle _{\mathbb{C}}$.
It is clear that $c_{2m}\left(v\right):\Lambda^{\textrm{even/odd}}\rightarrow\Lambda^{\textrm{odd/even}}$
switches the odd and even factors. For the Clifford algebra $Cl\left(W\right)\otimes\mathbb{C}$
of an odd dimensional vector space $W=V\oplus\mathbb{R}\left[e_{0}\right]$
there are exactly two irreducible representations. These two (spin)-representations
$S_{2m+1}^{+}=S_{2m+1}^{-}=\Lambda^{*}V^{1,0}$ are defined via 
\begin{eqnarray}
c_{2m+1}^{\pm}\left(v\right) & = & c_{2m}\left(v\right),\quad v\in V\nonumber \\
c_{2m+1}^{+}\left(e_{0}\right)\omega_{\textrm{even/odd}} & = & -c_{2m+1}^{-}\left(e_{0}\right)\omega_{\textrm{even/odd}}=\pm i\omega_{\textrm{even/odd}}.\label{eq:odd clifford representation}
\end{eqnarray}
Throughout the rest of the paper, we stick with the positive convention
and use the shorthands $c=c_{2m}$, $c=c_{2m+1}^{+}$ when the index
$2m$, $2m+1$ implicitly understood.

Pick an orthonormal basis $e_{1},e_{2},\ldots,e_{2m}$ for $V$ in
which the almost complex structure is given by $Je_{2j-1}=e_{2j}$,
$1\leq j\leq m$. An $h^{\mathbb{C}}$-orthonormal basis for $V^{1,0}$
is now given by $w_{j}=\frac{1}{\sqrt{2}}\left(e_{2j}+ie_{2j-1}\right)$,
$1\leq j\leq m$. A basis for $S_{2m}$ and $S_{2m+1}^{\pm}$ is given
by $w_{k}=w_{1}^{k_{1}}\wedge\ldots\wedge w_{m}^{k_{m}}$ with $k=\left(k_{1},k_{2},\ldots,k_{m}\right)\in\left\{ 0,1\right\} ^{m}$.
Ordering the above chosen bases lexicographically in $k$, we may
define the Clifford matrices, of rank $2^{m}$, via 
\begin{eqnarray*}
\gamma_{j}^{m} & = & c\left(e_{j}\right),\quad0\leq j\leq2m,
\end{eqnarray*}
for each $m$. Again, we often write $\gamma_{j}^{m}=\gamma_{j}$
with the index $m$ implicitly understood. Giving representations
of the Clifford algebra, these matrices satisfy the relation 
\begin{equation}
\gamma_{i}\gamma_{j}+\gamma_{j}\gamma_{i}=-2\delta_{ij}.\label{eq: Clifford relations}
\end{equation}

Next, one may further define the Clifford quantization map on the
exterior algebra
\begin{eqnarray}
c:\Lambda^{*}W\otimes\mathbb{C} & \rightarrow & \textrm{End}\left(S_{2m}\right)\nonumber \\
c\left(e_{0}^{k_{0}}\wedge\ldots\wedge e_{2m}^{k_{2m}}\right) & = & c\left(e_{0}\right)^{k}\ldots c\left(e_{2m}\right)^{k_{2m}}.\label{eq: clifford quantization}
\end{eqnarray}
An easy computation yields 
\begin{eqnarray*}
c\left(e_{0}\wedge\ldots\wedge e_{2m}\right) & = & i^{m+1}.
\end{eqnarray*}
Furthermore, if $e_{0}\wedge\ldots\wedge e_{2m}$ is designated to
give a positive orientation for $W$ then for $\omega\in\Lambda^{k}W$
we have 
\begin{eqnarray}
c\left(\ast\omega\right) & = & i^{m+1}\left(-1\right)^{\frac{k\left(k+1\right)}{2}}c\left(\omega\right)\label{eq:clifford quantization hodge dual}\\
c\left(\omega\right)^{*} & = & \left(-1\right)^{\frac{k\left(k+1\right)}{2}}c\left(\omega\right)\label{eq:clifford quantization adjoint}
\end{eqnarray}
under the Hodge star and $h^{\mathbb{C}}$-adjoint. The Clifford quantization
map \prettyref{eq: clifford quantization} is a linear surjection
with kernel spanned by elements of the form $\ast\omega-i^{m+1}\left(-1\right)^{\frac{k\left(k+1\right)}{2}}\omega$.
Thus, in particular one has linear isomorphisms 
\begin{equation}
c:\Lambda^{\textrm{even/odd}}W\otimes\mathbb{C}\rightarrow\textrm{End}\left(S_{2m}\right).\label{eq:clifford algebra is matrix algebra}
\end{equation}

Next, given $\left(r_{1},\ldots,r_{m}\right)\in\mathbb{R}^{m}\setminus0$,
we define
\begin{eqnarray}
I_{r} & \coloneqq & \left\{ j|r_{j}\neq0\right\} \subset\left\{ 1,2,\ldots,m\right\} \label{eq: I_r}\\
Z_{r} & \coloneqq & \left|I_{r}\right|\label{eq: Z_r}\\
V_{r} & \coloneqq & \bigoplus_{j\in I_{r}}\mathbb{C}\left[w_{j}\right]\subset V^{1,0}\label{eq: V_r}\\
\textrm{and }\quad w_{r} & \coloneqq & \sum_{j=1}^{m}r_{j}w_{j}\in V_{r}.\label{eq: w_r}
\end{eqnarray}
Clearly, $\left\Vert w_{r}\right\Vert =\left|r\right|$. Denoting
by $w_{r}^{\perp}$ the $h^{\mathbb{C}}$-orthogonal complement of
$w_{r}\subset V_{r}$, one clearly has $V_{r}=\mathbb{C}\left[w_{r}\right]\oplus w_{r}^{\perp}$.
Hence 
\begin{eqnarray}
\Lambda^{\textrm{even}}V_{r} & = & \left(\Lambda^{\textrm{even}}w_{r}^{\perp}\right)\oplus\frac{w_{r}}{\left|r\right|}\wedge\left(\Lambda^{\textrm{odd}}w_{r}^{\perp}\right)\nonumber \\
\Lambda^{\textrm{odd}}V_{r} & = & \left(\Lambda^{\textrm{odd}}w_{r}^{\perp}\right)\oplus\frac{w_{r}}{\left|r\right|}\wedge\left(\Lambda^{\textrm{even}}w_{r}^{\perp}\right).\label{eq: odd even decompositions}
\end{eqnarray}
Next, we define 
\begin{eqnarray}
\mathtt{i}_{r}:\Lambda^{*}V_{r} & \rightarrow & \Lambda^{*}V_{r},\quad\textrm{ via}\label{eq: involution}\\
\mathtt{i}_{r}\left(\omega\right) & \coloneqq & \frac{w_{r}}{\left|r\right|}\wedge\omega\nonumber \\
\mathtt{i}_{r}\left(\frac{w_{r}}{\left|r\right|}\wedge\omega\right) & \coloneqq & \omega\nonumber 
\end{eqnarray}
for $\omega\in\Lambda^{*}w_{r}^{\perp}$. Clearly, $\mathtt{i}_{r}^{2}=1$
with the decomposition \prettyref{eq: odd even decompositions} implying
that $\mathtt{i}_{r}$ is a linear isomorphism between 
\begin{eqnarray*}
\mathtt{i}_{r}:\Lambda^{\textrm{even}}V_{r} & \rightarrow & \Lambda^{\textrm{odd}}V_{r}\\
\mathtt{i}_{r}:\Lambda^{\textrm{odd}}V_{r} & \rightarrow & \Lambda^{\textrm{even}}V_{r}.
\end{eqnarray*}
Next, the endomorphism 
\begin{eqnarray}
c\left(\frac{w_{r}-\bar{w}_{r}}{\sqrt{2}}\right)=\left(w_{r}\wedge+\iota_{\bar{w}_{r}}\right):\Lambda^{\textrm{*}}V_{r} & \rightarrow & \Lambda^{\textrm{*}}V_{r}\label{eq: clifford multiplication partial vector}
\end{eqnarray}
has the form
\begin{equation}
c\left(\frac{w_{r}-\bar{w}_{r}}{\sqrt{2}}\right)=\begin{bmatrix} & \left|r\right|\mathtt{i}_{r}\\
\left|r\right|\mathtt{i}_{r}
\end{bmatrix}\label{eq: clifford multiplication blocks}
\end{equation}
with respect to the decomposition $\Lambda^{\textrm{*}}V_{r}=\Lambda^{\textrm{odd}}V_{r}\oplus\Lambda^{\textrm{even}}V_{r}$.
This finally allows us to write the eigenspaces of \prettyref{eq: clifford multiplication partial vector}
as 
\begin{equation}
V_{r}^{\pm}=\left(1\pm\mathtt{i}_{r}\right)\left(\Lambda^{\textrm{even}}V_{r}\right)\label{eq: eigenspaces clifford mult.}
\end{equation}
with eigenvalue $\pm\left|r\right|$ respectively.

\subsubsection{\label{sub:Magnetic Dirac operator Rm}Magnetic Dirac operator on
$\mathbb{R}^{m}$}

We now define the magnetic Dirac operator on $\mathbb{R}^{m}$ via
\begin{equation}
D_{\mathbb{R}^{m}}=\sum_{j=1}^{m}\left(\frac{\mu_{j}}{2}\right)^{\frac{1}{2}}\left[\gamma_{2j}\left(h\partial_{x_{j}}\right)+i\gamma_{2j-1}x_{j}\right]\in\Psi_{\textrm{cl}}^{1}\left(\mathbb{R}^{m};\mathbb{C}^{2^{m}}\right).\label{eq: magnetic Dirac Rm}
\end{equation}
Its square is computed in terms of the harmonic oscillator 
\begin{eqnarray}
D_{\mathbb{R}^{m}}^{2} & = & \mathtt{H}_{2}-ih\mathtt{R}_{2m+1},\:\textrm{with}\label{eq:square Euclidean Dirac}\\
\mathtt{H}_{2} & =\frac{1}{2} & \sum_{j=1}^{m}\mu_{j}\left[-\left(h\partial_{x_{j}}\right)^{2}+x_{j}^{2}\right]\label{eq:Harmonic oscillator}\\
\mathtt{R}_{2m+1} & =\frac{1}{2} & \sum_{j=1}^{m}\mu_{j}\left[\gamma_{2j-1}\gamma_{2j}\right].\nonumber 
\end{eqnarray}
It is an easy exercise to show that 
\begin{equation}
\mathtt{R}_{2m+1}w_{k}=\frac{i}{2}\left[\sum_{j=1}^{m}\left(-1\right)^{k_{j}-1}\mu_{j}\right]w_{k}.\label{eq:curvature operator formula}
\end{equation}
Next, define the lowering and raising operators $A_{j}=h\partial_{x_{j}}+x_{j},$
$A_{j}^{*}=-h\partial_{x_{j}}+x_{j}$ for $1\leq j\leq m$, and the
Hermite functions
\begin{align}
\psi_{\tau,k}\left(x\right) & \coloneqq\psi_{\tau}\left(x\right)\otimes w_{k}\nonumber \\
\psi_{\tau}\left(x\right) & \coloneqq\frac{1}{\left(\pi h\right)^{\frac{m}{4}}\left(2h\right)^{\frac{\left|\tau\right|}{2}}\sqrt{\tau!}}\left[\Pi_{j=1}^{m}\left(A_{j}^{*}\right)^{\tau_{j}}\right]e^{-\frac{\left|x\right|^{2}}{2h}},\label{eq: Hermite functions}\\
 & \qquad\qquad\qquad\qquad\qquad\textrm{for }\tau=\left(\tau_{1},\tau_{2},\ldots,\tau_{m}\right)\in\mathbb{N}_{0}^{m}.\nonumber 
\end{align}
It is well known that $\psi_{\tau,k}\left(x\right)$ form an orthonormal
basis for $L^{2}\left(\mathbb{R}^{m};\mathbb{C}^{2^{m}}\right)$.
Furthermore we have the standard relations 
\begin{eqnarray}
\left[A_{j},A_{j}^{*}\right] & = & 2h\nonumber \\
\mathtt{H}_{2} & = & \frac{1}{2}\sum_{j=1}^{m}\mu_{j}\left(A_{j}A_{j}^{*}-1\right).\label{eq:standard commutation}
\end{eqnarray}
It is clear from \prettyref{eq:square Euclidean Dirac}, \prettyref{eq:curvature operator formula}
and \prettyref{eq:standard commutation} that each $\psi_{\tau,k}\left(x\right)$
is an eigenvector of $D_{\mathbb{R}^{m}}^{2}$ with eigenvalue 
\[
\lambda_{\tau,k}=h\sum_{j=1}^{m}\left(2\tau_{j}+1+\left(-1\right)^{k_{j}-1}\right)\frac{\mu_{j}}{2}.
\]

Hence, clearly the kernel of $D_{\mathbb{R}^{m}}$ is one dimensional
and spanned by $\psi_{0,0}=e^{-\frac{\left|x\right|^{2}}{2h}}$. We
now find a decomposition of $L^{2}\left(\mathbb{R}^{m};\mathbb{C}^{2^{m}}\right)$
into eigenspaces of $D_{\mathbb{R}^{m}}$ . First, if we define
\begin{equation}
\overline{\partial}=\frac{1}{2}\sum_{j=1}^{m}\left(\frac{\mu_{j}}{2}\right)^{\frac{1}{2}}c\left(w_{j}\right)A_{j},\label{eq: dbar}
\end{equation}
then one quickly computes
\begin{equation}
\overline{\partial}^{*}=-\frac{1}{2}\sum_{j=1}^{m}\left(\frac{\mu_{j}}{2}\right)^{\frac{1}{2}}c\left(\overline{w}_{j}\right)A_{j}^{*}\label{eq: dbar*}
\end{equation}
and 
\begin{equation}
D_{\mathbb{R}^{m}}=\sqrt{2}\left(\overline{\partial}+\overline{\partial}^{*}\right).\label{eq: dolbeault Dirac}
\end{equation}
For each $\tau\in\mathbb{N}_{0}^{m}\setminus0$, we define $I_{\tau}$,
$V_{\tau}$ as in \prettyref{eq: I_r}, \prettyref{eq: V_r} and set
\[
E_{\tau}\coloneqq\bigoplus_{b\in\left\{ 0,1\right\} ^{I_{\tau}}}\mathbb{C}\left[\prod_{j\in I_{\tau}}\left(\frac{c\left(w_{j}\right)A_{j}}{\sqrt{2\tau_{j}h}}\right)^{b_{j}}\psi_{\tau,0}\right].
\]
It is clear that we have an orthogonal decomposition 
\[
L^{2}\left(\mathbb{R}^{m};\mathbb{C}^{2^{m}}\right)=\mathbb{C}\left[\psi_{0,0}\right]\oplus\bigoplus_{\tau\in\mathbb{N}_{0}^{m}\setminus0}E_{\tau}.
\]
Furthermore, we have the isomorphism 
\begin{eqnarray*}
\mathscr{I}_{\tau}:\Lambda^{*}V_{\tau} & \rightarrow & E_{\tau}\\
\mathscr{I}_{\tau}\left(\bigwedge_{j\in I_{\tau}}w_{j}^{b_{j}}\right) & \coloneqq & \prod_{j\in I_{\tau}}\left(\frac{c\left(w_{j}\right)A_{j}}{\sqrt{2\tau_{j}h}}\right)^{b_{j}}\psi_{\tau,0}.
\end{eqnarray*}
Each $E_{\tau}$ hence has dimension $2^{Z_{\tau}}$ and is closed
under $c\left(w_{j}\right)A_{j}$, $c\left(\overline{w}_{j}\right)A_{j}^{*}$
for $1\leq j\leq m$. We again have 
\begin{eqnarray}
E_{\tau} & = & E_{\tau}^{\textrm{even}}\oplus E_{\tau}^{\textrm{odd}},\quad\textrm{ where}\label{eq: even odd decomposition eigenspace}\\
E_{\tau}^{\textrm{even/odd}} & \coloneqq & \mathscr{I}_{\tau}\left(\Lambda^{\textrm{even/odd}}V_{\tau}\right),\nonumber 
\end{eqnarray}
thus giving the Landau decomposition 
\begin{equation}
L^{2}\left(\mathbb{R}^{m};\mathbb{C}^{2^{m}}\right)=\mathbb{C}\left[\psi_{0,0}\right]\oplus\bigoplus_{\tau\in\mathbb{N}_{0}^{m}\setminus0}\left(E_{\tau}^{\textrm{even}}\oplus E_{\tau}^{\textrm{odd}}\right).\label{eq: Landau Levels}
\end{equation}

The Dirac operator $D_{\mathbb{R}^{m}}$ by virtue of \prettyref{eq: dbar},
\prettyref{eq: dbar*}, \prettyref{eq: dolbeault Dirac} preserves
and acts on $E_{\tau}$ via 
\[
c\left(\frac{w_{r_{\tau}}+\bar{w}_{r_{\tau}}}{\sqrt{2}}\right)=\left(w_{r_{\tau}}\wedge+\iota_{\bar{w}_{r_{\tau}}}\right),
\]
under the isomorphism $\mathscr{I}_{\tau}$, where $r_{\tau}\coloneqq\left(\sqrt{\tau_{1}\mu_{1}h},\ldots,\sqrt{\tau_{m}\mu_{m}h}\right)$
and $w_{r_{\tau}}$ is as in \prettyref{eq: w_r}. Hence, if we define
$\mathtt{i}_{\tau}\coloneqq\mathscr{I}_{\tau}\mathtt{i}_{r_{\tau}}\mathscr{I}_{\tau}^{-1}:E_{\tau}^{\textrm{even/odd}}\rightarrow E_{\tau}^{\textrm{odd/even}}$
we have that the restriction of $D_{\mathbb{R}^{m}}$ to $E_{\tau}$
is of the form 
\begin{eqnarray}
D_{\mathbb{R}^{m}} & = & \begin{bmatrix} & \left|r_{\tau}\right|\mathtt{i}_{\tau}\\
\left|r_{\tau}\right|\mathtt{i}_{\tau}
\end{bmatrix}\label{eq: Dirac operator 2 by 2 block}
\end{eqnarray}
via \prettyref{eq: clifford multiplication blocks}. Also note that
since $E_{\tau}^{\textrm{even/odd}}\subset\mathscr{I}_{\tau}\left(C^{\infty}\left(\mathbb{R}^{m}\right)\otimes\Lambda^{\textrm{even/odd}}V^{1,0}\right)$
respectively, one has 
\begin{eqnarray}
c\left(e_{0}\right)E_{\tau}^{\textrm{even/odd}} & = & \pm iE_{\tau}^{\textrm{even/odd}}\label{eq: clifford mult. is diagonal on 2 by 2 block}
\end{eqnarray}
using \prettyref{eq:odd clifford representation}. The eigenspaces
for $D_{\mathbb{R}^{m}}$ are now given by 
\begin{eqnarray}
E_{\tau}^{\pm} & = & \mathscr{I}_{\tau}\left(V_{\tau}^{\pm}\right),\label{eq: eigenspaces model Dirac}
\end{eqnarray}
via \prettyref{eq: eigenspaces clifford mult.} with eigenvalue $\pm\left|r_{\tau}\right|=\pm\sqrt{\mu.\tau h}$
respectively. We now summarize.
\begin{prop}
\label{prop:eigenspaces magnetic Dirac}An orthogonal decomposition
of $L^{2}\left(\mathbb{R}^{m};\mathbb{C}^{2^{m}}\right)$ consisting
of eigenspaces of the magnetic Dirac operator $D_{\mathbb{R}^{m}}$
\prettyref{eq: magnetic Dirac Rm} is given by
\[
L^{2}\left(\mathbb{R}^{m};\mathbb{C}^{2^{m}}\right)=\mathbb{C}\left[\psi_{0,0}\right]\oplus\bigoplus_{\tau\in\mathbb{N}_{0}^{m}\setminus0}\left(E_{\tau}^{+}\oplus E_{\tau}^{-}\right).
\]
Here $E_{\tau}^{\pm}$, as in \prettyref{eq: eigenspaces model Dirac},
have dimension $2^{Z_{\tau}-1}$ and correspond to the eigenvalues
$\pm\sqrt{\mu.\tau h}$ respectively.
\end{prop}

\subsection{The Semi-classical calculus}

Finally, here we review the semi-classical pseudodifferential calculus
used throughout the paper with \cite{GuilleminSternberg-Semiclassical,Zworski}
being the detailed references. Let $\mathfrak{gl}\left(l\right)$
denote the space of all $l\times l$ complex matrices. For $A=\left(a_{ij}\right)\in\mathfrak{gl}\left(l\right)$
we denote $\left|A\right|=\max_{ij}\left|a_{ij}\right|$. Denote by
$\mathcal{S}\left(\mathbb{R}^{n};\mathbb{C}^{l}\right)$ the space
of Schwartz maps $f:\mathbb{R}^{n}\rightarrow\mathbb{C}^{l}$. We
define the symbol space $S^{m}\left(\mathbb{R}^{2n};\mathbb{C}^{l}\right)$
as the space of maps $a:\left(0,1\right]_{h}\rightarrow C^{\infty}\left(\mathbb{R}_{x,\xi}^{2n};\mathfrak{gl}\left(l\right)\right)$
such that each of the semi-norms 
\[
\left\Vert a\right\Vert _{\alpha,\beta}:=\text{sup}_{\substack{x,\xi}
,h}\langle\xi\rangle^{-m+|\beta|}\left|\partial_{x}^{\alpha}\partial_{\xi}^{\beta}a(x,\xi;h)\right|
\]
is finite $\forall\alpha,\beta\in\mathbb{N}_{0}^{n}$. Such a symbol
is said to lie in the more refined class $a\in S_{\textrm{cl}}^{m}\left(\mathbb{R}^{2n};\mathbb{C}^{l}\right)$
if there exists an $h$-independent sequence $a_{k}$, $k=0,1,\ldots$
of symbols such that
\begin{equation}
a-\left(\sum_{k=0}^{N}h^{k}a_{k}\right)\in h^{N+1}S^{m}\left(\mathbb{R}^{2n};\mathbb{C}^{l}\right),\;\forall N.\label{eq: symbolic expansion}
\end{equation}

Symbols as above can be Weyl quantized to define one-parameter families
of operators $a^{W}:\mathcal{S}\left(\mathbb{R}^{n};\mathbb{C}^{l}\right)\rightarrow\mathcal{S}\left(\mathbb{R}^{n};\mathbb{C}^{l}\right)$
with Schwartz kernels given by 
\[
a^{W}\coloneqq\frac{1}{\left(2\pi h\right)^{n}}\int e^{i\left(x-y\right).\xi/h}a\left(\frac{x+y}{2},\xi;h\right)d\xi
\]
We denote by $\Psi_{\textrm{cl}}^{m}\left(\mathbb{R}^{n};\mathbb{C}^{l}\right)$
the class of operators thus obtained by quantizing $S_{\textrm{cl}}^{m}\left(\mathbb{R}^{2n};\mathbb{C}^{l}\right)$.
This class of operators is closed under the standard operations of
composition and formal-adjoint. Indeed, the Weyl symbols of the composition
and adjoint satisfy 
\begin{align}
a^{W}\circ b^{W} & =\left(a\ast b\right)^{W}\label{eq: Weyl product}\\
 & \coloneqq\left[e^{\frac{ih}{2}\left(\partial_{r_{1}}\partial_{s_{2}}-\partial_{r_{2}}\partial_{s_{1}}\right)}\left(a\left(s_{1},r_{1};h\right)b\left(s_{2},r_{2};h\right)\right)\right]_{x=s_{1}=s_{2},\xi=r_{1}=r_{2}}^{W}\nonumber \\
\left(a^{W}\right)^{*} & =\left(a^{*}\right)^{W}.\nonumber 
\end{align}

Furthermore the class is invariant under changes of coordinates and
basis for $\mathbb{C}^{l}$. This allows one to define an invariant
class of operators $\Psi_{\textrm{cl}}^{m}\left(X;E\right)$ on $C^{\infty}\left(X;E\right)$
associated to any complex vector bundle on a smooth compact manifold
$X$. These define uniformly in $h$ bounded operators between the
Sobolev spaces $H^{s}\left(X;E\right)\rightarrow H^{s-m}\left(X;E\right)$
with the $h$-dependent norm on each Sobolev space defined via 
\[
\left\Vert u\right\Vert _{H^{s}\left(X\right)}\coloneqq\left\Vert \left(1+h^{2}\nabla^{E*}\nabla^{E}\right)^{s/2}u\right\Vert _{L^{2}},\quad s\in\mathbb{R},
\]
with respect to any metric $g^{TX},h^{E}$ on $X,E$ and unitary connection
$\nabla^{E}$. 

For $A\in\Psi_{\textrm{cl}}^{m}\left(X;E\right)$, its principal symbol
is well-defined as an element in $\sigma\left(A\right)\in S^{m}\left(X;\textrm{End}\left(E\right)\right)\subset C^{\infty}\left(X;\textrm{End}\left(E\right)\right).$
One has that $\sigma\left(A\right)=0$ if and only if $A\in h\Psi_{\textrm{cl}}^{m}\left(X;E\right)$.
We remark that $\sigma\left(A\right)$ is the restriction of standard
symbol in \cite{Zworski} to the refined class $\Psi_{\textrm{cl}}^{m}\left(X;E\right)$
and is locally given by the first coefficient $a_{0}$ in the expansion
of its Weyl symbol. The principal symbol satisfies the basic relations
$\sigma\left(AB\right)=\sigma\left(A\right)\sigma\left(B\right)$,
$\sigma\left(A^{*}\right)=\sigma\left(A\right)^{*}$ with the formal
adjoints being defined with respect to the same Hermitian metric $h^{E}$.
The principal symbol map has an inverse given by the quantization
map $\textrm{Op}:S^{m}\left(X;\textrm{End}\left(E\right)\right)\rightarrow\Psi_{\textrm{cl}}^{m}\left(X;E\right)$
satisfying $\sigma\left(\textrm{Op}\left(a\right)\right)=a\in S^{m}\left(X;\textrm{End}\left(E\right)\right)$.
We often use the alternate notation $\textrm{Op}\left(a\right)=a^{W}$.
For a scalar function $b\in S^{m}\left(X\right)$, it is clear from
the multiplicative property of the symbol that $\left[a^{W},b^{W}\right]\in h\Psi_{\textrm{cl}}^{m}\left(X;E\right)$
and we define $H_{b}\left(a\right)\coloneqq\frac{i}{h}\sigma\left(\left[a^{W},b^{W}\right]\right)\in S^{m}\left(X;\textrm{End}\left(E\right)\right)$.
If $a$ is self adjoint and $b$ real, then it is easy to see that
$H_{b}\left(a\right)$ is self-adjoint. We then define $\left|H_{b}\left(a\right)\right|=\max_{\lambda\in\textrm{Spec }H_{b}\left(a\right)}\left|\lambda\right|$. 

The wavefront set of an operator $A\in\Psi_{\textrm{cl}}^{m}\left(X;E\right)$
can be defined invariantly as a subset $WF\left(A\right)\subset\overline{T^{*}X}$
of the fibrewise radial compatification of its cotangent bundle. If
the local Weyl symbol of $A$ is given by $a$ then $\left(x_{0},\xi_{0}\right)\notin WF\left(A\right)$
if and only if there exists an open neighborhood $\left(x_{0},\xi_{0};0\right)\in U\subset\overline{T^{*}X}\times\left(0,1\right]_{h}$
such that $a\in h^{\infty}\left\langle \xi\right\rangle ^{-\infty}C^{k}\left(U;\mathbb{C}^{l}\right)$
for all $k$. The wavefront set satisfies the basic properties $WF\left(A+B\right)\subset WF\left(A\right)\cap WF\left(B\right)$,
$WF\left(AB\right)\subset WF\left(A\right)\cap WF\left(B\right)$
and $WF\left(A^{*}\right)=WF\left(A\right)$. The wavefront set $WF\left(A\right)=\emptyset$
is empty if and only if $A\in h^{\infty}\Psi^{-\infty}\left(X;E\right)$.
We say that two operators $A=B$ microlocally on $U\subset\overline{T^{*}X}$
if $WF\left(A-B\right)\cap U=\emptyset$. We also define by $\Psi_{\textrm{cl}}^{c}\left(X;E\right)$
the class of pseudodifferential operators $A$ with wavefront set
$WF\left(A\right)\Subset T^{*}X$ compactly contained in the cotangent
bundle. It is clear that $\Psi_{\textrm{cl}}^{c}\left(X;E\right)\subset\Psi_{\textrm{cl}}^{-\infty}\left(X;E\right)$.

An operator $A\in\Psi_{\textrm{cl}}^{m}\left(X;E\right)$ is said
to be elliptic if $\left\langle \xi\right\rangle ^{m}\sigma\left(A\right)^{-1}$
exists and is uniformly bounded on $T^{*}X$. If $A\in\Psi_{\textrm{cl}}^{m}\left(X;E\right)$,
$m>0$, is formally self-adjoint such that $A+i$ is elliptic then
it is essentially self-adjoint (with domain $C_{c}^{\infty}\left(X;E\right)$)
as an unbounded operator on $L^{2}\left(X;E\right)$. Its resolvent
$\left(A-z\right)^{-1}\in\Psi_{\textrm{cl}}^{-m}\left(X;E\right)$,
$z\in\mathbb{C}$, $\textrm{Im}z\neq0$, now exists and is pseudodifferential
by an application of Beals's lemma. The resolvent furthermore has
an expansion $\left(A-z\right)^{-1}\sim\sum_{j=0}^{\infty}h^{j}\textrm{Op}\left(a_{j}^{z}\right)$
in $\Psi_{\textrm{cl}}^{-m}\left(X;E\right)$. Here each symbol appearing
in the expansion has the form
\begin{eqnarray*}
a_{j}^{z} & = & \left(\sigma\left(A\right)-z\right)^{-1}a_{j,1}^{z}\left(\sigma\left(A\right)-z\right)^{-1}\ldots\left(\sigma\left(A\right)-z\right)^{-1}a_{j,2j}^{z}\left(\sigma\left(A\right)-z\right)^{-1}\\
 &  & \qquad\qquad\qquad\qquad\qquad\qquad\qquad\qquad\qquad\qquad\in S^{-m}\left(X;\textrm{End}\left(E\right)\right),
\end{eqnarray*}
for polynomial in $z$ symbols $a_{j,k}^{z}$, $k=1,\ldots,2j$. Given
a Schwartz function $f\in\mathcal{S}\left(\mathbb{R}\right)$, the
Helffer-Sjostrand formula now expresses the function $f\left(A\right)$
of such an operator in terms of its resolvent and an almost analytic
continuation $\tilde{f}$ via
\[
f\left(A\right)=\frac{1}{\pi}\int_{\mathbb{C}}\bar{\partial}\tilde{f}\left(z\right)\left(A-z\right)^{-1}dzd\bar{z}.
\]
Plugging the resolvent expansion into the above formula then shows
that the above lies in and has an expansion $f\left(A\right)\sim\sum_{j=0}^{\infty}h^{j}A_{j}^{f}$
in $\Psi_{\textrm{cl}}^{-\infty}\left(X;E\right)$. Finally, one defines
the classical $\lambda$-energy level of $A$ via 
\[
\Sigma_{\lambda}^{A}=\left\{ \left(x,\xi\right)\in T^{*}X|\det\left(\sigma\left(A\right)\left(x,\xi\right)-\lambda I\right)=0\right\} .
\]
Now, the form for the coefficients of the resolvent expansion also
shows $WF\left(f\left(A\right)\right)\subset\Sigma_{\textrm{spt}\left(f\right)}^{A}\coloneqq\bigcup_{\lambda\in\textrm{spt}\left(f\right)}\Sigma_{\lambda}^{A}$.

\subsubsection{The class $\Psi_{\delta}^{m}\left(X;E\right)$}

In \prettyref{sec:First Reductions} we shall need the more exotic
class of symbols $S_{\delta}^{m}\left(\mathbb{R}^{2n};\mathbb{C}\right)$
defined for each $0<\delta<\frac{1}{2}$. A function $a:\left(0,1\right]_{h}\rightarrow C^{\infty}\left(\mathbb{R}_{x,\xi}^{2n};\mathbb{C}\right)$
is said to be in this class if and only if 
\begin{equation}
\left\Vert a\right\Vert _{\alpha,\beta}:=\text{sup}_{\substack{x,\xi}
,h}\langle\xi\rangle^{-m+|\beta|}h^{-\left(\left|\alpha\right|+\left|\beta\right|\right)\delta}\left|\partial_{x}^{\alpha}\partial_{\xi}^{\beta}a(x,\xi;h)\right|\label{eq: delta pseudodifferential estimates}
\end{equation}
is finite $\forall\alpha,\beta\in\mathbb{N}_{0}^{n}.$ This class
of operators is closed under the standard operations of composition,
adjoint and changes of coordinates allowing the definition of the
exotic pseudodifferential algebra $\Psi_{\delta}^{m}\left(X\right)$
on a compact manifold. The class $S_{\delta}^{m}\left(X\right)$ is
a family of functions $a:\left(0,1\right]_{h}\rightarrow C^{\infty}\left(T^{*}X;\mathbb{C}\right)$
satisfying the estimates \prettyref{eq: delta pseudodifferential estimates}
in every coordinate chart and induced trivialization. Such a family
can be quantized to $a^{W}\in\Psi_{\delta}^{m}\left(X\right)$ satisfying
$a^{W}b^{W}=\left(ab\right)^{W}+h^{1-2\delta}\Psi_{\delta}^{m+m'-1}\left(X\right)$
for another $b\in S_{\delta}^{m'}\left(X\right)$. The operators in
$\Psi_{\delta}^{0}\left(X\right)$ are uniformly bounded on $L^{2}\left(X\right)$.
Finally, the wavefront an operator $A\in\Psi_{\delta}^{m}\left(X;E\right)$
is similarly defined and satisfies the same basic properties as before.

\subsubsection{Fourier integral operators}

We shall also need the local theory of Fourier integral operators.
Let $\kappa:U\rightarrow V$ be an exact symplectomorphism between
two open subsets $U\subset T^{*}X$, $V\subset T^{*}Y$ inside cotangent
spaces of manifolds of same dimension $n$. Assume that there exist
local coordinates $\left(x_{1},\ldots,x_{n}\right),$$\left(y_{1},\ldots y_{n}\right)$
on $\pi\left(U\right),\pi\left(V\right)$ respectively with induced
canonical coordinates $\left(x,\xi\right),\left(y,\eta\right)$ on
$U,V$. A function $S\left(x,\eta\right)\in C^{\infty}\left(\Omega\right)$
on an open subset $\Omega\subset\mathbb{R}_{x,\eta}^{2n}$ is said
to be a generating function for the graph of $\kappa$ if the Lagrangian
submanifolds 
\begin{eqnarray*}
\left(T^{*}X\right)\times\left(T^{*}Y\right)^{-}\supset\Lambda_{\kappa} & \coloneqq & \left\{ \left(\left(x,\xi\right);\kappa\left(x,\xi\right)\right)|\left(x,\xi\right)\in U\right\} \\
 & = & \left\{ \left(x,\partial_{x}S;\partial_{\eta}S,\eta\right)|\left(x,\eta\right)\in\Omega\right\} 
\end{eqnarray*}
are equal. Such a generating function always exists locally near any
point on $\Lambda_{\kappa}$. Letting $a:\left(0,1\right]_{h}\rightarrow C_{c}^{\infty}\left(\Omega\times\pi\left(V\right);\mathbb{C}\right)$,
which admits an expansion $a\left(x,y,\eta;h\right)\sim\sum_{k=0}^{\infty}h^{k}a_{k}\left(x,y,\eta\right)$,
one may now define a Fourier integral operator associated to $\kappa$
via 
\begin{eqnarray*}
A:L^{2}\left(Y\right) & \rightarrow & L^{2}\left(X\right)\\
\left(Af\right)\left(x\right) & = & \frac{1}{\left(2\pi h\right)^{n}}\int_{\mathbb{R}^{2n}}e^{\frac{i}{h}\left(S\left(x,\eta\right)-y.\eta\right)}a\left(x,y,\eta;h\right)f\left(y\right)dyd\eta.
\end{eqnarray*}
The symbol of $\sigma\left(A\right)\in C_{c}^{\infty}\left(\Lambda_{\kappa};\mathbb{C}\right)$
is defined using the generating function via $\sigma\left(A\right)\left(x,\eta\right)=a_{0}\left(x,\partial_{x}S,\eta\right)$.
The adjoint $A^{*}$, is again a Fourier integral operator associated
to the symplectomorphism $\kappa^{-1}$. The wavefront set of $A$
maybe defined as a subset $WF\left(A\right)\subset\overline{T^{*}X}\times\overline{T^{*}Y}$.
A point $\left(x,\xi;y,\eta\right)\notin WF\left(A\right)$ if and
only if there exist pseudodifferential operators $B\in\Psi_{\textrm{cl}}^{m}\left(X\right),C\in\Psi_{\textrm{cl}}^{m'}\left(Y\right)$
with $\left(x,\xi;y,\eta\right)\in WF\left(B\right)\times WF\left(C\right)$
such that $\left\Vert BAC\right\Vert _{H^{s}\left(Y\right)\rightarrow H^{s'}\left(X\right)}=O\left(h^{\infty}\right)$
for each $s,s'\in\mathbb{R}$. It can be shown that the wavefront
set is in fact a compact subset $WF\left(A\right)\subset\Lambda_{\kappa}$.
Given a pseudodifferential operator $B\in\Psi_{\textrm{cl}}^{m}\left(X\right)$,
Egorov's theorem says that the composite is a pseudodifferential operator
$A^{*}BA\in\Psi_{\textrm{cl}}^{m}\left(Y\right)$. Moreover its principal
symbol is given via $\sigma\left(A^{*}BA\right)=\left(\kappa^{-1}\right)^{*}\left|\sigma\left(A\right)\right|^{2}\sigma\left(B\right)\in C_{c}^{\infty}\left(V\right)$,
where we have again used the identification of $V$ with $\Lambda_{\kappa}$
given by the generating function. Finally one has the wavefront relation
$WF\left(A^{*}BA\right)\subset WF\left(A\right)\cap WF\left(B\right)$
again using the identifications of $U,V$ and $\Lambda_{\kappa}$. 

An important special case arises when $\kappa=e^{tH_{f}}$ is the
time $t$ flow of a Hamiltonian $f\in S^{m}\left(T^{*}X\right)$.
The operator $e^{\frac{it}{h}f^{W}}$, defined as a unitary operator
via Stone's theorem, is now a Fourier integral operator associated
to $\kappa$. Egorov's theorem now gives that the conjugation $e^{\frac{it}{h}f^{W}}Ae^{-\frac{it}{h}f^{W}}\in\Psi_{\textrm{cl}}^{m'}\left(X\right)$
is pseudodifferential for each $A\in\Psi_{\textrm{cl}}^{m'}\left(X\right)$
with principal symbol $\sigma\left(e^{\frac{it}{h}f^{W}}Ae^{-\frac{it}{h}f^{W}}\right)=\left(e^{tH_{f}}\right)^{*}\sigma\left(A\right)$.

\section{\label{sec:First Reductions}First reductions}

The trace expansion \prettyref{thm:main trace expansion} will be
proved in two steps based on the following two lemmas. Below, $\tau,T,T',f,\theta$
and $D$ are the same as before. 
\begin{lem}
\label{lem: O(h infty) LEMMA}Let $\epsilon\in\left(0,\frac{1}{2}\right)$
and $\vartheta\in C_{c}^{\infty}\left(\left(T'h^{\epsilon},T\right);\left[-1,1\right]\right)$.
Then 
\begin{eqnarray*}
\textrm{tr}\left[f\left(\frac{D}{\sqrt{h}}\right)\left(\mathcal{F}_{h}^{-1}\vartheta\right)\left(\lambda\sqrt{h}-D\right)\right] & =\\
\textrm{tr}\left[f\left(\frac{D}{\sqrt{h}}\right)\frac{1}{h}\check{\vartheta}\left(\frac{\lambda\sqrt{h}-D}{h}\right)\right] & = & O\left(h^{\infty}\right)
\end{eqnarray*}
 for all$\lambda\in\mathbb{R}$.
\end{lem}
We note that in the above lemma the function $\vartheta$ is allowed
to depend on $h$, while its support and range are contained in $h$-independent
intervals.
\begin{lem}
\label{lem: Easy trace expansion lemma}There exist smooth functions
$u_{j}\in C^{\infty}\left(\mathbb{R}\right)$ such that for each $\lambda\in\mathbb{R}$
and$\epsilon\in\left(0,\frac{1}{2}\right)$ one has a trace expansion
\begin{eqnarray*}
\textrm{tr}\left[f\left(\frac{D}{\sqrt{h}}\right)\left(\mathcal{F}_{h}^{-1}\theta_{\epsilon}\right)\left(\lambda\sqrt{h}-D\right)\right] & =\\
\textrm{tr}\left[f\left(\frac{D}{\sqrt{h}}\right)\frac{1}{h^{1-\epsilon}}\check{\theta}\left(\frac{\lambda\sqrt{h}-D}{h^{1-\epsilon}}\right)\right] & = & h^{-m-1}\left(\sum_{j=0}^{N-1}c_{j}h^{j/2}+O\left(h^{N/2}\right)\right)
\end{eqnarray*}
where $\theta_{\epsilon}\left(x\right)\coloneqq\theta\left(\frac{x}{h^{\epsilon}}\right)$.
\end{lem}
We note that the trace expansion \prettyref{thm:main trace expansion}
follows from the above two lemmas on simply splitting $\theta\left(x\right)=\theta_{\epsilon}\left(x\right)+\underbrace{\left[\theta\left(x\right)-\theta_{\epsilon}\left(x\right)\right]}_{\vartheta\left(x\right)}$
and applying \prettyref{lem: Easy trace expansion lemma} and \prettyref{lem: O(h infty) LEMMA}
to the first and second summands respectively. Lemma \prettyref{lem: Easy trace expansion lemma}
is a relatively classical expansion proved via local index theory
and will be deferred to \prettyref{sec:Local trace expansion}. Our
main occupation until then is in proving \prettyref{lem: O(h infty) LEMMA}.

As a first step one chooses a microlocal partition of unity $A_{\alpha}\in\Psi_{\textrm{cl}}^{0}\left(X\right)$,
$0\leq\alpha\leq N,$ satisfying
\begin{eqnarray}
\sum_{\alpha=0}^{N}A_{\alpha} & = & 1\nonumber \\
WF\left(A_{0}\right)\subset & U_{0}\subset & \overline{T^{*}X}\setminus\Sigma_{\left(-\tau,\tau\right)}^{D}\nonumber \\
WF\left(A_{\alpha}\right)\Subset & U_{\alpha}\subset & \Sigma_{\left(-2\tau,2\tau\right)}^{D},\;1\leq\alpha\leq N.\label{eq: microlocal partition of unity}
\end{eqnarray}
 subordinate to an open cover $\left\{ U_{\alpha}\right\} _{\alpha=0}^{N}$
of $T^{*}X$. Clearly, it suffices to prove 
\begin{equation}
\textrm{tr}\left[A_{\alpha}f\left(\frac{D}{\sqrt{h}}\right)\check{\vartheta}\left(\frac{\lambda\sqrt{h}-D}{h}\right)A_{\beta}\right]=O\left(h^{\infty}\right)\label{eq: microlocal estimate}
\end{equation}
for $1\leq\alpha,\beta\leq N$ with $WF\left(A_{\alpha}\right)\cap WF\left(A_{\beta}\right)\neq\emptyset$
. 

By the Helffer-Sjostrand formula we have the trace above is given
by 
\begin{equation}
\mathcal{T}_{\alpha\beta}^{\vartheta}\left(D\right)\coloneqq\frac{1}{\pi}\int_{\mathbb{C}}\bar{\partial}\tilde{f}\left(z\right)\check{\vartheta}\left(\frac{\lambda-z}{\sqrt{h}}\right)\textrm{tr }\left[A_{\alpha}\left(\frac{1}{\sqrt{h}}D-z\right)^{-1}A_{\beta}\right]dzd\bar{z}.\label{eq:Helffer Sjostrand formula}
\end{equation}
We note that the resolvent, the above trace as well as the left hand
side of \prettyref{eq: microlocal estimate} are well defined for
any essentially self-adjoint pseudodifferential operator in place
of $D$. The next reduction step attempts to modify $D$ without affecting
the asymptotics of $\mathcal{T}_{\alpha\beta}^{\vartheta}\left(D\right)$.
To this end, choose open subsets 
\begin{equation}
WF\left(A_{\alpha}\right)\cup WF\left(A_{\beta}\right)\subset V_{\alpha\beta}\Subset T^{*}X,\label{eq: def open neighnourhood V_ij}
\end{equation}
for each such pair $\alpha,\beta$ with $WF\left(A_{\alpha}\right)\cap WF\left(A_{\beta}\right)\neq\emptyset$.
With $d=\sigma\left(D\right)\in C^{\infty}\left(X;i\mathfrak{u}\left(S\right)\right)$,
define the required exit time
\begin{align}
T_{\alpha\beta} & \coloneqq\frac{1}{\inf_{g\in\mathcal{G}_{\alpha\beta}}\left|H_{g}d\right|},\quad\textrm{ where}\label{eq: necessary exit time}\\
\mathcal{G}_{\alpha\beta} & \coloneqq\left\{ g\in C^{\infty}\left(T^{*}X;\left[0,1\right]\right)|\left.g\right|_{WF\left(A_{\alpha}\right)\cap WF\left(A_{\beta}\right)}=1,\;\left.g\right|_{V_{\alpha\beta}^{c}}=0\right\} .\nonumber 
\end{align}
If one were to use a scalar symbol $d\in C^{\infty}\left(X\right)$
instead in \prettyref{eq: necessary exit time}, the required exit
time $T_{\alpha\beta}$ would have the following significance: any
Hamiltonian trajectory $\gamma\left(t\right)=e^{tH_{d}}$ with $\gamma\left(0\right)\in WF\left(A_{\alpha}\right)\cap WF\left(A_{\beta}\right),$$\gamma\left(T\right)\in V_{\alpha\beta}$,
would have length $T\geq T_{\alpha\beta}$ atleast the required exit
time. We now have the following. 
\begin{lem}
\label{lem:changing symbol near crit set} Let $D'\in\Psi_{\textrm{cl}}^{1}\left(X;E\right)$
be essentially self-adjoint such that $D=D'$ microlocally on $V_{\alpha\beta}$
. Then for $\vartheta\in C_{c}^{\infty}\left(\left(T_{\alpha\beta}'h^{\epsilon},T_{\alpha\beta}\right);\left[0,1\right]\right)$,
$0<T_{\alpha\beta}'<T_{\alpha\beta}$, one has 
\[
\mathcal{T}_{\alpha\beta}^{\vartheta}\left(D\right)=\mathcal{T}_{\alpha\beta}^{\vartheta}\left(D'\right)\quad\textrm{mod }h^{\infty}.
\]
\end{lem}
\begin{proof}
Let $B\in\Psi_{\textrm{cl}}^{0}\left(X\right)$ be a microlocal cutoff
such that $B=0$ on $WF\left(D-D'\right)$ and $B=1$ on $V_{\alpha\beta}$.
Then $\left(1-B\right)A_{\beta}=0$ microlocally implies 
\begin{align}
\left(z-\frac{1}{\sqrt{h}}D\right)B\left(z-\frac{1}{\sqrt{h}}D'\right)^{-1}A_{\beta} & =A_{\beta}-\left[\frac{1}{\sqrt{h}}D,B\right]\left(z-D'\right)^{-1}A_{\beta}\nonumber \\
 & +B\left(\frac{1}{\sqrt{h}}D'-\frac{1}{\sqrt{h}}D\right)\left(z-\frac{1}{\sqrt{h}}D'\right)^{-1}A_{\beta}\label{eq:first estimate}\\
 & \qquad\qquad\qquad\qquad\qquad\qquad\qquad\left(\textrm{mod}\:h^{\infty}\right)\nonumber 
\end{align}
in trace norm. Next, multiplying through by $A_{\alpha}\left(z-\frac{1}{\sqrt{h}}D\right)^{-1}$and
using $A_{\alpha}B=B$ microlocally gives
\begin{multline}
A_{\alpha}\left(z-\frac{1}{\sqrt{h}}D'\right)^{-1}A_{\beta}-A_{\alpha}\left(z-\frac{1}{\sqrt{h}}D\right)^{-1}A_{\beta}=\\
A_{\alpha}\left(z-\frac{1}{\sqrt{h}}D\right)^{-1}B\left(\frac{1}{\sqrt{h}}D'-\frac{1}{\sqrt{h}}D\right)\left(z-\frac{1}{\sqrt{h}}D'\right)^{-1}A_{\beta}\\
-A_{\alpha}\left(z-\frac{1}{\sqrt{h}}D\right)^{-1}\left[\frac{1}{\sqrt{h}}D,B\right]\left(z-\frac{1}{\sqrt{h}}D'\right)^{-1}A_{\beta}+O\left(\left|\textrm{Im}z\right|^{-1}h^{\infty}\right)\label{eq:second term difference resolvents}
\end{multline}
in trace norm. Now $B=0$ on $WF\left(D-D'\right)$ gives that the
first term on the right hand side above is $O\left(\left|\textrm{Im}z\right|^{-2}h^{\infty}\right)$. 

We now estimate the second term. Let $S_{\alpha\beta}<S_{\alpha\beta}''<S_{\alpha\beta}'''<T_{\alpha\beta}$
and $S_{\alpha\beta}'>T_{\alpha\beta}'$ be such that $\vartheta\in C_{c}^{\infty}\left(\left[S_{\alpha\beta}'h^{\epsilon},S_{\alpha\beta}\right];\left[0,1\right]\right)$.
Let $g_{0}\in\mathcal{G}_{\alpha\beta}$ with $\left|H_{g_{0}}\left(d\right)\right|\leq\frac{1}{S_{\alpha\beta}'''}$.
Set $g=\alpha_{z}g_{0}$, where
\[
\alpha_{z}=\min\left(\frac{S_{\alpha\beta}''\mbox{Im}z}{\sqrt{h}\log\frac{1}{h}},N\right)
\]
with the constant $N>0$ to be specified later. We note that 
\begin{eqnarray*}
G & = & \left(e^{g\log\frac{1}{h}}\right)^{W}\in h^{-N}\Psi_{\delta}^{0}\left(X\right)
\end{eqnarray*}
for each $0<\delta<\frac{1}{2}$. Since it has an elliptic symbol
we may construct its inverse by symbolic calculus $G^{-1}\in h^{N}\Psi_{\delta}^{0}\left(X\right)$.
Moreover 
\begin{align}
G\left(z-\frac{1}{\sqrt{h}}D_{h}\right)G^{-1} & =\left(z-\frac{1}{\sqrt{h}}D_{h}\right)+i\left(\alpha_{z}\sqrt{h}\log\frac{1}{h}\right)\left(H_{g_{0}}\left(d\right)\right)^{W}\label{eq: resolvent in exotci class}\\
 & \qquad\qquad\qquad\qquad\qquad\qquad+R^{W},\quad\textrm{ with}\nonumber \\
R & =O\left(h^{\frac{3}{2}}\alpha_{z}\log\frac{1}{h}\right)\quad\textrm{ in }S_{\delta}^{0}\left(X\right).
\end{align}
Now, since $\left|\left(\alpha_{z}\sqrt{h}\log\frac{1}{h}\right)H_{g_{0}}\left(d\right)\right|\leq\frac{S_{\alpha\beta}''}{S_{\alpha\beta}'''}\left|\mbox{Im}z\right|<\left|\mbox{Im}z\right|$,
the inverse $G\left(z-\frac{1}{\sqrt{h}}D_{h}\right)^{-1}G^{-1}$
of the above exists and is $O\left(\left|\textrm{Im}z\right|^{-1}\right)$
in operator norm for $\textrm{Im}z\neq0$, and $h$ sufficiently small. 

Next, $G=e^{\alpha_{z}\log\frac{1}{h}}$ on $WF\left(A_{\alpha}\right)$,$G=G^{-1}=I$
on $WF\left(B\right)\setminus V_{\alpha\beta}$ and $\left[D_{h},B\right]=0$
on $V_{\alpha\beta}$ imply 
\begin{multline*}
e^{\alpha_{z}\log\frac{1}{h}}A_{\alpha}\left(z-\frac{1}{\sqrt{h}}D_{h}\right)^{-1}\left[\frac{1}{\sqrt{h}}D_{h},B\right]\\
=A_{\alpha}G\left(z-\frac{1}{\sqrt{h}}D_{h}\right)^{-1}G^{-1}\left[\frac{1}{\sqrt{h}}D_{h},B\right]+O\left(\left|\textrm{Im}z\right|^{-1}h^{\infty}\right)
\end{multline*}
in trace norm. The above is now $O\left(\left|\textrm{Im}z\right|^{-1}h^{-n}\right)$
in trace norm. Hence 
\[
A_{\alpha}\left(z-\frac{1}{\sqrt{h}}D_{h}\right)^{-1}\left[\frac{1}{\sqrt{h}}D_{h},B\right]=O\left(\left|\textrm{Im}z\right|^{-1}h^{-n}\max\left(h^{N},e^{-\frac{S_{\alpha\beta}''\textrm{Im}z}{\sqrt{h}}}\right)\right)
\]
in trace norm. This now estimates the second term of \prettyref{eq:second term difference resolvents}
and gives 
\begin{align}
 & A_{\alpha}\left(z-\frac{1}{\sqrt{h}}D_{h}'\right)^{-1}A_{\beta}-A_{\alpha}\left(z-\frac{1}{\sqrt{h}}D_{h}\right)^{-1}A_{\beta}\label{eq:resolvent difference bound}\\
= & O\left(\left|\textrm{Im}z\right|^{-2}h^{-n}\max\left(h^{N},e^{-\frac{S_{\alpha\beta}''\textrm{Im}z}{\sqrt{h}}}\right)\right)\nonumber 
\end{align}
in trace norm.

Next, we have the Paley-Wiener estimate 
\begin{equation}
\check{\vartheta}\left(\frac{\lambda-z}{\sqrt{h}}\right)=\begin{cases}
O\left(e^{\frac{S_{\alpha\beta}\left(\textrm{Im}z\right)}{\sqrt{h}}}\right); & \textrm{Im}z>0\\
O\left(e^{\frac{S_{\alpha\beta}'\left(\textrm{Im}z\right)}{h^{\frac{1}{2}-\epsilon}}}\right); & \textrm{Im}z<0.
\end{cases}\label{eq: Paley Wiener estimate}
\end{equation}
Introduce $\psi\in C^{\infty}\left(\mathbb{R};\left[0,1\right]\right)$
such that $\psi\left(x\right)=\begin{cases}
1; & x\leq1\\
0; & x\geq2
\end{cases}$. Setting $\psi_{M}\left(z\right)=\psi\left(\frac{\textrm{Im}z}{M\sqrt{h}\log\frac{1}{h}}\right)$,
for another constant $M>1$ yet to be chosen, we have the estimate
\begin{equation}
\bar{\partial}\left(\psi_{M}\tilde{f}\right)=\begin{cases}
O\left(\psi_{M}\left|\textrm{Im}z\right|^{N}+\frac{1}{M\sqrt{h}\log\frac{1}{h}}1_{\left[1,2\right]}\left(\frac{\textrm{Im}z}{M\sqrt{h}\log\frac{1}{h}}\right)\right); & \textrm{Im}z>0\\
O\left(\left|\textrm{Im}z\right|^{N}\right); & \textrm{Im}z<0.
\end{cases}\label{eq: almost analytic estimate}
\end{equation}
Finally, \prettyref{eq:resolvent difference bound}, \prettyref{eq: Paley Wiener estimate}
and \prettyref{eq: almost analytic estimate} along with the observation
$\psi_{M}\left|\textrm{Im}z\right|^{N}=O\left(\left(M\sqrt{h}\log\frac{1}{h}\right)^{N}\right)$
gives
\begin{eqnarray*}
 &  & \mathcal{T}_{\alpha\beta}^{\vartheta}\left(D'\right)-\mathcal{T}_{\alpha\beta}^{\vartheta}\left(D\right)\\
 & = & \frac{1}{\pi}\int_{\mathbb{C}}\bar{\partial}\left(\psi_{M}\tilde{f}\right)\check{\vartheta}\left(\frac{\lambda-z}{\sqrt{h}}\right)\left[A_{\alpha}\left(z-\frac{1}{\sqrt{h}}D_{h}'\right)^{-1}A_{\beta}\right.\\
 &  & \qquad\qquad\qquad\qquad\qquad\qquad\qquad\qquad-\left.A_{\alpha}\left(z-\frac{1}{\sqrt{h}}D_{h}\right)^{-1}A_{\beta}\right]dzd\bar{z}\\
 & = & O\left(h^{\infty}\right)+\\
 &  & O\left[\int_{\left\{ M\sqrt{h}\log\frac{1}{h}\leq\mbox{Im}z\leq2M\sqrt{h}\log\frac{1}{h}\right\} }\frac{h^{-n}}{\sqrt{h}\log\frac{1}{h}}\max\left(h^{N}e^{\frac{S_{\alpha\beta}\left(\textrm{Im}z\right)}{\sqrt{h}}},e^{-\frac{\left(S_{\alpha\beta}''-S_{\alpha\beta}\right)\textrm{Im}z}{\sqrt{h}}}\right)\right]\\
 & = & O\left[\max\left(h^{N-2MS_{\alpha\beta}-n},h^{M\left(S_{\alpha\beta}''-S_{\alpha\beta}\right)-n}\right)\right].
\end{eqnarray*}
Choosing $M\gg\frac{n}{\left(S_{\alpha\beta}''-S_{\alpha\beta}\right)}$
and furthermore $N\gg2MS_{\alpha\beta}+n$ gives the result.
\end{proof}
In the proof above we have closely followed \cite{Dimassi-Sjostrand}
Lemma 12.7. Again, the proof above avoids the use of an unknown parametrix
for $e^{\frac{it}{h}D}$ which, following the significance of the
required exit time $T_{\alpha\beta}$ noted before, maybe used to
give an alternate proof in the case when $d$ is scalar.

\section{Reduction to $\mathbb{R}^{n}$\label{sec:Reduction to R^n} }

In this section we shall further reduce to the case of a Dirac operator
on $\mathbb{R}^{n}$ . First we cover $X$ by a finite set of Darboux
charts $\left\{ \varphi_{s}:\Omega_{s}\rightarrow\Omega_{s}^{0}\subset\mathbb{R}^{n}\right\} _{s\in S}$
for the contact form $a$, centered at points $\left\{ x_{s}\right\} _{s\in S}\in X$.
By shrinking the partition of unity \prettyref{eq: microlocal partition of unity}
we may assume that for each pair $\alpha,\beta$, with $WF\left(A_{\alpha}\right)\cap WF\left(A_{\beta}\right)\neq\emptyset$,
the open sets $V_{\alpha\beta}\subset T^{*}\Omega_{s}$ in \prettyref{eq: def open neighnourhood V_ij}
are contained in some Darboux chart. Now consider such a chart $\Omega_{s}$
with coordinates$\left(x_{0},\ldots,x_{2m}\right)$ centered at $x_{s}\in X$
and an orthonormal frame $\left\{ e_{j}=w_{j}^{k}\partial_{x_{k}}\right\} ,0\leq j\leq2m$
for the tangent bundle on $\Omega_{s}$. We hence have 
\begin{equation}
w_{j}^{k}g_{kl}w_{r}^{l}=\delta_{jr},\label{eq: diagonalizing metric}
\end{equation}
 where $g_{kl}$ is the metric in these coordinates and the Einstein
summation convention is being used. Let $\Gamma_{jk}^{l}$ be the
Christoffel symbols for the Levi-Civita connection in the orthonormal
frame $e_{i}$ satisfying $\nabla_{e_{j}}e_{k}=\Gamma_{jk}^{l}e_{l}$.
This orthonormal frame induces an orthonormal frame $u_{j}$, $1\leq j\leq2^{m}$,
for the spin bundle $S$. We further choose a local orthonormal section
$\mathtt{l}\left(x\right)$ for the Hermitian line bundle $L$ and
define via $\nabla_{e_{j}}^{A_{0}}\mathtt{l}=\Upsilon_{j}\left(x\right)\mathtt{l}$,
$0\leq j\leq2m$ the Christoffel symbols of the unitary connection
$A_{0}$ on $L$. In terms of the induced frame $u_{j}\otimes\mathtt{l}$,
$1\leq j\leq2^{m}$, for $S\otimes L$ the Dirac operator \prettyref{eq:Semiclassical Magnetic Dirac}
has the form (cf. \cite{Berline-Getzler-Vergne} Section 3.3) 
\begin{eqnarray}
D & = & \gamma^{j}w_{j}^{k}P_{k}+h\left(\frac{1}{4}\Gamma_{jk}^{l}\gamma^{j}\gamma^{k}\gamma_{l}+\Upsilon_{j}\gamma^{j}\right),\quad\mbox{where}\label{eq: Dirac in coords}\\
P_{k} & = & h\partial_{x_{k}}+ia_{k},\label{eq: cov diff in cords}
\end{eqnarray}
and 
\begin{equation}
a\left(x\right)=a_{k}dx^{k}=dx_{0}+\sum_{j=1}^{m}\left(x_{j}dx_{j+m}-x_{j+m}dx_{j}\right)\label{eq:standard contact form}
\end{equation}
is the standard contact one form in these coordinates. 

The expression in \prettyref{eq: Dirac in coords} is formally self-adjoint
with respect to the Riemannian density $e^{1}\wedge\ldots\wedge e^{n}=\sqrt{g}dx\coloneqq\sqrt{g}dx^{1}\wedge\ldots\wedge dx^{n}$
with $g=\det\left(g_{ij}\right)$. To get an operator self-adjoint
with respect to the Euclidean density $dx$ one expresses the Dirac
operator in the framing $g^{\frac{1}{4}}u_{j}\otimes\mathtt{l},1\leq j\leq2^{m}$.
In this new frame the expression \prettyref{eq: Dirac in coords}
for the Dirac operator needs to be conjugated by $g^{\frac{1}{4}}$
and hence the term $h\gamma^{j}w_{j}^{k}g^{-\frac{1}{4}}\left(\partial_{x_{k}}g^{\frac{1}{4}}\right)$
added. Hence, the Dirac operator in the new frame has the form 
\[
D=\left[\sigma^{j}w_{j}^{k}\left(\xi_{k}+a_{k}\right)\right]^{W}+hE\in\Psi_{\textrm{cl}}^{1}\left(\Omega_{s}^{0};\mathbb{C}^{2^{m}}\right),
\]
with $\sigma^{j}=i\gamma^{j}$, for some self-adjoint endomorphism
$E\left(x\right)\in C^{\infty}\left(\Omega_{s}^{0};i\mathfrak{u}\left(\mathbb{C}^{2^{m}}\right)\right)$.

The one form $a$ is extended to all of $\mathbb{R}^{n}$ by the same
formula \prettyref{eq:standard contact form}. The functions $w_{j}^{k}$
are extended such that
\[
\left.\left(w_{j}^{k}\partial_{x_{k}}\otimes dx^{j}\right)\right|_{\left(K_{s}^{0}\right)^{c}}=\partial_{x_{0}}\otimes dx^{0}+\sum_{j=1}^{m}\mu_{j}^{\frac{1}{2}}\left(\partial_{x_{j}}\otimes dx^{j}+\partial_{x_{j+m}}\otimes dx^{j+m}\right)
\]
 (and hence $\left.g\right|_{\left(K_{s}^{0}\right)^{c}}=dx_{0}^{2}+\sum_{j=1}^{m}\mu_{j}\left(dx_{j}^{2}+dx_{j+m}^{2}\right)$)
outside a compact neighborhood $\Omega_{s}^{0}\Subset K_{s}^{0}$.
These extensions may further be chosen such that the suitability assumption
\prettyref{def: Diagonalizability assumption} holds globally on $\mathbb{R}^{n}$
and for an extended positive function $\nu\in C_{c}^{\infty}\left(\mathbb{R}^{n}\right)$
satisfying 
\begin{equation}
\nu_{0}\leq\mu_{1}\left(\inf_{\mathbb{R}^{n}}\nu\right).\label{eq: infimum nu}
\end{equation}
The endomorphism $E\left(x\right)\in C_{c}^{\infty}\left(\mathbb{R}^{n};i\mathfrak{u}\left(\mathbb{C}^{2^{m}}\right)\right)$
is extended to an arbitrary self-adjoint endomorphism of compact support.
This now gives 
\begin{equation}
D_{0}=\left[\sigma^{j}w_{j}^{k}\left(\xi_{k}+a_{k}\right)\right]^{W}+hE\in\Psi_{\textrm{cl}}^{1}\left(\mathbb{R}^{n};\mathbb{C}^{2^{m}}\right)\label{eq:Dirac op weyl quantization}
\end{equation}
as a well defined formally self adjoint operator on $\mathbb{R}^{n}$.
Furthermore, the symbol of $D_{0}+i$ is elliptic in the class $S^{0}\left(m\right)$
for the order function $m=\sqrt{1+\sum_{k=0}^{m}\left(\xi_{k}+a_{k}\right)^{2}}$
and hence $D_{0}$ is essentially self adjoint (see \cite{Dimassi-Sjostrand}
Ch. 8). Below $\vartheta\in C_{c}^{\infty}\left(\left(T_{\alpha\beta}'h^{\epsilon},T_{\alpha\beta}\right);\left[0,1\right]\right)$,
$0<T_{\alpha\beta}'<T_{\alpha\beta}$, as before and we set $V_{\alpha\beta}^{0}\coloneqq\left(d\varphi_{s}\right)^{*}V_{\alpha\beta}\subset T^{*}\Omega_{s}^{0}$.
\begin{prop}
\label{prop:Reduction to R^n}There exist $A_{\alpha}^{0},A_{\beta}^{0}\in\Psi_{\textrm{cl}}^{0}\left(\mathbb{R}^{n}\right)$,
with $WF\left(A_{\alpha}^{0}\right)\cup WF\left(A_{\beta}^{0}\right)\Subset V_{\alpha\beta}^{0}\subset T^{*}\tilde{\Omega}_{s}$,
such that 
\[
\mathcal{T}_{\alpha\beta}^{\vartheta}\left(D\right)=\underbrace{\textrm{tr}\left[A_{\alpha}^{0}f\left(\frac{D_{0}}{\sqrt{h}}\right)\check{\vartheta}\left(\frac{\lambda\sqrt{h}-D_{0}}{h}\right)A_{\beta}^{0}\right]}_{\coloneqq\mathcal{T}_{\alpha\beta}^{\vartheta}\left(D_{0}\right)}\quad\textrm{mod }h^{\infty}.
\]
\end{prop}
\begin{proof}
Let $K_{\alpha\beta}',K_{\alpha\beta}''$ and $V_{\alpha\beta}',V_{\alpha\beta}''$
be compact and open subsets respectively satisfying $V_{\alpha\beta}\subset K_{\alpha\beta}'\subset V_{\alpha\beta}'\subset K_{\alpha\beta}''\subset V_{\alpha\beta}''\subset T^{*}\Omega_{s}$.
Choose $D'\in\Psi_{\mbox{cl}}^{0}\left(X;S\right)$ self-adjoint such
that $D=D'$ microlocally on $K_{\alpha\beta}'$ and
\begin{equation}
\Sigma_{\left(-\infty,2\tau\right]}^{D'}\subset V_{\alpha\beta}'\label{eq:energy levels A0 bounded}
\end{equation}
and set $E=D'-3\tau\in\Psi_{\mbox{cl}}^{0}\left(X;S\right)$. Pick
a cutoff function $\chi\left(x;y,\eta\right)\in C_{c}^{\infty}\left(\pi\left(V_{\alpha\beta}''\right)\times\left(d\varphi_{s}\right)^{*}V_{\alpha\beta}'';\left[0,1\right]\right)$
such that $\chi=1$ on $\pi\left(K_{\alpha\beta}''\right)\times\left(d\varphi_{s}\right)^{*}K_{\alpha\beta}''$.
Now define the operator 
\begin{eqnarray*}
U:L^{2}\left(\mathbb{R}^{n};\mathbb{C}^{2^{m}}\right) & \rightarrow & L^{2}\left(X;S\right),\\
\left(Uf\right)\left(x\right) & = & \frac{1}{\left(2\pi h\right)^{n}}\int e^{\frac{i}{h}\left(\varphi_{s}\left(x\right)-y\right).\eta}\chi\left(x;y,\eta\right)f\left(y\right)dyd\eta,\quad x\in X.
\end{eqnarray*}
The above is a semi-classical Fourier integral operator associated
to symplectomorphism $\kappa=\left(d\varphi_{s}^{-1}\right)^{*}$
given by the canonical coordinates. Its adjoint $U^{*}:L^{2}\left(X;S\right)\rightarrow L^{2}\left(\mathbb{R}^{n};\mathbb{C}^{2^{m}}\right)$
is again a semi-classical Fourier integral operator associated to
the symplectomorphism $\kappa^{-1}=\left(d\varphi_{s}\right)^{*}$.
A simple computation gives the following compositions are pseudodifferential
with 
\begin{eqnarray}
UU^{*} & = & I\quad\mbox{microlocally on }K_{\alpha\beta}''\quad\mbox{and}\label{eq:UU*=00003DI micro.}\\
U^{*}U & = & I\quad\mbox{microlocally on }\kappa\left(K_{\alpha\beta}''\right).\label{eq:U*U=00003DI micro.}
\end{eqnarray}
The composition 
\[
E'=E_{0}:=U^{*}EU\in\Psi_{\mbox{cl}}^{0}\left(\mathbb{R}^{n};\mathbb{C}^{2^{m}}\right)
\]
is now a pseudodifferential operator by Egorov's theorem with symbol
\begin{equation}
\sigma\left(E_{0}\right)=\left(d\varphi_{s}\right)^{*}\chi^{2}.\sigma\left(E\right).\label{eq:symbol B0}
\end{equation}
Similarly, $E_{0}':=UE_{0}U^{*}\in\Psi_{\mbox{cl}}^{0}\left(X;S\right)$
and
\begin{equation}
\sigma\left(E_{0}'\right)=\left(d\varphi_{s}\right)^{*}\chi^{4}.\sigma\left(E_{0}\right).\label{eq:symbol A0'}
\end{equation}
From \prettyref{eq:energy levels A0 bounded}, \prettyref{eq:symbol B0}
and \prettyref{eq:symbol A0'} we have $\Sigma_{\left(-\infty,-\tau\right]}^{E_{0}}\subset\kappa\left(V_{\alpha\beta}'\right)$
and $\Sigma_{\left(-\infty,-\tau\right]}^{E_{0}'}\subset V_{\alpha\beta}'$.
Hence by proposition \prettyref{prop:Dicrete spectrum criterion}
$E,E',E_{0}$ and $E_{0}'$ all have discrete spectrum in $\left(-\infty,-\tau\right]$.
We now select $g\in C_{c}^{\infty}\left(-5\tau,-\tau\right)$ such
that $g=1$ on $\left[-4\tau,-2\tau\right]$. We have 
\[
WF\left(g\left(E\right)\right)\subset\Sigma_{\textrm{spt}\left(g\right)}^{E}\subset\Sigma_{\left(-\infty,-\tau\right]}^{E}\subset V_{\alpha\beta}'.
\]
Combined with \prettyref{eq:U*U=00003DI micro.} this gives $\left(U^{*}U-I\right)g\left(E\right)\in h^{\infty}\Psi_{\textrm{cl}}^{-\infty}\left(X;S\right)$
and hence $\left\Vert \left(U^{*}U-I\right)g\left(E\right)\right\Vert =O\left(h^{\infty}\right)$
as an operator on $L^{2}\left(X;S\right)$. This in turn now gives
\begin{equation}
\left\Vert \left(U^{*}U-I\right)\Pi^{E}\right\Vert \left(\left\Vert E\right\Vert \left\Vert U\right\Vert +1\right)=O\left(h^{\infty}\right)\label{eq:hyp 1a}
\end{equation}
with $\Pi^{E}=\Pi_{\left[-4\tau,-2\tau\right]}^{E}$. Similarly, we
get 
\begin{equation}
\left\Vert \left(UU^{*}-I\right)\Pi^{E_{0}}\right\Vert \left(\left\Vert E_{0}\right\Vert \left\Vert U^{*}\right\Vert +1\right)=O\left(h^{\infty}\right).\label{eq:hyp 1b}
\end{equation}
Another easy computation gives $E=E_{0}'$ microlocally on $K_{\alpha\beta}''$
and we may similarly estimate similarly have 
\begin{equation}
\left\Vert \left(E-E_{0}'\right)\Pi^{E_{0}'}\right\Vert =O\left(h^{\infty}\right).\label{eq:hyp 2a}
\end{equation}
Next we define $A_{\alpha}^{0}\coloneqq U^{*}A_{\alpha}U,\,A_{\beta}^{0}\coloneqq U^{*}A_{\beta}U\in\Psi_{\textrm{cl}}^{0}\left(\mathbb{R}^{n}\right)$
and again note

\begin{eqnarray*}
UA_{\alpha}^{0}A_{\beta}^{0}U^{*} & = & A_{\alpha}A_{\beta}\quad\mbox{microlocally on }K_{\alpha\beta}''\\
U^{*}A_{\alpha}A_{\beta}U & = & A_{\alpha}^{0}A_{\beta}^{0}\quad\mbox{microlocally on }\kappa\left(K_{\alpha\beta}''\right).
\end{eqnarray*}
This again gives 
\begin{eqnarray}
\left\Vert \left[UA_{\alpha}^{0}A_{\beta}^{0}U^{*}-A_{\alpha}A_{\beta}\right]\Pi^{E}\right\Vert  & = & O\left(h^{\infty}\right)\label{eq:hyp 3a}\\
\left\Vert \left[U^{*}A_{\alpha}A_{\beta}U-A_{\alpha}^{0}A_{\beta}^{0}\right]\Pi^{E_{0}}\right\Vert  & = & O\left(h^{\infty}\right).\label{eq:hyp 3b}
\end{eqnarray}
Now using \prettyref{eq:hyp 1a}, \prettyref{eq:hyp 1b}, \prettyref{eq:hyp 2a},
\prettyref{eq:hyp 3a}, \prettyref{eq:hyp 3b} and using the cyclicity
of the trace we may apply \prettyref{prop:apdx prop for R^n red.}
of \prettyref{sec:Appendix A} with $\rho\left(x\right)=f\left(\frac{x+3\tau}{\sqrt{h}}\right)\check{\vartheta}\left(\frac{\lambda\sqrt{h}-3\tau-x}{h}\right)$
to get 
\begin{align*}
 & \textrm{tr}\left[A_{\alpha}f\left(\frac{D'}{\sqrt{h}}\right)\check{\vartheta}\left(\frac{\lambda\sqrt{h}-D'}{h}\right)A_{\beta}\right]-\textrm{tr}\left[A_{\alpha}^{0}f\left(\frac{D_{0}'}{\sqrt{h}}\right)\check{\vartheta}\left(\frac{\lambda\sqrt{h}-D_{0}'}{h}\right)A_{\beta}^{0}\right]\\
= & O\left(h^{\infty}\right)
\end{align*}
for $D_{0}'\coloneqq E_{0}+3\tau$. Finally observing $D=D'$ on $V_{\alpha\beta}$,
$D_{0}=D_{0}'$ on $V_{\alpha\beta}^{0}$ and using \prettyref{lem:changing symbol near crit set}
completes the proof.
\end{proof}

\section{\label{sec: Birkhoff normal form}Birkhoff normal form for the Dirac
operator}

In this section we derive a Birkhoff normal form for the Dirac operator
\prettyref{eq:Dirac op weyl quantization} on $\mathbb{R}^{n}$. First
consider the function 
\[
f_{0}:=-\frac{4x_{0}}{\pi}+\sum_{j=1}^{m}\left(x_{j}x_{j+m}+\xi_{j}\xi_{j+m}\right).
\]
If $H_{f_{0}}$ and $e^{tH_{f_{0}}}$ denote the Hamilton vector field
and time $t$ flow of $f_{0}$ respectively then it is easy to compute
\begin{eqnarray*}
e^{\frac{\pi}{4}H_{f_{0}}}\left(x_{0},\xi_{0}\right) & = & \left(x_{0},\xi_{0}+1\right)\\
e^{\frac{\pi}{4}H_{f_{0}}}\left(x_{j},\xi_{j};x_{j+m}\xi_{j+m}\right) & = & \left(\frac{x_{j}+\xi_{j+m}}{\sqrt{2}},\frac{-x_{j+m}+\xi_{j}}{\sqrt{2}};\frac{x_{j+m}+\xi_{j}}{\sqrt{2}},\frac{-x_{j}+\xi_{j+m}}{\sqrt{2}}\right).
\end{eqnarray*}
We abbreviate $\left(x',\xi'\right)=\left(x_{1},\ldots,x_{m};\xi_{1},\ldots,\xi_{m}\right)$,
\\
$\left(x'',\xi''\right)=\left(x_{m+1},\ldots,x_{2m};\xi_{m+1},\ldots,\xi_{2m}\right)$
and $\left(x,\xi\right)=\left(x_{0},x',x'';\xi_{0},\xi',\xi''\right)$.
Further, let $o_{N}\subset S_{\textrm{cl}}^{1}\left(\mathbb{R}^{2n};\mathbb{C}^{l}\right)$
denote the subspace of self-adjoint symbols $a:\left(0,1\right]_{h}\rightarrow C^{\infty}\left(\mathbb{R}_{x,\xi}^{2n};i\mathfrak{u}\left(2^{m}\right)\right)$
such that each of the coefficients $a_{k}$, $k=0,1,2,\ldots$ in
its symbolic expansion \prettyref{eq: symbolic expansion} vanishes
to order $N$ in $\left(x_{0},x',\xi'\right)$. We also denote by
$o_{N}$ the space of Weyl quantizations of such symbols.

Using Egorov's theorem, the operator \prettyref{eq:Dirac op weyl quantization}
is conjugated to 
\begin{align}
e^{\frac{i\pi}{4h}f_{0}^{W}}D_{0}e^{-\frac{i\pi}{4h}f_{0}^{W}} & =d_{0}^{W},\quad\textrm{ with}\label{eq:first conjugation normal form}\\
d_{0} & =\sqrt{2}\left(\sigma^{j}w_{j,f_{0}}^{0}\xi_{0}+\sigma^{j}w_{j,f_{0}}^{k}\xi_{k}+\sigma^{j}w_{j,f_{0}}^{k+m}x_{k}\right)+ho_{0}\label{eq:symbol Dirac operator}\\
\textrm{where }\;w_{j,f_{0}}^{k} & =\left(e^{-\frac{\pi}{4}H_{f_{0}}}\right)^{*}w_{j}^{k}
\end{align}
 A Taylor expansion of $d_{0}$ \prettyref{eq:symbol Dirac operator}
now gives $r_{j}^{0}\in o_{2}$, $0\leq j\leq2m$, such that 
\begin{eqnarray*}
d_{0} & = & \sqrt{2}\sigma^{j}\left(\bar{w}_{j}^{0}\xi_{0}+\bar{w}_{j}^{k}\xi_{k}+\bar{w}_{j}^{k+m}x_{k}\right)+\sigma^{j}r_{j}^{0}+ho_{0}
\end{eqnarray*}
and where $\bar{w}_{j}^{k}\left(x_{0},x'',\xi''\right)=w_{j}^{k}\left(x_{0},-\frac{\xi''}{\sqrt{2}},\frac{x''}{\sqrt{2}}\right)$.
On squaring using \prettyref{eq: diagonalizing metric} we obtain
\begin{eqnarray*}
\left(d_{0}^{W}\right)^{2} & = & Q_{0}^{W}+ho_{1}+o_{3}+h^{2}o_{0},\:\textrm{ with}\\
Q_{0} & = & \begin{bmatrix}x' & \xi_{0} & \xi'\end{bmatrix}\begin{bmatrix}\bar{g}^{\left(k+m\right)\left(l+m\right)}\left(x_{0},x'',\xi''\right) & \bar{g}^{\left(k+m\right)0}\left(x_{0},x'',\xi''\right) & \bar{g}^{\left(k+m\right)l}\left(x_{0},x'',\xi''\right)\\
\bar{g}^{0\left(l+m\right)}\left(x_{0},x'',\xi''\right) & \bar{g}^{00}\left(x_{0},x'',\xi''\right) & \bar{g}^{0l}\left(x_{0},x'',\xi''\right)\\
\bar{g}^{k\left(l+m\right)}\left(x_{0},x'',\xi''\right) & \bar{g}^{k0}\left(x_{0},x'',\xi''\right) & \bar{g}^{kl}\left(x_{0},x'',\xi''\right)
\end{bmatrix}\begin{bmatrix}x'\\
\xi_{0}\\
\xi'
\end{bmatrix}.
\end{eqnarray*}
Here $\bar{g}^{kl}\left(x_{0},x'',\xi''\right)=2g^{kl}\left(x_{0},-\frac{\xi''}{\sqrt{2}},\frac{x''}{\sqrt{2}}\right)$
and $g^{kl}$ the components of the inverse metric on $T^{*}\mathbb{R}^{n}$.

Next we consider another function $f_{1}$ of the form 
\[
f_{1}=\frac{1}{2}\begin{bmatrix}x' & \xi_{0} & \xi'\end{bmatrix}\begin{bmatrix}\alpha_{m\times m}\left(x_{0},x'',\xi''\right) & \gamma_{m\times m+1}\left(x_{0},x'',\xi''\right)\\
\gamma_{m+1\times m}^{t}\left(x_{0},x'',\xi''\right) & \beta_{m+1\times m+1}\left(x_{0},x'',\xi''\right)
\end{bmatrix}\begin{bmatrix}x'\\
\xi_{0}\\
\xi'
\end{bmatrix}
\]
where $\alpha,\beta$ and $\gamma$ are matrix valued functions of
the given orders, with $\alpha,\beta$ being symmetric. An easy computation
now shows 
\begin{eqnarray*}
\left(e^{H_{f_{1}}}\right)^{*}\begin{bmatrix}x'\\
\xi_{0}\\
\xi'
\end{bmatrix} & = & e^{\Lambda}\begin{bmatrix}x'\\
\xi_{0}\\
\xi'
\end{bmatrix}+o_{2}\quad\textrm{with}\\
\Lambda\left(x_{0},x'',\xi''\right) & = & \begin{bmatrix}0 & -I_{m+1\times m+1}\\
I_{m\times m} & 0
\end{bmatrix}\begin{bmatrix}\alpha_{m\times m}\left(x_{0},x'',\xi''\right) & \gamma_{m\times m+1}\left(x_{0},x'',\xi''\right)\\
\gamma_{m+1\times m}^{t}\left(x_{0},x'',\xi''\right) & \beta_{m+1\times m+1}\left(x_{0},x'',\xi''\right)
\end{bmatrix}.
\end{eqnarray*}
From the suitability assumption \prettyref{eq:Diagonalizability assumption-1},
we have that there exists a smooth matrix valued functions $\alpha,\beta$
and $\gamma$ such that 
\begin{eqnarray*}
 &  & \begin{bmatrix}x' & \xi_{0} & \xi'\end{bmatrix}e^{\Lambda^{t}}\begin{bmatrix}\bar{g}^{\left(k+m\right)\left(l+m\right)}\left(x_{0},x'',\xi''\right) & \bar{g}^{\left(k+m\right)0}\left(x_{0},x'',\xi''\right) & \bar{g}^{\left(k+m\right)l}\left(x_{0},x'',\xi''\right)\\
\bar{g}^{0\left(l+m\right)}\left(x_{0},x'',\xi''\right) & \bar{g}^{00}\left(x_{0},x'',\xi''\right) & \bar{g}^{0l}\left(x_{0},x'',\xi''\right)\\
\bar{g}^{k\left(l+m\right)}\left(x_{0},x'',\xi''\right) & \bar{g}^{k0}\left(x_{0},x'',\xi''\right) & \bar{g}^{kl}\left(x_{0},x'',\xi''\right)
\end{bmatrix}e^{\Lambda}\begin{bmatrix}x'\\
\xi_{0}\\
\xi'
\end{bmatrix}\\
 & = & \xi_{0}^{2}+\bar{\nu}\left[\sum_{j=1}^{m}\mu_{j}\left(x_{j}^{2}+\xi_{j}^{2}\right)\right]+o_{3}
\end{eqnarray*}
where
\begin{equation}
\bar{\nu}\left(x_{0},x'',\xi''\right)=\nu\left(x_{0},-\frac{\xi''}{\sqrt{2}},\frac{x''}{\sqrt{2}}\right).\label{eq: bar nu}
\end{equation}
Letting 
\[
H_{2}=\frac{1}{2}\sum_{j=1}^{m}\mu_{j}\left(x_{j}^{2}+\xi_{j}^{2}\right),
\]
 Egorov's theorem now gives
\begin{eqnarray}
e^{\frac{i}{h}f_{1}^{W}}d_{0}^{W}e^{-\frac{i}{h}f_{1}^{W}} & = & \left(\sum_{j=0}^{2m}\sigma_{j}b_{j}\right)^{W}+ho_{0}\quad\textrm{with }\label{eq:symbol dirac d_1}\\
\sum_{j=0}^{2m}b_{j}^{2} & = & \left(\xi_{0}^{2}+2\bar{\nu}H_{2}\right)^{W}+o_{3}.\nonumber 
\end{eqnarray}
Another Taylor expansion in the variables $\left(x',\xi_{0},\xi'\right)$
gives $A=\left(a_{jk}\left(x_{0},x'',\xi''\right)\right)\in C^{\infty}\left(\mathbb{R}_{\left(x_{0},x'',\xi''\right)}^{n};\mathfrak{so}\left(n\right)\right)$
and $r_{j}\in o_{2}$, $j=0,\ldots,2m$, such that 
\[
e^{-A}\begin{bmatrix}b_{0}\\
\vdots\\
b_{2m}
\end{bmatrix}=\begin{bmatrix}\xi_{0}\\
\left(2\bar{\nu}\mu_{1}\right)^{\frac{1}{2}}x_{1}\\
\left(2\bar{\nu}\mu_{1}\right)^{\frac{1}{2}}\xi_{1}\\
\vdots\\
\left(2\bar{\nu}\mu_{m}\right)^{\frac{1}{2}}x_{m}\\
\left(2\bar{\nu}\mu_{m}\right)^{\frac{1}{2}}\xi_{m}
\end{bmatrix}+\begin{bmatrix}r_{0}\\
\vdots\\
r_{2m}
\end{bmatrix}.
\]
We may now set $c_{A}=\frac{1}{i}a_{jk}\sigma^{j}\sigma^{k}\in C^{\infty}\left(\mathbb{R}_{\left(x_{0},x'',\xi''\right)}^{n};i\mathfrak{u}\left(2^{m}\right)\right)$
and compute 
\begin{align}
e^{ic_{A}^{W}}e^{\frac{i}{h}f_{1}^{W}}d_{0}^{W}e^{-\frac{i}{h}f_{1}^{W}}e^{-ic_{A}^{W}} & =d_{1}^{W},\quad\textrm{where}\label{eq: d0 =000026 d1 are conjugate}\\
d_{1} & =H_{1}+\sigma^{j}r_{j}+ho_{0},\quad\textrm{and}\label{eq:dirac operator for formal birkhoff}\\
H_{1} & \coloneqq\xi_{0}\sigma_{0}+\left(2\bar{\nu}\right)^{\frac{1}{2}}\sum_{j=1}^{m}\mu_{j}^{\frac{1}{2}}\left(x_{j}\sigma_{2j-1}+\xi_{j}\sigma_{2j}\right).\label{eq: model Dirac symbol}
\end{align}

\subsection{\label{sub: Weyl product and Koszul}Weyl product and Koszul complexes}

We now derive a formal Birkhoff normal form for the symbol $d_{1}$
in \prettyref{eq:dirac operator for formal birkhoff}. First denote
by $R=C^{\infty}\left(x_{0},x'',\xi''\right)$ the ring of real valued
functions in the given $2m+1$ variables. Further define 
\[
S\coloneqq R\left\llbracket x',\xi_{0},\xi';h\right\rrbracket 
\]
the ring of formal power series in the further given $2m+2$ variables
with coefficients in $R$. The ring $S\otimes\mathbb{C}$ is now equipped
with the Weyl product 
\[
a\ast b\coloneqq\left[e^{\frac{ih}{2}\left(\partial_{r_{1}}\partial_{s_{2}}-\partial_{r_{2}}\partial_{s_{1}}\right)}\left(a\left(s_{1},r_{1};h\right)b\left(s_{2},r_{2};h\right)\right)\right]_{x=s_{1}=s_{2},\xi=r_{1}=r_{2}},
\]
corresponding to the composition formula \prettyref{eq: Weyl product}
for pseudodifferential operators, with 
\begin{eqnarray*}
\left[a,b\right] & \coloneqq & a\ast b-b\ast a
\end{eqnarray*}
being the corresponding Weyl bracket. It is an easy exercise to show
that for $a,b\in S$ real valued, the commutator $i\left[a,b\right]\in S$
is real valued.

Next, we define a filtration on $S$. Each monomial $h^{k}\xi_{0}\left(x'\right)^{\alpha}\left(\xi'\right)^{\beta}$
in $S$ is given the weight $2k+a+\left|\alpha\right|+\left|\beta\right|$.
The ring $S$ is equipped with a decreasing filtration 
\begin{eqnarray*}
S=O_{0} & \supset & O_{1}\supset\ldots\supset O_{N}\supset\ldots,\\
\bigcap_{N}O_{N} & = & \left\{ 0\right\} ,
\end{eqnarray*}
where $O_{N}$ consists of those power series with monomials of weight
$N$ or more. It is an exercise to show that 
\begin{eqnarray*}
O_{N}\ast O_{M} & \subset & O_{N+M}\\
\left[O_{N},O_{M}\right] & \subset & ihO_{N+M-2}.
\end{eqnarray*}
The associated grading is given by 
\[
S=\bigoplus_{N=0}^{\infty}S_{N}
\]
where $S_{N}$ consists of those power series with monomials of weight
exactly $N$ . We also define the quotient ring $D_{N}\coloneqq S/O_{N+1}$
whose elements may be identified with the set of homogeneous polynomials
with monomials of weight at most $N$. The ring $D_{N}$ is also similarly
graded and filtered. In similar vein, we may also define the ring
\[
S\left(m\right)=S\otimes\mathfrak{gl}_{\mathbb{C}}\left(2^{m}\right)
\]
of $R\otimes\mathfrak{gl}_{\mathbb{C}}\left(2^{m}\right)$ valued
formal power series in $\left(x',\xi_{0},\xi';h\right)$. The ring
$S\left(m\right)$ is equipped with an induced product $\ast$ and
decreasing filtration
\begin{eqnarray*}
O_{0}\left(m\right) & \supset & O_{1}\left(m\right)\supset\ldots\supset O_{N}\left(m\right)\supset\ldots,\\
\bigcap_{N}O_{N}\left(m\right) & = & \left\{ 0\right\} ,
\end{eqnarray*}
where $O_{N}\left(m\right)=O_{N}\otimes\mathfrak{gl}_{\mathbb{C}}\left(2^{m}\right)$.
It is again a straightforward exercise to show that for $a,b\in S\otimes i\mathfrak{u}_{\mathbb{C}}\left(2^{m}\right)$
self-adjoint, the commutator $i\left[a,b\right]\in S\otimes i\mathfrak{u}_{\mathbb{C}}\left(2^{m}\right)$
is self-adjoint.

\subsubsection{Koszul complexes}

Let us now again consider the $2m$ and $2m+1$ dimensional real inner
product spaces $V=\mathbb{R}\left[e_{1},\ldots,e_{2m}\right]$ and
$W=\mathbb{R}\left[e_{0}\right]\oplus V$ from \prettyref{sub:Clifford algebra}.
Considering the chain groups $D_{N}\otimes\Lambda^{k}V$, $k=0,1,\ldots,n$,
one may define four differentials 
\begin{eqnarray*}
w_{x}^{0} & = & \sum_{j=1}^{m}\mu_{j}^{\frac{1}{2}}\left(x_{j}e_{2j-1}\land+\xi_{j}e_{2j}\land\right)\\
i_{x}^{0} & = & \sum_{j=1}^{m}\mu_{j}^{\frac{1}{2}}\left(x_{j}i_{e_{2j-1}}+\xi_{j}i_{e_{2j}}\right)\\
w_{\partial}^{0} & = & \sum_{j=1}^{m}\mu_{j}^{\frac{1}{2}}\left(\partial_{x_{j}}e_{2j-1}\land+\partial_{\xi_{j}}e_{2j}\land\right)\\
i_{\partial}^{0} & = & \sum_{j=1}^{m}\mu_{j}^{\frac{1}{2}}\left(\partial_{x_{j}}i_{e_{2j-1}}+\partial_{\xi_{j}}i_{e_{2j}}\right).
\end{eqnarray*}
We equip $D_{N}$ with the $R\left\llbracket h\right\rrbracket $-valued
inner products where the distinct monomials $\frac{1}{\sqrt{a!\alpha!\beta!}}\xi_{0}^{a}\left(x'\right)^{\alpha}\left(\xi'\right)^{\beta}$
are orthonormal. With these inner products $w_{x}^{0},i_{\partial}^{0}$
and $w_{\partial}^{0},i_{x}^{0}$ are respectively adjoints. The combinatorial
Laplacians $\Delta^{0}=w_{x}^{0}i_{\partial}^{0}+i_{\partial}^{0}w_{x}^{0}=w_{\partial}^{0}i_{x}^{0}+i_{x}^{0}w_{\partial}^{0}$,
are computed to be equal and act on basis elements $\xi_{0}^{a}\left(x'\right)^{\alpha}\left(\xi'\right)^{\beta}\left(\bigwedge e_{j}^{\gamma_{j}}\right)$
via multiplication by $\mu.\left(2\left(\alpha+\beta\right)+\gamma\right)$.
It now follows that these have (co-)homology only in degree zero given
by $R\left\llbracket h\right\rrbracket $.

Similarly, we may consider the chain groups $D_{N}\otimes\Lambda^{k}W$,
$k=0,1,\ldots,n$, one may define four differentials
\begin{eqnarray*}
w_{x} & = & \xi_{0}e_{0}\land+\left(2\bar{\nu}\right)^{\frac{1}{2}}w_{x}^{0}\\
i_{x} & = & \xi_{0}i_{e_{0}}+\left(2\bar{\nu}\right)^{\frac{1}{2}}i_{x}^{0}\\
w_{\partial} & = & \partial_{\xi_{0}}e_{0}\land+\left(2\bar{\nu}\right)^{\frac{1}{2}}w_{\partial}^{0}\\
i_{\partial} & = & \partial_{\xi_{0}}i_{e_{0}}+\left(2\bar{\nu}\right)^{\frac{1}{2}}i_{\partial}^{0}.
\end{eqnarray*}
Again these complexes have cohomology only in degree zero given by
$R\left\llbracket h\right\rrbracket $.

Next, we define twisted Koszul differentials on $D_{N}\otimes\Lambda^{k}V$
via
\begin{eqnarray*}
\tilde{w}_{\partial}^{0} & = & \frac{i}{h}\sum_{j=1}^{m}\mu_{j}^{\frac{1}{2}}\left(\textrm{ad}_{x_{j}}e_{2j-1}\land+\textrm{ad}_{\xi_{j}}e_{2j}\land\right)=\sum_{j=1}^{m}\mu_{j}^{\frac{1}{2}}\left(\partial_{x_{j}}e_{2j}\land-\partial_{\xi_{j}}e_{2j-1}\land\right)\\
\tilde{i}_{\partial}^{0} & = & \frac{i}{h}\sum_{j=1}^{m}\mu_{j}^{\frac{1}{2}}\left(\textrm{ad}_{x_{j}}i_{e_{2j-1}}+\textrm{ad}_{\xi_{j}}i_{e_{2j}}\right)=\sum_{j=1}^{m}\mu_{j}^{\frac{1}{2}}\left(\partial_{x_{j}}i_{e_{2j}}-\partial_{\xi_{j}}i_{e_{2j-1}}\right).
\end{eqnarray*}
We note that the above are symplectic adjoints to their untwisted
counterparts with respect to the symplectic pairing $\sum_{j=1}^{m}e_{2j-1}\wedge e_{2j}$
on $V$. 

Similar twisted Koszul differentials on $D_{N}\otimes\Lambda^{k}W$
are defined via

\begin{eqnarray*}
\tilde{w}_{\partial} & = & \frac{i}{h}\textrm{ad}_{\xi_{0}}e_{0}\wedge+\left(2\bar{\nu}\right)^{\frac{1}{2}}\tilde{w}_{\partial}^{0}=-\partial_{x_{0}}e_{0}\wedge+\left(2\bar{\nu}\right)^{\frac{1}{2}}\tilde{w}_{\partial}^{0}\\
\tilde{i}_{\partial} & = & \frac{i}{h}i_{e_{0}}\textrm{ad}_{\xi_{0}}+\left(2\bar{\nu}\right)^{\frac{1}{2}}\tilde{i}_{\partial}^{0}=-\partial_{x_{0}}i_{e_{0}}+\left(2\bar{\nu}\right)^{\frac{1}{2}}\tilde{i}_{\partial}^{0}.
\end{eqnarray*}
These twisted differentials correspond to the untwisted ones by a
mere change of basis in $V$, $W$ and hence also have (co-)homology
only in degree zero given by $R\left\llbracket h\right\rrbracket $. 

We now compute the twisted combinatorial Laplacian to be 
\begin{eqnarray*}
\tilde{\Delta}^{0} & = & \tilde{w}_{\partial}^{0}i_{x}^{0}+i_{x}^{0}\tilde{w}_{\partial}^{0}\\
 & = & -\left(w_{x}^{0}\tilde{i}_{\partial}^{0}+\tilde{i}_{\partial}^{0}w_{x}^{0}\right)\\
 & = & \sum_{j=1}^{m}\mu_{j}\left[\xi_{j}\partial_{x_{j}}-x_{j}\partial_{\xi_{j}}+e_{2j}i_{e_{2j-1}}-e_{2j-1}i_{e_{2j}}\right].
\end{eqnarray*}
One may similarly define $\tilde{\Delta}$. Next, we define the space
of twisted $\tilde{\Delta}^{0}$-harmonic, $\xi_{0}$- independent
elements 
\begin{eqnarray*}
\mathcal{H}_{N}^{k} & = & \left\{ \omega\in D_{N}\otimes\Lambda^{k}W|\,\tilde{\Delta}^{0}\omega=0,\,\partial_{\xi_{0}}\omega=0\right\} \\
\mathcal{H}^{k} & = & \left\{ \omega\in S\otimes\Lambda^{k}W|\,\tilde{\Delta}^{0}\omega=0,\,\partial_{\xi_{0}}\omega=0\right\} .
\end{eqnarray*}
We now prove a twisted version of the Hodge decomposition theorem.
\begin{lem}
\label{lem:Koszul-Hodge decomposition }The $k$-th chain group is
spanned by the three subspaces 
\[
D_{N}\otimes\Lambda^{k}W=\mathbb{R}\left[\textrm{Im}\left(i_{x}\tilde{w}_{\partial}\right),\textrm{Im}\left(\tilde{w}_{\partial}i_{x}\right),\mathcal{H}_{N}^{k}\right].
\]
\end{lem}
\begin{proof}
We first compute $\tilde{\Delta}$ in terms of $\tilde{\Delta}^{0}$
to be 
\[
\tilde{\Delta}=-\xi_{0}\partial_{x_{0}}+2\bar{\nu}\tilde{\Delta}^{0}-2\left(\partial_{x_{0}}\bar{\nu}^{\frac{1}{2}}\right)e_{0}i_{x}^{0}.
\]
Next, since $\tilde{\Delta}^{0}$ is skew-adjoint, we may decompose
\[
D_{N}\otimes\Lambda^{k}W=E_{0}\oplus\bigoplus_{\lambda>0}\left[E_{i\lambda}\oplus E_{-i\lambda}\right]
\]
into its eigenspaces. We may now invert $\tilde{\Delta}$ on the non-zero
eigenspaces of $\tilde{\Delta}^{0}$ above using the Volterra series
\[
\tilde{\Delta}^{-1}=\left(2\bar{\nu}\tilde{\Delta}^{0}\right)^{-1}\sum_{j=0}^{\infty}\left[\left(2\bar{\nu}\tilde{\Delta}^{0}\right)^{-1}\left(\xi_{0}\partial_{x_{0}}+2\left(\partial_{x_{0}}\bar{\nu}^{\frac{1}{2}}\right)e_{0}i_{x}^{0}\right)\right]^{j}.
\]
The sum above is finite since $\xi_{0}\partial_{x_{0}}+2\left(\partial_{x_{0}}\bar{\nu}^{\frac{1}{2}}\right)e_{0}i_{x}^{0}$
is nilpotent on $D_{N}\otimes\Lambda^{k}W$. Thus we have 
\[
\bigoplus_{\lambda>0}\left[E_{i\lambda}\oplus E_{-i\lambda}\right]\subset\textrm{Im}\left(\tilde{\Delta}\right)\subset\mathbb{R}\left[\textrm{Im}\left(i_{x}\tilde{w}_{\partial}\right),\textrm{Im}\left(\tilde{w}_{\partial}i_{x}\right)\right].
\]
Finally, we decompose 
\[
E_{0}=\bigoplus_{j=0}^{N}\xi_{0}^{j}\mathcal{H}_{N}^{k}
\]
and write each $\omega\in\xi_{0}^{j}\mathcal{H}_{N}^{k}$, $j\geq1$,
as 
\begin{eqnarray*}
\omega & = & \omega_{0}+\tilde{\Delta}\omega_{1}\\
\omega_{0} & = & \left[-2\left(\partial_{x_{0}}\bar{\nu}^{\frac{1}{2}}\right)e_{0}i_{x}^{0}\xi_{0}^{-1}\int_{0}^{x_{0}}\right]^{j}\omega\in\mathcal{H}_{N}^{k}\\
\omega_{1} & = & -\left(\xi_{0}^{-1}\int_{0}^{x_{0}}\right)\sum_{l=0}^{j-1}\left[-2\left(\partial_{x_{0}}\bar{\nu}^{\frac{1}{2}}\right)e_{0}i_{x}^{0}\xi_{0}^{-1}\int_{0}^{x_{0}}\right]^{l}\omega
\end{eqnarray*}
to complete the proof.
\end{proof}

\subsection{Formal Birkhoff normal form}

The importance of the Koszul complexes introduced in the previous
subsection is in continuing the Birkhoff normal form procedure for
the symbol $d_{1}$ in \prettyref{eq:dirac operator for formal birkhoff}.
The remaining steps in the procedure are formal.

First let us define the Clifford quantization of an element in $a\in S\otimes\Lambda^{k}W$,
using \prettyref{eq: clifford quantization} as an element in 
\[
c_{0}\left(a\right)\coloneqq i^{\frac{k\left(k+1\right)}{2}}c\left(a\right)\in S\left(m\right).
\]
It is clear from \prettyref{eq:clifford quantization adjoint} and
\prettyref{eq:clifford algebra is matrix algebra} this gives an isomorphism
\begin{equation}
c_{0}:S\otimes\Lambda^{\textrm{odd/even}}W\rightarrow S\otimes i\mathfrak{u}_{\mathbb{C}}\left(2^{m}\right)\label{eq: clifford isomorphism}
\end{equation}
 of real elements of the even or odd exterior algebra with self-adjoint
elements in $S\left(m\right)$. It is clear from \prettyref{eq:dirac operator for formal birkhoff}
that 
\begin{equation}
d_{1}=H_{1}+c_{0}\left(r\right)+hS\otimes i\mathfrak{u}_{\mathbb{C}}\left(2^{m}\right)\label{eq:starting point normal form}
\end{equation}
for $r\coloneqq\sum_{j=1}^{n}r_{j}e_{j}\in O_{2}\otimes W$. 

For $a\in\Lambda^{k}W$, we define $\left[a\right]\coloneqq\left[\frac{k}{2}\right]$.
Now for $f\in O_{N},$ $N\geq3$ and $a\in O_{N}\otimes\Lambda^{\textrm{even}}W$,
$N\geq1$, we may compute the conjugations 
\begin{align}
e^{\frac{i}{h}f}H_{1}e^{-\frac{i}{h}f} & =H_{1}+c_{0}\left(\tilde{w}_{\partial}f\right)+O_{N}\otimes i\mathfrak{u}_{\mathbb{C}}\left(2^{m}\right)\label{eq: conjugation Koszul 1}\\
e^{ic_{0}\left(a\right)}H_{1}e^{-ic_{0}\left(a\right)} & =H_{1}+\left(-1\right)^{\left[a\right]+1}2c_{0}\left(i_{x}a\right)+hc_{0}\left(\tilde{w}_{\partial}a\right)+O_{N+2}\otimes i\mathfrak{u}_{\mathbb{C}}\left(2^{m}\right)\label{eq: conjugation Koszul 2}
\end{align}
in terms of the Koszul differentials.

We now come to the formal Birkhoff normal form for the symbol $d_{1}$.
\begin{prop}
\label{prop: normal form d1}There exist $f\in O_{3}$, $a\in O_{2}\otimes\Lambda^{\textrm{even}}W$
and $\omega\in\mathcal{H}^{\textrm{odd}}\cap O_{2}$ such that 
\begin{equation}
e^{ic_{0}\left(a\right)}e^{\frac{i}{h}f}d_{1}e^{-\frac{i}{h}f}e^{-ic_{0}\left(a\right)}=H_{1}+c_{0}\left(\omega\right).\label{eq: normal form d1}
\end{equation}
\end{prop}
\begin{proof}
We first prove that for each $N\geq1$, there exist $f_{N}\in O_{3}$,
$a_{N}^{0}\in O_{1}\otimes\Lambda^{\textrm{2}}W$, $\omega_{N}^{0}\in\mathcal{H}^{1}\cap O_{2}$
and $r_{N}^{0}\in O_{N+1}\otimes W$ such that 
\begin{align}
e^{ic_{0}\left(a_{N}^{0}\right)}e^{\frac{i}{h}f_{N}}d_{1}e^{-\frac{i}{h}f_{N}}e^{-ic_{0}\left(a_{N}^{0}\right)} & =H_{1}+c_{0}\left(\omega_{N}^{0}\right)+c_{0}\left(r_{N}^{0}\right)+hS\otimes i\mathfrak{u}_{\mathbb{C}}\left(2^{m}\right),\label{eq: induction hyp normal form}\\
f_{N+1}-f_{N} & \in O_{N+2},\nonumber \\
a_{N+1}^{0}-a_{N}^{0} & \in O_{N},\nonumber \\
\omega_{N+1}^{0}-\omega_{N}^{0} & \in O_{N+1}.\nonumber 
\end{align}
The base case $N=1$ is given by \prettyref{eq:starting point normal form}
with $a_{1}^{0}=f_{1}=\omega_{1}^{0}=0$ and $r_{1}^{0}=r$. To complete
the induction step we decompose 
\begin{eqnarray}
r_{N}^{0} & = & \underbrace{u_{N}^{0}}_{\in S_{N+1}\otimes W}+\underbrace{r_{N+1}^{0}}_{\in O_{N+2}\otimes W}.\label{eq: homogeneous decomposition remainder}
\end{eqnarray}
Next we use \prettyref{lem:Koszul-Hodge decomposition } to find $b_{N},g_{N}\in O_{N+1}\otimes W$
and $\upsilon_{N}^{0}\in\mathcal{H}^{1}\cap S_{N+1}$ such that 
\begin{equation}
u_{N}^{0}=\upsilon_{N}^{0}-i_{x}\tilde{w}_{\partial}b_{N}^{0}-\tilde{w}_{\partial}i_{x}g_{N}^{0}+O_{N+2}\label{eq: Hodge decomposition remainder}
\end{equation}
Next, define $f_{N+1}=f_{N}+i_{x}g_{N}^{0}\in O_{3}$ , $a_{N+1}^{0}=a_{N}^{0}+\frac{1}{2}\tilde{w}_{\partial}b_{N}^{0}\in O_{1}\otimes\Lambda^{\textrm{2}}W$
and $\omega_{N+1}^{0}=\omega_{N}^{0}+\upsilon_{N}^{0}$. We now use
\prettyref{eq: conjugation Koszul 1}, \prettyref{eq: conjugation Koszul 2},
\prettyref{eq: homogeneous decomposition remainder} and \prettyref{eq: Hodge decomposition remainder}
to compute 
\begin{eqnarray*}
 &  & e^{ic_{0}\left(a_{N+1}^{0}\right)}e^{\frac{i}{h}f_{N+1}}d_{1}e^{-\frac{i}{h}f_{N+1}}e^{-ic_{0}\left(a_{N+1}^{0}\right)}\\
 & = & e^{ic_{0}\left(\frac{1}{2}\tilde{w}_{\partial}b_{N}^{0}\right)}e^{\frac{i}{h}i_{x}g_{N}^{0}}H_{1}e^{-\frac{i}{h}i_{x}g_{N}^{0}}e^{-ic_{0}\left(\frac{1}{2}\tilde{w}_{\partial}b_{N}^{0}\right)}\\
 &  & \qquad\qquad\qquad\qquad+c_{0}\left(\omega_{N}^{0}\right)+c_{0}\left(r_{N}^{0}\right)+hS\otimes i\mathfrak{u}_{\mathbb{C}}\left(2^{m}\right)\\
 & = & H_{1}+c_{0}\left(\omega_{N+1}^{0}\right)+c_{0}\left(r_{N+1}^{0}\right)+hS\otimes i\mathfrak{u}_{\mathbb{C}}\left(2^{m}\right)
\end{eqnarray*}
completing the induction step. Now setting $f=\lim_{N\rightarrow\infty}f_{N}$,
$a_{0}=\lim_{N\rightarrow\infty}a_{N}^{0}$ and $\omega_{0}=\lim_{N\rightarrow\infty}\omega_{N}^{0}$
and letting $N\rightarrow\infty$ in \prettyref{eq: induction hyp normal form}
gives the relation 
\begin{equation}
e^{ic_{0}\left(a_{0}\right)}e^{\frac{i}{h}f}d_{1}e^{-\frac{i}{h}f}e^{-ic_{0}\left(a_{0}\right)}=H_{1}+c_{0}\left(\omega_{0}\right)+hS\otimes i\mathfrak{u}_{\mathbb{C}}\left(2^{m}\right).\label{eq: first step formal normal form}
\end{equation}

Next we claim that for each $N\geq0$, there exist $a_{N}\in O_{1}\otimes\Lambda^{\textrm{even}}W$,
$\omega_{N}\in\mathcal{H}^{\ast}\cap O_{2}$ and such that 
\begin{align}
e^{ic_{0}\left(a_{N}\right)}e^{\frac{i}{h}f}d_{1}e^{-\frac{i}{h}f}e^{-ic_{0}\left(a_{N}\right)} & =H_{1}+c_{0}\left(\omega_{N}\right)+hO_{N}\otimes i\mathfrak{u}_{\mathbb{C}}\left(2^{m}\right)\label{eq: 2nd induction hyp normal form}\\
a_{N+1}-a_{N} & \in O_{N+1}\otimes\Lambda^{\textrm{even}}W\nonumber \\
\omega_{N+1}-\omega_{N} & \in\mathcal{H}^{\textrm{odd}}\cap O_{N}\nonumber 
\end{align}
The base case $N=0$ is now provided by \prettyref{eq: first step formal normal form}.
To complete the induction step, we use the isomorphism \prettyref{eq: clifford isomorphism}
to decompose the remainder term in \prettyref{eq: 2nd induction hyp normal form}
above as 
\[
c_{0}\left(u_{N}\right)+ihO_{N+1}\otimes\mathfrak{u}_{\mathbb{C}}\left(2^{m}\right)
\]
for $u_{N}\in S_{N}\otimes\Lambda^{\textrm{odd}}W$. Next we use \prettyref{lem:Koszul-Hodge decomposition }
to find $b_{N},g_{N}\in O_{N}\otimes\Lambda^{\textrm{odd}}W$ and
$\upsilon_{N}\in\mathcal{H}^{\textrm{odd}}\cap S_{N}$ such that 
\begin{equation}
u_{N}=\upsilon_{N}-i_{x}\tilde{w}_{\partial}b_{N}-\tilde{w}_{\partial}i_{x}g_{N}+O_{N+1}\label{eq: Hodge decomposition remainder-1}
\end{equation}
Now define $a_{N+1}=a_{N}+i_{x}g_{N}+h\frac{\left(-1\right)^{\left[b_{N}\right]}}{2}\tilde{w}_{\partial}b_{N}\in O_{1}$
and $\omega_{N+1}=\omega_{N}+\upsilon_{N}$. We now use \prettyref{eq: conjugation Koszul 1},
\prettyref{eq: conjugation Koszul 2}, \prettyref{eq: homogeneous decomposition remainder}
and \prettyref{eq: Hodge decomposition remainder-1} to compute 
\begin{eqnarray*}
 &  & e^{ic_{0}\left(a_{N+1}\right)}e^{\frac{i}{h}f}d_{1}e^{-\frac{i}{h}f}e^{-ic_{0}\left(a_{N+1}\right)}\\
 & = & H_{1}+c_{0}\left(\omega_{N+1}\right)+ihO_{N+1}\otimes\mathfrak{u}_{\mathbb{C}}\left(2^{m}\right).
\end{eqnarray*}
completing the induction step. Now setting $a=\lim_{N\rightarrow\infty}a_{N}$
and $\omega=\lim_{N\rightarrow\infty}\omega_{N}$ and letting $N\rightarrow\infty$
in \prettyref{eq: 2nd induction hyp normal form} gives the proposition.
\end{proof}
$\quad$Finally, we show how the Birkhoff normal form maybe used to
perform a further reduction on the trace. First note that we may similarly
use \prettyref{eq: clifford quantization} to define a self-adjoint
Clifford-Weyl quantization map 
\[
c_{0}^{W}\coloneqq\textrm{Op}\otimes c_{0}:S_{\textrm{cl}}^{0}\left(\mathbb{R}^{2n};\mathbb{C}\right)\otimes\Lambda^{\textrm{odd/even}}W\rightarrow\Psi_{\textrm{cl}}^{0}\left(\mathbb{R}^{n};\mathbb{C}^{2^{m}}\right)
\]
which maps real valued symbols $S_{\textrm{cl}}^{0}\left(\mathbb{R}^{2n};\mathbb{R}\right)\otimes\Lambda^{\textrm{odd/even}}W$
to self-adjoint operators in $\Psi_{\textrm{cl}}^{0}\left(\mathbb{R}^{n};\mathbb{C}^{2^{m}}\right)$.
Similarly we define a space of real-valued, twisted $\tilde{\Delta}^{0}$-harmonic,
$\xi_{0}$- independent symbols 
\[
\mathcal{H}^{k}S_{\textrm{cl}}^{0}\coloneqq\left\{ \omega\in S_{\textrm{cl}}^{0}\left(\mathbb{R}^{2n};\mathbb{R}\right)\otimes\Lambda^{k}W|\,\tilde{\Delta}^{0}\omega=0,\,\partial_{\xi_{0}}\omega=0\right\} .
\]
Next, an application of Borel's lemma by virtue of \prettyref{eq:first conjugation normal form},
\prettyref{eq: d0 =000026 d1 are conjugate} and \prettyref{eq: normal form d1}
gives the existence of
\begin{align*}
\bar{a}\sim & \sum_{j=0}^{\infty}h^{j}\bar{a}_{j}\in S_{\textrm{cl}}^{0}\left(\mathbb{R}^{2n};\mathbb{R}\right)\otimes\Lambda^{\textrm{odd}}W\\
\bar{r}\sim & \sum_{j=0}^{\infty}h^{j}\bar{r}_{j}\in S_{\textrm{cl}}^{0}\left(\mathbb{R}^{2n};\mathbb{R}\right)\otimes\Lambda^{\textrm{odd}}W\\
\bar{f}\sim & \sum_{j=0}^{\infty}h^{j}\bar{f}_{j}\in S_{\textrm{cl}}^{0}\left(\mathbb{R}^{2n};\mathbb{R}\right)\\
\bar{\omega}\sim & \sum_{j=0}^{\infty}h^{j}\bar{\omega}_{j}\in\mathcal{H}^{\textrm{odd}}S_{\textrm{cl}}^{0}
\end{align*}
such that 
\begin{equation}
e^{ic_{0}^{W}\left(\bar{a}\right)}e^{\frac{i}{h}\bar{f}^{W}}d_{0}^{W}e^{-\frac{i}{h}\bar{f}^{W}}e^{-ic_{0}^{W}\left(\bar{a}\right)}=\underbrace{H_{1}^{W}+c_{0}^{W}\left(\bar{\omega}\right)}_{\coloneqq\bar{D}}+c_{0}^{W}\left(\bar{r}\right)\label{eq:normal form conjugation}
\end{equation}
$\textrm{on }\bar{V}_{\alpha\beta}\coloneqq e^{X_{\bar{f}_{0}}}\left(V_{\alpha\beta}^{0}\right)$.
Here $\left\{ \bar{r}_{j}\right\} _{j\in\mathbb{N}_{0}},$$\bar{f}_{0}$,
$\bar{\omega}_{0}$ vanish to infinite, second and second order respectively
along 
\[
\Sigma_{0}^{D_{0}}=\Sigma_{0}^{\bar{D}}=\Sigma_{0}^{\bar{D}+c_{0}^{W}\left(\bar{r}\right)}=\left\{ \xi_{0}=x'=\xi'=0\right\} .
\]
Note that on account of \prettyref{eq: infimum nu} and \prettyref{eq: bar nu}
one again has 
\[
\nu_{0}=\mu_{1}\min_{x\in X}\nu\left(x\right)\leq\mu_{1}\inf_{\mathbb{R}_{x_{0},x'',\xi''}^{n}}\bar{\nu}
\]
Furthermore, since $\bar{\omega}_{0}$ vanishes to second order we
may choose $\bar{\omega}_{0}$ arbitrarily small satisfying the estimate
\begin{equation}
\left\Vert \bar{\omega}_{0}\right\Vert _{C^{1}}<\varepsilon,\label{eq: omega 0 smaall C1 norm}
\end{equation}
for any $\varepsilon>0$, while still satisfying \prettyref{eq:normal form conjugation}.

We note that $\bar{D}\in\Psi_{\textrm{cl}}^{1}\left(\mathbb{R}^{n};\mathbb{C}^{2^{m}}\right)$,
with $\bar{D}+i$ having an elliptic symbol in the class $S^{0}\left(\left\langle \xi_{0},\xi'\right\rangle \right)$,
and is hence essentially self-adjoint as an unbounded operator on
$L^{2}\left(\mathbb{R}^{n};\mathbb{C}^{2^{m}}\right)$. The domain
of its unique self-adjoint extension is $H^{1}\left(\mathbb{R}_{x_{0}}\right)\otimes L^{2}\left(\mathbb{R}_{x',x''}^{n-1};\mathbb{C}^{2^{m}}\right)$
(cf. Ch. 8 in \cite{Dimassi-Sjostrand}). We now set 
\begin{align}
\bar{A}_{\alpha} & \coloneqq e^{ic_{0}^{W}\left(\bar{a}\right)}e^{\frac{i}{h}\bar{f}^{W}}A_{\alpha}^{0}e^{-\frac{i}{h}\bar{f}^{W}}e^{-ic_{0}^{W}\left(\bar{a}\right)}\nonumber \\
\mathcal{T}_{\alpha\beta}^{\vartheta}\left(\bar{D}\right) & \coloneqq\textrm{tr}\left[\bar{A}_{\alpha}f\left(\frac{\bar{D}}{\sqrt{h}}\right)\check{\vartheta}\left(\frac{\lambda\sqrt{h}-\bar{D}}{h}\right)\bar{A}_{\beta}\right]\label{eq: Helffer Sjostrand Dbar}\\
 & =\frac{1}{\pi}\int_{\mathbb{C}}\bar{\partial}\tilde{f}\left(z\right)\check{\vartheta}\left(\frac{\lambda-z}{\sqrt{h}}\right)\textrm{tr }\left[\bar{A}_{\alpha}\left(\frac{1}{\sqrt{h}}\bar{D}-z\right)^{-1}\bar{A}_{\beta}\right]dzd\bar{z}.
\end{align}
We next have the following proposition.
\begin{prop}
\label{prop: Reduction to normal form operator}We have 
\[
\mathcal{T}_{\alpha\beta}^{\vartheta}\left(D_{0}\right)=\mathcal{T}_{\alpha\beta}^{\vartheta}\left(\bar{D}\right)\quad\textrm{mod }h^{\infty}.
\]
\end{prop}
\begin{proof}
Since the conjugations in \prettyref{eq:first conjugation normal form}
and \prettyref{eq:normal form conjugation} are unitary and\\
 $WF\left(\bar{A}_{\alpha}\right),WF\left(\bar{A}_{\beta}\right)\subset\bar{V}_{\alpha\beta}$,
we have 
\[
\mathcal{T}_{\alpha\beta}^{\vartheta}\left(D_{0}\right)=\frac{1}{\pi}\int_{\mathbb{C}}\bar{\partial}\tilde{f}\left(z\right)\check{\vartheta}\left(\frac{\lambda-z}{\sqrt{h}}\right)\textrm{tr }\left[\bar{A}_{\alpha}\left(\frac{1}{\sqrt{h}}\left(\bar{D}+c_{0}^{W}\left(\bar{r}\right)\right)-z\right)^{-1}\bar{A}_{\beta}\right]dzd\bar{z}.
\]
It now remains to do away with the $c_{0}^{W}\left(\bar{r}\right)$
above. Since this term vanishes to infinite order along $\Sigma_{0}^{\bar{D}}=\Sigma_{0}^{\bar{D}+c_{0}^{W}\left(\bar{r}\right)}$,
we may use symbolic calculus to find $P_{N},Q_{N}\in\Psi_{\mbox{cl}}^{0}\left(\mathbb{R}^{n};\mathbb{C}^{2^{m}}\right)$,
$\forall N\geq1$ such that 
\begin{eqnarray}
c_{0}^{W}\left(\bar{r}\right) & = & P_{N}\left(\bar{D}+c_{0}^{W}\left(\bar{r}\right)\right)^{N}\label{eq: c(r) =00003D P D^N}\\
c_{0}^{W}\left(\bar{r}\right) & = & Q_{N}\left(\bar{D}\right)^{N}.\label{eq: c(r) =00003D Q D^N}
\end{eqnarray}
Modifying $\bar{D}$ outside a neighborhood of $\bar{V}_{\alpha\beta}$
using \prettyref{lem:changing symbol near crit set} and \prettyref{prop:Dicrete spectrum criterion}
we may assume that $\bar{D},\bar{D}+c_{0}^{W}\left(\bar{r}\right)$
have discrete spectrum in $\left(-\sqrt{2\nu_{0}},\sqrt{2\nu_{0}}\right)$
and hence 
\begin{eqnarray*}
\mathcal{T}_{\alpha\beta}^{\vartheta}\left(\bar{D}\right) & = & \textrm{tr}\left[\bar{A}_{\alpha}f\left(\frac{\bar{D}}{\sqrt{h}}\right)\check{\vartheta}\left(\frac{\lambda\sqrt{h}-\bar{D}}{h}\right)\bar{A}_{\beta}\right]\\
\mathcal{T}_{\alpha\beta}^{\vartheta}\left(D_{0}\right) & = & \textrm{tr}\left[\bar{A}_{\alpha}f\left(\frac{\bar{D}+c_{0}^{W}\left(\bar{r}\right)}{\sqrt{h}}\right)\check{\vartheta}\left(\frac{\lambda\sqrt{h}-\bar{D}-c_{0}^{W}\left(\bar{r}\right)}{h}\right)\bar{A}_{\beta}\right].
\end{eqnarray*}
Next, with $\Pi^{\bar{D}}=\Pi_{\left[-\sqrt{2\nu_{0}h},\sqrt{2\nu_{0}h}\right]}^{\bar{D}}$
and $\Pi^{\bar{D}+c_{0}^{W}\left(\bar{r}\right)}=\Pi_{\left[-\sqrt{2\nu_{0}h},\sqrt{2\nu_{0}h}\right]}^{\bar{D}+c_{0}^{W}\left(\bar{r}\right)}$
denoting the spectral projections, \prettyref{eq: c(r) =00003D P D^N}
and \prettyref{eq: c(r) =00003D Q D^N} give 
\begin{eqnarray*}
\left\Vert c_{0}^{W}\left(\bar{r}\right)\Pi^{\bar{D}}\right\Vert  & = & O\left(h^{\frac{N}{2}}\right)\\
\left\Vert c_{0}^{W}\left(\bar{r}\right)\Pi^{\bar{D}+c_{0}^{W}\left(\bar{r}\right)}\right\Vert  & = & O\left(h^{\frac{N}{2}}\right)
\end{eqnarray*}
 for each $N\geq1$. Finally applying \prettyref{prop:apdx prop for R^n red.}
with $\rho\left(x\right)=f\left(\frac{x}{\sqrt{h}}\right)\check{\vartheta}\left(\frac{\lambda\sqrt{h}-x}{h}\right)$
and using the cyclicity of the trace gives $\mathcal{T}_{\alpha\beta}^{\vartheta}\left(D_{0}\right)-\mathcal{T}_{\alpha\beta}^{\vartheta}\left(\bar{D}\right)=O\left(h^{-1}h^{\frac{N}{4096}}\right)$,
$\forall N\geq1$, completing the proof.
\end{proof}

\section{\label{sec: Extension of a resolvent}Extension of a resolvent}

In this section we complete the proof of \prettyref{lem: O(h infty) LEMMA}.
On account of the reductions in \prettyref{prop:Reduction to R^n}
and \prettyref{prop: Reduction to normal form operator} in the previous
sections, it suffices to now consider the trace $\mathcal{T}_{\alpha\beta}^{\vartheta}\left(\bar{D}\right)$
. First let $\bar{A}_{\alpha}=a_{\alpha}^{W}$, $\bar{A}_{\beta}=a_{\beta}^{W}$
for $a_{\alpha},a_{\beta}\in S_{\textrm{cl}}^{0}\left(\mathbb{R}^{2n}\right)$.
The conjugations $e^{\frac{it}{h}x_{0}}\bar{A}_{\alpha}e^{-\frac{it}{h}x_{0}}=a_{\alpha,t}^{W}$
and $e^{\frac{it}{h}x_{0}}\bar{A}_{\beta}e^{-\frac{it}{h}x_{0}}=a_{\beta,t}^{W}$
are easily computed in terms of the one-parameter family of symbols
$a_{\alpha,t}\left(\xi_{0},\ldots\right)=a_{\alpha}\left(\xi_{0}+t,\ldots\right),\,a_{\beta}=a_{\beta}\left(\xi_{0}+t,\ldots\right)\in S_{\textrm{cl}}^{0}\left(\mathbb{R}^{2n}\right)$,
$t\in\mathbb{R}$, obtained by translating in the $\xi_{0}$ direction.
One now introduces almost analytic continuations of the symbols $a_{\alpha,t}$,
$a_{\beta,t}$ $\in S_{\textrm{cl}}^{0}\left(\mathbb{R}^{2n}\right)$,
defined for $t\in\mathbb{C}$, such that all the Frechet semi-norms
of $\bar{\partial}a_{\alpha,t}$, $\bar{\partial}a_{\beta,t}$ are
$O\left(\left|\textrm{Im}t\right|^{\infty}\right)$. These maybe further
chosen to have the property that their wavefront sets have uniform
compact support when $t$ is restricted to compact subsets of $\mathbb{C}$.
Again one clearly has
\begin{eqnarray}
a_{\alpha,t}^{W} & = & e^{-\frac{i\textrm{Re }t}{h}x_{0}}\left(a_{\alpha,i\textrm{Im}t}\right)^{W}e^{-\frac{i\textrm{Re }t}{h}x_{0}},\quad\textrm{ and}\label{eq: almost anal conj}\\
a_{\beta,t}^{W} & = & e^{-\frac{i\textrm{Re }t}{h}x_{0}}\left(a_{\beta,i\textrm{Im}t}\right)^{W}e^{-\frac{i\textrm{Re }t}{h}x_{0}}.\label{eq: almost anal conj 2}
\end{eqnarray}
 In similar vein we may define 
\begin{align}
\bar{D}_{t}\coloneqq e^{-\frac{it}{h}x_{0}}\bar{D}e^{\frac{it}{h}x_{0}} & =H_{1,t}^{W}+c_{0}^{W}\left(\bar{\omega}\right)\label{eq: almost analytic extension Dirac}\\
H_{1,t} & =\left(\xi_{0}+t\right)\sigma_{0}+\left(2\bar{\nu}\right)^{\frac{1}{2}}\sum_{j=1}^{m}\mu_{j}^{\frac{1}{2}}\left(x_{j}\sigma_{2j-1}+\xi_{j}\sigma_{2j}\right)\in S_{\textrm{cl}}^{1}\left(\mathbb{R}^{2n}\right),\label{eq: almost analytic extension of H1}
\end{align}
for $t\in\mathbb{R}$, on account of the $\xi_{0}$-independence of
$\bar{\omega}$. An almost analytic continuation of $\bar{D}_{t}$
is easily introduced by simply allowing $t\in\mathbb{C}$ to be complex
in \prettyref{eq: almost analytic extension of H1} above. The resolvent
$\left(\bar{D}_{t}-z\right)^{-1}:L^{2}\left(\mathbb{R}^{n};\mathbb{C}^{2^{m}}\right)\rightarrow L^{2}\left(\mathbb{R}^{n};\mathbb{C}^{2^{m}}\right)$
is well-defined and holomorphic in the region $\textrm{Im}z>\left|\textrm{Im}t\right|$. 

In the lemma below we set $t=i\gamma\left(M,\delta\right)\coloneqq i2Mh^{\delta}\log\frac{1}{h}$,
for $\delta=1-\epsilon\in\left(\frac{1}{2},1\right)$ with $\epsilon$
as in \prettyref{lem: O(h infty) LEMMA} and $M>1$ . We now have
the following.
\begin{lem}
\label{lem: holomorph continuation resolvent}For $h$ sufficiently
small and $\forall\varepsilon_{0}>0$, the resolvent 
\[
\left(\frac{1}{\sqrt{h}}\bar{D}_{i\gamma}-z\right)^{-1}:L^{2}\left(\mathbb{R}^{n};\mathbb{C}^{2^{m}}\right)\rightarrow L^{2}\left(\mathbb{R}^{n};\mathbb{C}^{2^{m}}\right)
\]
extends holomorphically, and is uniformly $O\left(h^{-\frac{1}{2}}\right)$,
in the region $\textrm{Im}z>-Mh^{\delta-\frac{1}{2}}\log\frac{1}{h}$,
$\left|\textrm{Re}z\right|\leq\sqrt{2\nu_{0}}-\varepsilon_{0}$. \end{lem}
\begin{proof}
We begin with the orthogonal Landau decomposition \prettyref{eq: Landau Levels}
\begin{align}
L^{2}\left(\mathbb{R}^{n};\mathbb{C}^{2^{m}}\right) & =L^{2}\left(\mathbb{R}_{x_{0},x''}^{m+1}\right)\otimes\underbrace{\left(\mathbb{C}\left[\psi_{0,0}\right]\oplus\bigoplus_{\varLambda\in\mu.\left(\mathbb{N}_{0}^{m}\setminus0\right)}\left[E_{\varLambda}^{\textrm{even}}\oplus E_{\varLambda}^{\textrm{odd}}\right]\right)}_{=L^{2}\left(\mathbb{R}_{x'}^{m};\mathbb{C}^{2^{m}}\right)}\;\textrm{where}\label{eq: Landau Decomposition}\\
E_{\varLambda}^{\textrm{even}} & \coloneqq\bigoplus_{\begin{subarray}{l}
\tau\in\mathbb{N}_{0}^{m}\setminus0\\
\varLambda=\mu.\tau
\end{subarray}}E_{\tau}^{\textrm{even}}\nonumber \\
E_{\varLambda}^{\textrm{odd}} & \coloneqq\bigoplus_{\begin{subarray}{l}
\tau\in\mathbb{N}_{0}^{m}\setminus0\\
\varLambda=\mu.\tau
\end{subarray}}E_{\tau}^{\textrm{odd}}\nonumber 
\end{align}
according to the eigenspaces of the squared magnetic Dirac operator
$D_{\mathbb{R}^{m}}^{2}$ \prettyref{eq: magnetic Dirac Rm} on $\mathbb{R}^{m}$.
It is clear from \prettyref{eq: almost analytic extension of H1}
that 
\[
H_{1,t}^{W}=\left(\xi_{0}+t\right)\sigma_{0}+\left[\left(2\bar{\nu}\right)^{\frac{1}{2}}\right]^{W}\otimes D_{\mathbb{R}^{m}}
\]
in terms of the above decomposition. Furthermore one has the commutation
relations 
\begin{eqnarray*}
\left[\sigma_{0},D_{\mathbb{R}^{m}}^{2}\right] & = & 0\\
\left[c_{0}^{W}\left(\bar{\omega}\right),D_{\mathbb{R}^{m}}^{2}\right] & = & ihc_{0}^{W}\left(\tilde{\Delta}^{0}\bar{\omega}\right)=0
\end{eqnarray*}
since $\bar{\omega}$ is $\tilde{\Delta}^{0}$-harmonic. The above
and \prettyref{eq: almost analytic extension Dirac} show that the
$\left(\frac{1}{\sqrt{h}}\bar{D}_{t}-z\right)$ preserves the eigenspaces
in the decomposition \prettyref{eq: Landau Decomposition} $\forall t\in\mathbb{C}$.
It hence suffices to consider the restriction of $\left(\frac{1}{\sqrt{h}}\bar{D}_{i\gamma}-z\right)$
to each eigenspace.

Let $E_{0}\coloneqq\mathbb{C}\left[\psi_{0,0}\right],\;E_{\varLambda}\coloneqq E_{\varLambda}^{\textrm{even}}\oplus E_{\varLambda}^{\textrm{odd}}$
and $\mathtt{P}_{0},\mathtt{P}_{\varLambda}$ denote the projection
onto the corresponding summands of \prettyref{eq: Landau Decomposition}.
Define the restrictions 
\begin{eqnarray*}
\varOmega_{0}\coloneqq\mathtt{P}_{0}c_{0}^{W}\left(\bar{\omega}\right)\mathtt{P}_{0} & : & L^{2}\left(\mathbb{R}_{x_{0},x''}^{m+1}\right)\rightarrow L^{2}\left(\mathbb{R}_{x_{0},x''}^{m+1}\right)\\
\varOmega_{\varLambda}\coloneqq\mathtt{P}_{\varLambda}c_{0}^{W}\left(\bar{\omega}\right)\mathtt{P}_{\varLambda} & : & L^{2}\left(\mathbb{R}_{x_{0},x''}^{m+1};E_{\varLambda}^{\textrm{even}}\oplus E_{\varLambda}^{\textrm{odd}}\right)\rightarrow L^{2}\left(\mathbb{R}_{x_{0},x''}^{m+1};E_{\varLambda}^{\textrm{even}}\oplus E_{\varLambda}^{\textrm{odd}}\right),\;\varLambda>0.
\end{eqnarray*}
Now $\bar{\omega}\sim\sum_{j=0}^{\infty}h^{j}\bar{\omega}_{j}\in\mathcal{H}^{\textrm{odd}}S_{\textrm{cl}}^{0}$
with $\xi_{0}$-independent $\bar{\omega}_{0}$ vanishing to second
order along $\Sigma_{0}^{D_{0}}=\Sigma_{0}^{\bar{D}}=\left\{ \xi_{0}=x'=\xi'=0\right\} $.
Hence we may Taylor expand
\[
\bar{\omega}_{0}=\sum_{i\leq j}\left[a_{ij}z_{i}z_{j}+\bar{a}_{ij}\bar{z}_{i}\bar{z}_{j}+b_{ij}\bar{z}_{i}z_{j}+\bar{b}_{ij}z_{i}\bar{z}_{j}\right],
\]
 in terms of the complex coordinates $z_{j}=x_{j}+i\xi_{j}$, $\bar{z}_{j}=x_{j}-i\xi_{j}$
,$1\leq j\leq m$, with $a_{ij},b_{ij}\in S_{\textrm{cl}}^{0}\left(\mathbb{R}^{2n};\mathbb{R}\right)\otimes\Lambda^{\textrm{odd}}W$.
The self-adjoint Clifford-Weyl quantization now yields
\begin{eqnarray*}
c_{0}^{W}\left(\bar{\omega}_{0}\right) & = & \sum_{i\leq j}\left[c_{0}^{W}\left(a_{ij}\right)A_{i}A_{j}+A_{j}^{*}A_{i}^{*}c_{0}^{W}\left(\bar{a}_{ij}\right)+c_{0}^{W}\left(b_{ij}\right)A_{i}^{*}A_{j}+A_{j}^{*}A_{i}c_{0}^{W}\left(\bar{b}_{ij}\right)\right]\\
 &  & \qquad\qquad\qquad\qquad\qquad\qquad\qquad+h\Psi_{\textrm{cl}}^{0}\left(\mathbb{R}^{n};\mathbb{C}^{2^{m}}\right)
\end{eqnarray*}
in terms of the raising and lowering operators in \prettyref{eq:standard commutation}.
Since each lowering operator $A_{j}$ annihilates $\psi_{0,0}$, this
leads to the estimate 
\begin{equation}
\left\Vert \varOmega_{0}\right\Vert =O\left(h\right).\label{eq: bound Omega0}
\end{equation}
Next, on account of \prettyref{eq: omega 0 smaall C1 norm} one may
also expand $\bar{\omega}_{0}=\sum_{j=1}^{m}\left[a_{j}z_{j}+\bar{a}_{j}\bar{z}_{j}\right]$,
with $a_{j}\in S_{\textrm{cl}}^{0}\left(\mathbb{R}^{2n};\mathbb{R}\right)\otimes\Lambda^{\textrm{odd}}W$,
satisfying $\left\Vert a_{j}\right\Vert _{C^{0}}\leq\varepsilon<1$.
On self-adjoint quantization this now gives 
\[
c_{0}^{W}\left(\bar{\omega}_{0}\right)=\sum_{j=1}^{m}\left[c_{0}^{W}\left(a_{j}\right)A_{j}+A_{j}^{*}c_{0}^{W}\left(\bar{a}_{j}\right)\right]+h\Psi_{\textrm{cl}}^{0}\left(\mathbb{R}^{n};\mathbb{C}^{2^{m}}\right)
\]
where
\begin{eqnarray*}
\left\Vert c_{0}^{W}\left(a_{j}\right)\right\Vert _{L^{2}\rightarrow L^{2}},\left\Vert c_{0}^{W}\left(\bar{a}_{j}\right)\right\Vert _{L^{2}\rightarrow L^{2}} & = & \left\Vert a_{j}\right\Vert _{C^{0}}+O\left(h\right)\\
 & \leq & \varepsilon+O\left(h\right).
\end{eqnarray*}
Knowing the action of the lowering and raising operators $A_{j}$,
$A_{j}^{*}$ on each eigenstate \prettyref{eq: Hermite functions}
of $D_{\mathbb{R}^{m}}^{2}$ gives the estimate 
\begin{equation}
\left\Vert \varOmega_{\varLambda}\right\Vert \leq\varepsilon\sqrt{\varLambda h}+O\left(h\right)\label{eq: bound Omega N}
\end{equation}
with the $O\left(h\right)$ term above being uniform in $\varLambda$. 

Next we compute the restriction of $\left(\frac{1}{\sqrt{h}}\bar{D}_{i\gamma}-z\right)$
to the $E_{0}$ eigenspace in \prettyref{eq: Landau Decomposition}
using \prettyref{eq:odd clifford representation} to be 
\begin{equation}
D_{i\gamma,0}\left(z\right)\coloneqq\mathtt{P}_{0}\left(\frac{1}{\sqrt{h}}\bar{D}_{i\gamma}-z\right)\mathtt{P}_{0}=\frac{1}{\sqrt{h}}\left[-\xi_{0}-i\gamma-z\sqrt{h}+\varOmega_{0}\right].\label{eq: restriction to E0}
\end{equation}
The above is again understood as a closed unbounded operator on $L^{2}\left(\mathbb{R}_{x_{0},x''}^{m+1}\right)$
with domain $H^{1}\left(\mathbb{R}_{x_{0}}\right)\otimes L^{2}\left(\mathbb{R}_{x''}^{m}\right)$.
Set $R_{i\gamma,0}\left(z\right)=\left[r_{i\gamma,0}\left(z\right)\right]^{W}$,
with 
\[
r_{i\gamma,0}\left(z\right)=\frac{\sqrt{h}}{-\xi_{0}-i\gamma-z\sqrt{h}},
\]
which is well defined for $\textrm{Im}z>-\frac{\gamma}{2\sqrt{h}}=-Mh^{\delta-\frac{1}{2}}\log\frac{1}{h}$,
and compute 
\begin{eqnarray*}
R_{i\gamma,0}\left(z\right)D_{i\gamma,0}\left(z\right) & = & I+O\left(h^{1-\delta}\right)\\
D_{i\gamma,0}\left(z\right)R_{i\gamma,0}\left(z\right) & = & I+O\left(h^{1-\delta}\right)
\end{eqnarray*}
using \prettyref{eq: bound Omega0}. This shows that the inverse $D_{i\gamma,0}\left(z\right)^{-1}$
exists and is $O\left(R_{i\gamma,0}\left(z\right)\right)=O\left(h^{\frac{1}{2}-\delta}\right)$. 

Next, we compute the restriction of $\left(\frac{1}{\sqrt{h}}\bar{D}_{i\gamma}-z\right)$
to the $E_{\varLambda}$, $\varLambda>0$, eigenspace in \prettyref{eq: Landau Decomposition}.
Using \prettyref{eq: Dirac operator 2 by 2 block}, \prettyref{eq: clifford mult. is diagonal on 2 by 2 block}
this has the form 
\begin{eqnarray*}
D_{i\gamma,\varLambda}\left(z\right) & \coloneqq & \mathtt{P}_{\varLambda}\left(\frac{1}{\sqrt{h}}\bar{D}_{i\gamma}-z\right)\mathtt{P}_{\varLambda}\\
 & = & \frac{1}{\sqrt{h}}\begin{bmatrix}-\xi_{0}-i\gamma-z\sqrt{h} & \left(\sqrt{2\bar{\nu}\varLambda h}\right)^{W}\\
\left(\sqrt{2\bar{\nu}\varLambda h}\right)^{W} & \xi_{0}+i\gamma-z\sqrt{h}
\end{bmatrix}+\frac{1}{\sqrt{h}}\varOmega_{\varLambda}
\end{eqnarray*}
with respect to the $\mathbb{Z}_{2}$- grading $E_{\varLambda}=E_{\varLambda}^{\textrm{even}}\oplus E_{\varLambda}^{\textrm{odd}}$.
Here we leave the identification $\mathtt{i}_{\tau}$ in \prettyref{eq: Dirac operator 2 by 2 block}
between the odd and even parts as being understood. Set $R_{i\gamma,\varLambda}\left(z\right)=\left[r_{i\gamma,\varLambda}\left(z\right)\right]^{W}$
\[
r_{i\gamma,\varLambda}\left(z\right)\coloneqq\frac{\sqrt{h}\begin{bmatrix}-\xi_{0}-i\gamma-z\sqrt{h} & \left(\sqrt{2\bar{\nu}\varLambda h}\right)\\
\left(\sqrt{2\bar{\nu}\varLambda h}\right) & \xi_{0}+i\gamma-z\sqrt{h}
\end{bmatrix}}{z^{2}h-\left(\xi_{0}+i\gamma\right)^{2}-2\bar{\nu}\varLambda h}
\]
which is well defined for $\left|\textrm{Re}z\right|\leq\sqrt{2\nu_{0}}-\varepsilon_{0}<\inf_{\mathbb{R}^{n}}\sqrt{2\bar{\nu}\varLambda}$,
and $h$ sufficiently small. We now compute 
\begin{eqnarray*}
\left\Vert R_{i\gamma,\varLambda}\left(z\right)D_{i\gamma,\varLambda}\left(z\right)-I\right\Vert  & \leq & C\varepsilon+O\left(h\right)\\
\left\Vert D_{i\gamma,\varLambda}\left(z\right)R_{i\gamma,\varLambda}\left(z\right)-I\right\Vert  & \leq & C\varepsilon+O\left(h\right)
\end{eqnarray*}
using \prettyref{eq: bound Omega N} with the constants above being
uniform in $\varLambda$. Choosing $\varepsilon$ sufficiently small
in \prettyref{eq: omega 0 smaall C1 norm} shows that the inverse
$D_{i\gamma,\varLambda}\left(z\right)^{-1}$ exists and is $O\left(R_{i\gamma,\varLambda}\left(z\right)\right)=O\left(h^{-\frac{1}{2}}\right)$
uniformly. 
\end{proof}
We now finally finish the proof of \prettyref{lem: O(h infty) LEMMA}.
\begin{proof}[Proof of \prettyref{lem: O(h infty) LEMMA}]
 As noted in the beginning of the section, on account of \prettyref{eq: microlocal estimate},
\prettyref{eq:Helffer Sjostrand formula} and the reductions \prettyref{prop:Reduction to R^n}
and \prettyref{prop: Reduction to normal form operator}, it suffices
to show $\mathcal{T}_{\alpha\beta}^{\vartheta}\left(\bar{D}\right)=O\left(h^{\infty}\right)$.
We now define the trace 
\begin{equation}
\tau_{\alpha\beta,t}\left(z\right)\coloneqq\textrm{tr}\left[a_{\alpha,t}^{W}\left(\frac{1}{\sqrt{h}}\bar{D}_{t}-z\right)^{-1}a_{\beta,t}^{W}\right],\quad\textrm{Im}z>\left|\textrm{Im}t\right|,\label{eq: trace cutoff resolvent}
\end{equation}
 in terms of the almost analytic continuations. We clearly have 
\begin{eqnarray*}
\tau_{\alpha\beta,t}\left(z\right) & = & O\left(h^{-n}\left|\textrm{Im}z\right|^{-1}\right)\\
\frac{\partial}{\partial\bar{t}}\tau_{\alpha\beta,t}\left(z\right) & = & O\left(h^{-n}\left|\textrm{Im}t\right|^{\infty}\left|\textrm{Im}z\right|^{-2}\right).
\end{eqnarray*}
Furthermore, by \prettyref{eq: almost anal conj}, \prettyref{eq: almost anal conj 2}
and \prettyref{eq: almost analytic extension Dirac} $\tau_{\alpha\beta,t}\left(z\right)$
only depends on $\textrm{Re}t$ and we have 
\begin{equation}
\tau_{\alpha\beta,i\textrm{ Im}t}\left(z\right)=\tau_{\alpha\beta,0}\left(z\right)+O\left(h^{-n}\left|\textrm{Im}t\right|^{\infty}\left|\textrm{Im}z\right|^{-2}\right).\label{eq: trace asymp unchanged almost anal. ext}
\end{equation}
As before, we again introduce $\psi\in C^{\infty}\left(\mathbb{R};\left[0,1\right]\right)$
such that $\psi\left(x\right)=\begin{cases}
1; & x\leq1\\
0; & x\geq2
\end{cases}$ and set $\psi_{M}\left(z\right)=\psi\left(\frac{\textrm{Im}z}{M\sqrt{h}\log\frac{1}{h}}\right)$.
The estimates \prettyref{eq: Paley Wiener estimate}, \prettyref{eq: almost analytic estimate}
along with the observation $\psi_{M}\left|\textrm{Im}z\right|^{N}=O\left(\left(M\sqrt{h}\log\frac{1}{h}\right)^{N}\right)$
now give 
\begin{eqnarray*}
\mathcal{T}_{\alpha\beta}^{\vartheta}\left(\bar{D}\right) & = & \frac{1}{\pi}\int_{\mathbb{C}}\bar{\partial}\left(\psi_{M}\tilde{f}\right)\check{\vartheta}\left(\frac{\lambda-z}{\sqrt{h}}\right)\tau_{\alpha\beta,0}\left(z\right)dzd\bar{z}\\
 & = & O\left(h^{\infty}\right)+\\
 &  & \frac{1}{\pi}\int_{\left\{ M\sqrt{h}\log\frac{1}{h}\leq\textrm{Im}z\leq2M\sqrt{h}\log\frac{1}{h}\right\} }\bar{\partial}\left(\psi_{M}\tilde{f}\right)\check{\vartheta}\left(\frac{\lambda-z}{\sqrt{h}}\right)\tau_{\alpha\beta,0}\left(z\right)dzd\bar{z}.
\end{eqnarray*}
Using \prettyref{eq: trace asymp unchanged almost anal. ext} and
$\gamma=2Mh^{\delta}\log\frac{1}{h}$, $\delta\in\left(\frac{1}{2},1\right)$,
the above now equals 
\begin{eqnarray*}
\mathcal{T}_{\alpha\beta}^{\vartheta}\left(\bar{D}\right) & = & O\left(h^{\infty}\right)+\\
 &  & \frac{1}{\pi}\int_{\left\{ M\sqrt{h}\log\frac{1}{h}\leq\textrm{Im}z\leq2M\sqrt{h}\log\frac{1}{h}\right\} }\bar{\partial}\left(\psi_{M}\tilde{f}\right)\check{\vartheta}\left(\frac{\lambda-z}{\sqrt{h}}\right)\tau_{\alpha\beta,i\gamma}\left(z\right)dzd\bar{z}.
\end{eqnarray*}
Since the resolvent $\left(\frac{1}{\sqrt{h}}\bar{D}_{i\gamma}-z\right)^{-1}$,
and hence the trace $\tau_{\alpha\beta,i\gamma}\left(z\right)$, extends
holomorphically to $\textrm{Im}z>-Mh^{\delta-\frac{1}{2}}\log\frac{1}{h},\left|\textrm{Re}z\right|\leq\sqrt{2\nu_{0}}-\varepsilon_{0}$
by \prettyref{lem: holomorph continuation resolvent} we may replace
the integral in the last line above
\begin{eqnarray*}
\mathcal{T}_{\alpha\beta}^{\vartheta}\left(\bar{D}\right) & = & O\left(h^{\infty}\right)+\\
 &  & \frac{1}{\pi}\int_{\left\{ -\frac{1}{2}Mh^{\delta-\frac{1}{2}}\log\frac{1}{h}\leq\textrm{Im}z\leq-\frac{1}{4}Mh^{\delta-\frac{1}{2}}\log\frac{1}{h}\right\} }\bar{\partial}\left(\psi_{M}\tilde{f}\right)\check{\vartheta}\left(\frac{\lambda-z}{\sqrt{h}}\right)\tau_{\alpha\beta,i\gamma}\left(z\right)dzd\bar{z}\\
 & = & O\left(h^{\infty}\right)+\\
 &  & O\left[\int_{\left\{ -\frac{1}{2}Mh^{\delta-\frac{1}{2}}\log\frac{1}{h}\leq\textrm{Im}z\leq-\frac{1}{4}Mh^{\delta-\frac{1}{2}}\log\frac{1}{h}\right\} }\frac{h^{-n-\frac{1}{2}}}{\sqrt{h}\log\frac{1}{h}}e^{\frac{S_{\alpha\beta}'\left(\textrm{Im}z\right)}{h^{\frac{1}{2}-\epsilon}}}dzd\bar{z}\right]\\
 & = & O\left[h^{\frac{M}{4}\left(S_{\alpha\beta}'\right)-n-\frac{1}{2}}\right]
\end{eqnarray*}
 using \prettyref{eq: Paley Wiener estimate} and $O\left(h^{-\frac{1}{2}}\right)$
estimate on the resolvent $\left(\frac{1}{\sqrt{h}}\bar{D}_{i\gamma}-z\right)^{-1}$.
Choosing $M$ sufficiently large now gives the result.
\end{proof}

\section{\label{sec:Local trace expansion}Local trace expansion}

In this section we prove \prettyref{lem: Easy trace expansion lemma}.
This is a relatively classical trace expansion. A parametrix construction
for the operator $e^{\frac{it}{h}D_{h}^{2}}$ may potentially be employed
in its proof since the principal symbol of $D_{h}^{2}$ is Morse-Bott
critical as in \cite{Brummelhuis-Paul-Uribe}. However \prettyref{lem: Easy trace expansion lemma}
would require an understanding of the large time behaviour of parametrix
left open in \cite{Brummelhuis-Paul-Uribe} (see \cite{Camus,Khuat-Duy}).
Here we prove the expansion using the alternate methods of local index
theory. The expansion is analogous to the heat trace expansions arising
in the analysis of the Bergman kernel \cite{Bismut,Ma-Marinescu}.
Here we adopt a modification of the approach in \cite{Ma-Marinescu}
Ch. 1, 4. 

First, fix a point $p\in X$. On account of \prettyref{def: Diagonalizability assumption}
there is on orthonormal basis $e_{0,p}=R_{p}$,$e_{j,p},\,e_{j+m,p}$,
$j=1,\ldots,m$ of $T_{p}X$ consisting of eigenvectors of $\mathfrak{J}_{p}$
with eigenvalues $0,\pm\lambda_{j,p}\left(\coloneqq\pm i\mu_{j}\nu\left(p\right)\right)$,
$j=1,\ldots,m$, such that 
\begin{equation}
da\left(p\right)=\sum_{j=1}^{m}\lambda_{j}\left(p\right)e_{j,p}^{*}\wedge e_{j+m,p}^{*}.\label{eq: da diagonal form}
\end{equation}
Using the parallel transport from this basis fix a geodesic coordinate
system $\left(x_{0},\ldots,x_{2m}\right)$ on an open neighborhood
of $p\in\Omega$. Let $e_{j}=w_{j}^{k}\partial_{x_{k}}$, $0\leq j\leq2m$,
be the local orthonormal frame of $TX$ obtained by parallel transport
of $e_{j,p}=\left.\partial_{x_{j}}\right|_{p}$,$0\leq j\leq2m$,
along geodesics. Hence we again have $w_{j}^{k}g_{kl}w_{r}^{l}=\delta_{jr}$;
$\left.w_{j}^{k}\right|_{p}=\delta_{j}^{k}$ with $g_{kl}$ being
the components of the metric in these coordinates. Choose an orthonormal
basis $\left\{ s_{j,p}\right\} _{j=1}^{2^{m}}$for $S_{p}$ in which
Clifford multiplication
\begin{equation}
\left.c\left(e_{j}\right)\right|_{p}=\gamma_{j}\label{eq: Clifford multiplication standard}
\end{equation}
is standard. Choose an orthonormal basis $\mathtt{l}_{p}$ for $L_{p}$.
Parallel transport the bases $\left\{ s_{j,p}\right\} _{j=1}^{2^{m}}$,
$\mathtt{l}_{p}$ along geodesics using the spin connection $\nabla^{S}$
and unitary family of connections $\nabla^{h}=A_{0}+\frac{i}{h}a$
to obtain trivializations $\left\{ s_{j}\right\} _{j=1}^{2^{m}}$,
$\mathtt{l}$ of $S$, $L$ on $\Omega$. Since Clifford multiplication
is parallel, the relation \prettyref{eq: Clifford multiplication standard}
now holds on $\Omega$. The connection $\nabla^{S\otimes L}=\nabla^{S}\otimes1+1\otimes\nabla^{h}$
can be expressed in this frame and these coordinates as
\begin{equation}
\nabla^{S\otimes L}=d+A_{j}^{h}dx^{j}+\Gamma_{j}dx^{j},
\end{equation}
 where each $A_{j}^{h}$ is a Christoffel symbol of $\nabla^{h}$
and each $\Gamma_{j}$ is a Christoffel symbol of the spin connection
$\nabla^{S}$. Since the section $l$ is obtained via parallel transport
along geodesics, the connection coefficient $A_{j}^{h}$ maybe written
in terms of the curvature $F_{jk}^{h}dx^{j}\wedge dx^{k}$ of $\nabla^{h}$
via
\begin{equation}
A_{j}^{h}(x)=\int_{0}^{1}d\rho\left(\rho x^{k}F_{jk}^{h}\left(\rho x\right)\right).
\end{equation}
The dependence of the curvature coefficients $F_{jk}^{h}$ on the
parameter $h$ is seen to be linear in $\frac{1}{h}$ via 
\begin{equation}
F_{jk}^{h}=F_{jk}^{0}+\frac{i}{h}\left(da\right){}_{jk}\label{eq:curvature linear in 1/h}
\end{equation}
 despite the fact that they are expressed in the $h$ dependent frame
$\mathtt{l}$. This is because a gauge transformation from an $h$
independent frame into $\mathtt{l}$ changes the curvature coefficient
by conjugation. Since $L$ is a line bundle this is conjugation by
a function and hence does not change the coefficient. Furthermore,
the coefficients in the Taylor expansion of \prettyref{eq:curvature linear in 1/h}
at $0$ maybe expressed in terms of the covariant derivatives $\left(\nabla^{A_{0}}\right)^{l}F_{jk}^{0},$
$\left(\nabla^{A_{0}}\right)^{l}\left(da\right){}_{jk}$ evaluated
at $p$. Next, using the Taylor expansion
\begin{equation}
\left(da\right){}_{jk}=\left(da\right){}_{jk}\left(0\right)+x^{l}a_{jkl},\label{eq: Taylor expansion da}
\end{equation}
we see that the connection $\nabla^{S\otimes L}$ has the form
\begin{equation}
\nabla^{S\otimes L}=d+\left[\frac{i}{h}\left(\frac{x^{k}}{2}\left(da\right){}_{jk}\left(0\right)+x^{k}x^{l}A_{jkl}\right)+x^{k}A_{jk}^{0}+\Gamma_{j}\right]dx^{j}\label{eq: connection in geodesic}
\end{equation}
where 
\begin{eqnarray*}
A_{jk}^{0} & = & \int_{0}^{1}d\rho\left(\rho F_{jk}^{0}\left(\rho x\right)\right)\\
A_{jkl} & = & \int_{0}^{1}d\rho\left(\rho a_{jkl}\left(\rho x\right)\right)
\end{eqnarray*}
and $\Gamma_{j}$ are all independent of $h$. Finally from \prettyref{eq: Clifford multiplication standard}
and \prettyref{eq: connection in geodesic} may write down the expression
for the Dirac operator \prettyref{eq:Semiclassical Magnetic Dirac}
also given as $D=hc\circ\left(\nabla^{S\otimes L}\right)$ in terms
of the chosen frame and coordinates to be 
\begin{align}
D & =\gamma^{r}w_{r}^{j}\left[h\partial_{x_{j}}+i\frac{x^{k}}{2}\left(da\right){}_{jk}\left(0\right)+ix^{k}x^{l}A_{jkl}+h\left(x^{k}A_{jk}^{0}+\Gamma_{j}\right)\right]\label{eq: Dirac operator geodesic coordinates}\\
 & =\gamma^{r}\left[w_{r}^{j}h\partial_{x_{j}}+iw_{r}^{j}\frac{x^{k}}{2}\left(da\right){}_{jk}\left(0\right)+\frac{1}{2}hg^{-\frac{1}{2}}\partial_{x_{j}}\left(g^{\frac{1}{2}}w_{r}^{j}\right)\right]+\\
 & \gamma^{r}\left[iw_{r}^{j}x^{k}x^{l}A_{jkl}+hw_{r}^{j}\left(x^{k}A_{jk}^{0}+\Gamma_{j}\right)-\frac{1}{2}hg^{-\frac{1}{2}}\partial_{x_{j}}\left(g^{\frac{1}{2}}w_{r}^{j}\right)\right]\in\Psi_{\textrm{cl}}^{1}\left(\Omega_{s}^{0};\mathbb{C}^{2^{m}}\right)\nonumber 
\end{align}
In the second expression above both square brackets are self-adjoint
with respect to the Riemannian density $e^{1}\wedge\ldots\wedge e^{n}=\sqrt{g}dx\coloneqq\sqrt{g}dx^{1}\wedge\ldots\wedge dx^{n}$
with $g=\det\left(g_{ij}\right)$. Again one may obtain an expression
self-adjoint with respect to the Euclidean density $dx$ in the framing
$g^{\frac{1}{4}}u_{j}\otimes\mathtt{l},1\leq j\leq2^{m}$, with the
result being an addition of the term $h\gamma^{j}w_{j}^{k}g^{-\frac{1}{4}}\left(\partial_{x_{k}}g^{\frac{1}{4}}\right)$. 

Let $i_{g}$ be the injectivity radius of $g^{TX}$ . Define the cutoff
$\chi\in C_{c}^{\infty}\left(-1,1\right)$ such that $\chi=1$ on
$\left(-\frac{1}{2},\frac{1}{2}\right)$. We now modify the functions
$w_{j}^{k}$, outside the ball $B_{i_{g}/2}\left(p\right)$, such
that $w_{j}^{k}=\delta_{j}^{k}$ (and hence $g_{jk}=\delta_{jk}$)
are standard outside the ball $B_{i_{g}}\left(p\right)$ of radius
$i_{g}$ centered at $p$. This again gives 
\begin{align}
\mathbb{D} & =\gamma^{r}\left[w_{r}^{j}h\partial_{x_{j}}+iw_{r}^{j}\frac{x^{k}}{2}\left(da\right){}_{jk}\left(0\right)+\frac{1}{2}hg^{-\frac{1}{2}}\partial_{x_{j}}\left(g^{\frac{1}{2}}w_{r}^{j}\right)\right]+\label{eq: Localized Dirac}\\
 & \chi\left(\left|x\right|/i_{g}\right)\gamma^{r}\left[iw_{r}^{j}x^{k}x^{l}A_{jkl}+hw_{r}^{j}\left(x^{k}A_{jk}^{0}+\Gamma_{j}\right)-\frac{1}{2}hg^{-\frac{1}{2}}\partial_{x_{j}}\left(g^{\frac{1}{2}}w_{r}^{j}\right)\right]\nonumber \\
 & \qquad\qquad\qquad\qquad\qquad\qquad\qquad\qquad\qquad\in\Psi_{\textrm{cl}}^{1}\left(\mathbb{R}^{n};\mathbb{C}^{2^{m}}\right)\nonumber 
\end{align}
as a well defined operator on $\mathbb{R}^{n}$ formally self adjoint
with respect to $\sqrt{g}dx$. Again $\mathbb{D}+i$ being elliptic
in the class $S^{0}\left(m\right)$ for the order function 
\[
m=\sqrt{1+g^{jl}\left(\xi_{j}+\frac{x^{k}}{2}\left(da\right)_{jk}\left(0\right)\right)\left(\xi_{l}+\frac{x^{r}}{2}\left(da\right)_{lr}\left(0\right)\right)},
\]
the operator $\mathbb{D}$ is essentially self adjoint.
\begin{prop}
\label{prop: local trace expansion}There exist tempered distributions
$u_{j}\in\mathcal{S}'\left(\mathbb{R}_{s}\right)$, $j=0,1,2,\ldots$,
such that one has a trace expansion
\begin{equation}
\textrm{tr }\phi\left(\frac{D}{\sqrt{h}}\right)=h^{-n/2}\left(\sum_{j=0}^{N}u_{j}\left(\phi\right)h^{j/2}\right)+h^{\left(N+1-n\right)/2}O\left(\sum_{k=0}^{n+1}\left\Vert \left\langle \xi\right\rangle ^{N}\hat{\phi}^{\left(k\right)}\right\Vert _{L^{1}}\right)\label{eq: local trace expansion}
\end{equation}
for each $N\in\mathbb{N}$, $\phi\in\mathcal{S}\left(\mathbb{R}_{s}\right)$. \end{prop}
\begin{proof}
We begin by writing $\phi=\phi_{0}+\phi_{1}$, with 
\begin{eqnarray*}
\phi_{0}\left(s\right) & = & \frac{1}{2\pi}\int_{\mathbb{R}}e^{i\xi s}\hat{\phi}\left(\xi\right)\chi\left(\frac{2\xi\sqrt{h}}{i_{g}}\right)d\xi\\
\phi_{1}\left(s\right) & = & \frac{1}{2\pi}\int_{\mathbb{R}}e^{i\xi s}\hat{\phi}\left(\xi\right)\left[1-\chi\left(\frac{2\xi\sqrt{h}}{i_{g}}\right)\right]d\xi
\end{eqnarray*}
via Fourier inversion. 

First considering $\phi_{1}$, integration by parts gives the estimate
\[
\left|s^{n+1}\phi_{1}\left(s\right)\right|\leq C_{N}h^{\frac{N-1}{2}}\left(\sum_{k=0}^{n+1}\left\Vert \xi^{N}\hat{\phi}^{\left(k\right)}\right\Vert _{L^{1}}\right),
\]
$\forall N\in\mathbb{N}$. Hence,
\[
\left\Vert D^{n+1-a}\phi_{1}\left(\frac{D}{\sqrt{h}}\right)D^{a}\right\Vert _{L^{2}\rightarrow L^{2}}=C_{N}h^{\frac{n+N}{2}}\left(\sum_{k=0}^{n+1}\left\Vert \xi^{N}\hat{\phi}^{\left(k\right)}\right\Vert _{L^{1}}\right),
\]
$\forall N\in\mathbb{N},\:\forall a=0,\ldots,n+1$. Semi-classical
elliptic estimate and Sobolev's inequality now give the estimate 
\begin{equation}
\left|\phi_{1}\left(\frac{D}{\sqrt{h}}\right)\right|_{C^{0}\left(X\times X\right)}\leq C_{N}h^{\frac{n+N}{2}}\left(\sum_{k=0}^{n+1}\left\Vert \xi^{N}\hat{\phi}^{\left(k\right)}\right\Vert _{L^{1}}\right)\label{eq: Schw ker localizes}
\end{equation}
$\forall N\in\mathbb{N}$, on the Schwartz kernel.

Next, considering $\phi_{0}$, we first use the change of variables
$\alpha=\xi\sqrt{h}$ to write 
\[
\phi_{0}\left(\frac{D}{\sqrt{h}}\right)=\frac{1}{2\pi\sqrt{h}}\int_{\mathbb{R}}e^{i\alpha\left(D_{A_{0}}+ih^{-1}c\left(a\right)\right)}\hat{\phi}\left(\frac{\alpha}{\sqrt{h}}\right)\chi\left(\frac{2\alpha}{i_{g}}\right)d\alpha.
\]
Now since $D=\mathbb{D}$ on $B_{i_{g}/2}\left(p\right)$, we may
use the finite propagation speed of the wave operators $e^{i\alpha h^{-1}D}$,
$e^{i\alpha h^{-1}\mathbb{D}}$ (cf. D.2.1 in \cite{Ma-Marinescu})
to conclude 
\begin{equation}
\phi_{0}\left(\frac{D}{\sqrt{h}}\right)\left(p,\cdot\right)=\phi_{0}\left(\frac{\mathbb{D}}{\sqrt{h}}\right)\left(0,\cdot\right).\label{eq: finite propagation}
\end{equation}
The right hand side above is defined using functional calculus of
self-adjoint operators, with standard local elliptic regularity arguments
implying the smoothness of its Schwartz kernel. By virtue of \prettyref{eq: Schw ker localizes},
a similar estimate for $\phi_{1}\left(\frac{\mathbb{D}}{\sqrt{h}}\right)$,
and \prettyref{eq: finite propagation} it now suffices to consider
$\phi\left(\frac{\mathbb{D}}{\sqrt{h}}\right)$.

We now introduce the rescaling operator $\mathscr{R}:C^{\infty}\left(\mathbb{R}^{n};\mathbb{C}^{2^{m}}\right)\rightarrow C^{\infty}\left(\mathbb{R}^{n};\mathbb{C}^{2^{m}}\right)$;
$\left(\mathscr{R}s\right)\left(x\right)\coloneqq s\left(\frac{x}{\sqrt{h}}\right)$.
Conjugation by $\mathscr{R}$ amounts to the rescaling of coordinates
$x\rightarrow x\sqrt{h}$. A Taylor expansion in \prettyref{eq: Localized Dirac}
now gives the existence of classical ($h$-independent) self-adjoint,
first-order differential operators $\mathrm{\mathtt{D}}_{j}=a_{j}^{k}\left(x\right)\partial_{x_{k}}+b_{j}\left(x\right)$,
$j=0,1\ldots$, with polynomial coefficients (of degree at most $j+1$)
as well as $h$-dependent self-adjoint, first-order differential operators
$\mathrm{E}_{j}=\sum_{\left|\alpha\right|=N+1}x^{\alpha}\left[c_{j,\alpha}^{k}\left(x;h\right)\partial_{x_{k}}+d_{j,\alpha}\left(x;h\right)\right]$,
$j=0,1\ldots$, with uniformly $C^{\infty}$ bounded coefficients
$c_{j,\alpha}^{k},\,d_{j,\alpha}$ such that 
\begin{eqnarray}
\mathscr{R}\mathbb{D}\mathscr{R}^{-1} & = & \sqrt{h}\mathrm{\mathtt{D}}\quad\textrm{ with}\label{eq: rescaled Dirac}\\
\mathrm{\mathtt{D}} & = & \left(\sum_{j=0}^{N}h^{j/2}\mathrm{\mathtt{D}}_{j}\right)+h^{\left(N+1\right)/2}\mathrm{E}_{N+1},\;\forall N.\label{eq: Taylor expansion Dirac}
\end{eqnarray}
The coefficients of the polynomials $a_{j}^{k}\left(x\right),\,b_{j}\left(x\right)$
again involve the covariant derivatives of the curvatures $F^{TX},F^{A_{0}}$
and $da$ evaluated at $p$. Furthermore, the leading term in \prettyref{eq: Taylor expansion Dirac}
is easily computed 
\begin{align}
\mathrm{\mathtt{D}}_{0} & =\gamma^{j}\left[\partial_{x_{j}}+i\frac{x^{k}}{2}\left(da\right){}_{jk}\left(0\right)\right]\label{eq: leading term rescaled Dirac}\\
 & =\gamma^{0}\partial_{x_{0}}+\underbrace{\gamma^{j}\left[\partial_{x_{j}}+\frac{i\lambda_{j}\left(p\right)}{2}x_{j+m}\right]+\gamma^{j+m}\left[\partial_{x_{j+m}}-\frac{i\lambda_{j}\left(p\right)}{2}x_{j}\right]}_{\coloneqq\mathrm{\mathtt{D}}_{00}}\label{eq: leading term rescaled Dirac-1}
\end{align}
using \prettyref{eq: da diagonal form}, \prettyref{eq: Taylor expansion da}.
It is now clear from \prettyref{eq: rescaled Dirac} that 
\begin{equation}
\phi\left(\frac{\mathbb{D}}{\sqrt{h}}\right)\left(x,x'\right)=h^{-n/2}\phi\left(\mathrm{\mathtt{D}}\right)\left(\frac{x}{\sqrt{h}},\frac{x'}{\sqrt{h}}\right).\label{eq: rescaling Schw kernel}
\end{equation}
Next, let $I_{j}=\left\{ k=\left(k_{0},k_{1},\ldots\right)|k_{\alpha}\in\mathbb{N},\:\sum k_{\alpha}=j\right\} $
denote the set of partitions of the integer $j$ and set 
\begin{equation}
\mathtt{C}_{j}^{z}=\sum_{k\in I_{j}}\left(z-\mathrm{\mathtt{D}}_{0}\right)^{-1}\left[\Pi_{\alpha}\mathrm{\mathtt{D}}_{k_{\alpha}}\left(z-\mathrm{\mathtt{D}}_{0}\right)^{-1}\right].\label{eq: jth term kernel expansion}
\end{equation}
Local elliptic regularity estimates again give $\left(z-\mathrm{\mathtt{D}}\right)^{-1}=O_{L_{\textrm{loc}}^{2}\rightarrow L_{\textrm{loc}}^{2}}\left(\left|\textrm{Im}z\right|^{-1}\right)$
and $\mathtt{C}_{j}^{z}=O_{L_{\textrm{loc}}^{2}\rightarrow L_{\textrm{loc}}^{2}}\left(\left|\textrm{Im}z\right|^{-j-1}\right)$,
$j=0,1,\ldots$. A straightforward computation using \prettyref{eq: Taylor expansion Dirac}
then yields 
\begin{equation}
\left(z-\mathrm{\mathtt{D}}\right)^{-1}-\left(\sum_{j=0}^{N}h^{j/2}\mathtt{C}_{j}^{z}\right)=O_{L_{\textrm{loc}}^{2}\rightarrow L_{\textrm{loc}}^{2}}\left(\left(\left|\textrm{Im}z\right|^{-1}h^{\frac{1}{2}}\right)^{N+1}\right).\label{eq: resolvent expansion}
\end{equation}
A similar expansion as \prettyref{eq: Taylor expansion Dirac} for
the operator $\left(1+\mathrm{\mathtt{D}}^{2}\right)^{\left(n+1\right)/2}\left(z-\mathrm{\mathtt{D}}\right)$
also gives the bounds 
\begin{equation}
\left(1+\mathrm{\mathtt{D}}^{2}\right)^{-\left(n+1\right)/2}\left(z-\mathrm{\mathtt{D}}\right)^{-1}-\left(\sum_{j=0}^{N}h^{j/2}\mathtt{C}_{j,n+1}^{z}\right)=O_{H_{\textrm{loc}}^{s}\rightarrow H_{\textrm{loc}}^{s+n+1}}\left(\left(\left|\textrm{Im}z\right|^{-1}h^{\frac{1}{2}}\right)^{N+1}\right)\label{eq: rescaled resolvent expansion}
\end{equation}
$\forall s\in\mathbb{R}$, for classical ($h$-independent) Sobolev
spaces $H_{\textrm{loc}}^{s}$. Here each $\mathtt{C}_{j,n+1}^{z}=O_{H_{\textrm{loc}}^{s}\rightarrow H_{\textrm{loc}}^{s+n+1}}\left(\left|\textrm{Im}z\right|^{-j-1}\right)$
with the leading term 
\[
\mathtt{C}_{0,n+1}^{z}=\left(1+\mathrm{\mathtt{D}}_{0}^{2}\right)^{-\left(n+1\right)/2}\left(z-\mathrm{\mathtt{D}}_{0}\right)^{-1}.
\]
Finally, plugging the expansion \prettyref{eq: rescaled resolvent expansion}
into the Helffer-Sjostrand formula 
\[
\phi\left(\mathrm{\mathtt{D}}\right)=-\frac{1}{\pi}\int_{\mathbb{C}}\bar{\partial}\tilde{\varrho}\left(z\right)\left(1+\mathrm{\mathtt{D}}^{2}\right)^{-\left(n+1\right)/2}\left(z-\mathrm{\mathtt{D}}\right)^{-1}dzd\bar{z},
\]
with $\varrho\left(x\right)\coloneqq\left\langle x\right\rangle ^{n+1}\phi\left(x\right)$,
gives 
\begin{equation}
\phi\left(\mathrm{\mathtt{D}}\right)\left(0,0\right)=\left(\sum_{j=0}^{N}h^{j/2}U_{j,p}\left(\phi\right)\right)+h^{\left(N+1\right)/2}O\left(\sum_{k=0}^{n+1}\left\Vert \left\langle \xi\right\rangle ^{N}\hat{\phi}^{\left(k\right)}\right\Vert _{L^{1}}\right)\label{eq: Diagonal kernel expansion}
\end{equation}
using Sobolev's inequality. Here each
\begin{equation}
U_{j,p}\left(\phi\right)=-\frac{1}{\pi}\int_{\mathbb{C}}\bar{\partial}\tilde{\varrho}\left(z\right)\mathtt{C}_{j,n+1}^{z}\left(0,0\right)dzd\bar{z}\in\textrm{End}S_{p}^{TX}\label{eq: jth coefficient functional kernal}
\end{equation}
defines a smooth family (in $p\in X$) of distributions $U_{j}$ and
the remainder term in \prettyref{eq: Diagonal kernel expansion} comes
from the estimate $\bar{\partial}\tilde{\varrho}=O\left(\left|\textrm{Im}z\right|^{N+1}\sum_{k=0}^{n+1}\left\Vert \left\langle \xi\right\rangle ^{N}\hat{\phi}^{\left(k\right)}\right\Vert _{L^{1}}\right)$
on the almost analytic continuation (cf. \cite{Zworski} Sec. 3.1).
Integrating the trace of \prettyref{eq: Diagonal kernel expansion}
over $X$ and using \prettyref{eq: rescaling Schw kernel} gives \prettyref{eq: local trace expansion}.
\end{proof}
Next we would like to understand the structure of the distributions
$u_{j}$ appearing in \prettyref{eq: local trace expansion}. Clearly,
\begin{eqnarray}
u_{j} & = & \int_{X}u_{j,p}\quad\textrm{ with }\nonumber \\
u_{j,p} & \coloneqq & \textrm{tr }U_{j,p}\in C^{\infty}\left(X;\mathcal{S}'\left(\mathbb{R}_{s}\right)\right)\label{eq: pointwise trace distribution}
\end{eqnarray}
being the smooth family of tempered distributions parametrized by
$X$ defined via the point-wise trace of \prettyref{eq: jth coefficient functional kernal}.
Letting $H\left(s\right)\in\mathcal{S}'\left(\mathbb{R}_{s}\right)$
denote the Heaviside distribution, we now define the following elementary
tempered distributions 
\begin{align}
v_{a;p}\left(s\right) & \coloneqq s^{a},\;a\in\mathbb{N}_{0}\label{eq: elementary distribution 1}\\
v_{a,b,c,\varLambda;p}\left(s\right) & \coloneqq\partial_{s}^{a}\left[\left|s\right|s^{b}\left(s^{2}-2\nu_{p}\varLambda\right)^{c-\frac{1}{2}}H\left(s^{2}-2\nu_{p}\varLambda\right)\right],\label{eq: elementary distribution 2}\\
 & \;\qquad\qquad\qquad\qquad\left(a,b,c;\varLambda\right)\in\mathbb{N}_{0}\times\mathbb{Z}\times\mathbb{N}_{0}\times\mu.\left(\mathbb{N}_{0}^{m}\setminus0\right).\nonumber 
\end{align}
We now have the following. 
\begin{prop}
\label{prop: structure trace distributions}For each $j$, the distribution
\prettyref{eq: pointwise trace distribution} can be written in terms
of \prettyref{eq: elementary distribution 1}, \prettyref{eq: elementary distribution 2}
\begin{equation}
u_{j,p}\left(s\right)=\sum_{a\leq2j+2}c_{j;a}\left(p\right)s^{a}+\sum_{\begin{subarray}{l}
\varLambda\in\mu.\left(\mathbb{N}_{0}^{m}\setminus0\right).\\
a,\left|b\right|,c\leq4j+4
\end{subarray}}c_{j;a,b,c,\varLambda}\left(p\right)v_{a,b,c,\varLambda;p}\left(s\right).\label{eq: trace distribution structure}
\end{equation}
Moreover, the coefficient functions $c_{j;a}$, $c_{j;a,b,c,\varLambda}\in C^{\infty}\left(X\right)$
above are evaluations at $p$ of polynomials in the covariant derivatives
(with respect to $\nabla^{TX}\otimes1+1\otimes\nabla^{A_{0}}$) of
the curvatures $F^{TX},F^{A_{0}}$ of the Levi-Civita connection $\nabla^{TX}$,
$\nabla^{A_{0}}$ and $da$. \end{prop}
\begin{proof}
It suffices to consider the restriction of $u_{j}$ to the interval
$\left(-\sqrt{2\nu M},\sqrt{2\nu M}\right)$ for each $0<M\notin\mu.\left(\mathbb{N}_{0}^{m}\setminus0\right)$.
We begin by finding the spectrum of the operator $\mathrm{\mathtt{D}}_{00}$
in \prettyref{eq: leading term rescaled Dirac-1}. To this end, define
the unitary operator $\mathtt{U}_{\lambda}:C^{\infty}\left(\mathbb{R}^{n};\mathbb{C}^{2^{m}}\right)\rightarrow C^{\infty}\left(\mathbb{R}^{n};\mathbb{C}^{2^{m}}\right)$
\begin{eqnarray*}
\left(\mathtt{U}_{\lambda}s\right)\left(x_{0},x_{1},x_{2},\ldots\right) & = & \left(\Pi_{j=1}^{m}\lambda_{j}\right)s\left(x_{0},\lambda_{1}^{-\frac{1}{2}}x_{1},\lambda_{1}^{-\frac{1}{2}}x_{2},\lambda_{2}^{-\frac{1}{2}}x_{3},\lambda_{2}^{-\frac{1}{2}}x_{4},\ldots\right)\\
\textrm{and }\;f & = & \sum_{j=1}^{m}\left(x_{j}x_{j+m}+\xi_{j}\xi_{j+m}\right)\in C^{\infty}\left(\mathbb{R}^{2m}\right).
\end{eqnarray*}
Next, as in \prettyref{eq:first conjugation normal form} we compute
the conjugate 
\[
e^{\frac{i\pi}{4}f_{0}^{W}}\mathtt{U}_{\lambda}\mathrm{\mathtt{D}}_{00}\mathtt{U}_{\lambda}^{*}e^{-\frac{i\pi}{4}f_{0}^{W}}=\left[2\nu\left(p\right)\right]^{\frac{1}{2}}\left.D_{\mathbb{R}^{m}}\right|_{h=1}
\]
of the operator in \prettyref{eq: leading term rescaled Dirac-1}
in terms of the magnetic Dirac operator on $\mathbb{R}^{m}$ \prettyref{eq: magnetic Dirac Rm}
evaluated at $h=1$. Hence the eigenspaces of $\mathrm{\mathtt{D}}_{00}$
are 
\begin{eqnarray*}
\mathtt{U}_{\lambda}^{*}e^{-\frac{i\pi}{4}f_{0}^{W}}\left(E_{0}\otimes L^{2}\left(\mathbb{R}_{x_{0},x''}^{m+1}\right)\right),\\
\mathtt{U}_{\lambda}^{*}e^{-\frac{i\pi}{4}f_{0}^{W}}\left(E_{\varLambda}^{\pm}\otimes L^{2}\left(\mathbb{R}_{x_{0},x''}^{m+1}\right)\right); &  & \varLambda\in\mu.\left(\mathbb{N}_{0}^{m}\setminus0\right),
\end{eqnarray*}
with eigenvalues $0$, $\pm\sqrt{2\nu\varLambda}$ respectively, where
\begin{eqnarray*}
E_{0} & \coloneqq & \mathbb{C}\left[\left.\psi_{0,0}\right|_{h=1}\right]\\
E_{\varLambda}^{\pm} & = & \bigoplus_{\begin{subarray}{l}
\tau\in\mathbb{N}_{0}^{m}\setminus0\\
\varLambda=\mu.\tau
\end{subarray}}\left.E_{\tau}^{\pm}\right|_{h=1},
\end{eqnarray*}
are as in \prettyref{eq: Landau Decomposition}. We again let $\mathtt{P}_{0}$,
$\mathtt{P}_{\varLambda}^{\pm}$ denote the respective projections
onto the eigenspaces of $\mathrm{\mathtt{D}}_{00}$ and $\mathtt{P}_{\varLambda}=\mathtt{P}_{\varLambda}^{+}\oplus\mathtt{P}_{\varLambda}^{-}$.
We also denote by $\mathtt{P}_{>M}=\oplus_{\varLambda>M}\mathtt{P}_{\varLambda}$
the projection onto eigenspaces with eigenvalue greater than $\sqrt{2\nu M}$
in absolute value. 

Now, since expansions in $L_{\textrm{loc}}^{2}$ are unique it suffices
to work with the resolvent expansion \prettyref{eq: resolvent expansion}
in the computation of $u_{j}$. The $j$th term in the expansion is
of the form 
\begin{equation}
\mathtt{C}_{j}^{z}=\sum_{k\in I_{j}}\left(z-\mathrm{\mathtt{D}}_{0}\right)^{-1}\left[\Pi_{\alpha}\mathrm{\mathtt{D}}_{k_{\alpha}}\left(z-\mathrm{\mathtt{D}}_{0}\right)^{-1}\right]\label{eq: jth term expansion of kernel}
\end{equation}
where each $\mathrm{\mathtt{D}}_{k_{\alpha}}$ is a differential operator
with polynomial coefficients involving the covariant derivatives of
the curvatures $F^{TX},F^{A_{0}}$ and $da$. Now using \prettyref{eq: leading term rescaled Dirac-1}
we decompose each resolvent term above according to the eigenspaces
of $\mathrm{\mathtt{D}}_{00}$ 
\begin{eqnarray}
\left(z-\mathrm{\mathtt{D}}_{0}\right)^{-1} & = & \mathtt{P}_{0}\left(\frac{1}{z-\gamma^{0}\partial_{x_{0}}}\right)\mathtt{P}_{0}\oplus\nonumber \\
 &  & \bigoplus_{\varLambda\in\mu.\mathbb{N}_{0}^{m}\cap\left(0,M\right)}\mathtt{P}_{\varLambda}\left(\frac{z+\gamma^{0}\partial_{x_{0}}+\mathrm{\mathtt{D}}_{00}}{z^{2}+\partial_{x_{0}}^{2}-2\nu\varLambda}\right)\mathtt{P}_{\varLambda}\label{eq: decoposition resolvent}\\
 &  & \oplus\mathtt{P}_{>M}\left(\frac{z+\gamma^{0}\partial_{x_{0}}+\mathrm{\mathtt{D}}_{00}}{z^{2}+\partial_{x_{0}}^{2}-\mathrm{\mathtt{D}}_{00}^{2}}\right)\mathtt{P}_{>M}.\nonumber 
\end{eqnarray}
Next, we plug \prettyref{eq: decoposition resolvent} into \prettyref{eq: jth term expansion of kernel}.
This gives an expansion for $\mathtt{C}_{j}^{z}$ with some of the
terms given by 
\begin{eqnarray*}
T^{z}\left[\Pi_{\alpha}\mathrm{\mathtt{D}}_{k_{\alpha}}T^{z}\right] & ; & \textrm{ where}\\
T^{z} & = & \mathtt{P}_{>M}\left(\frac{z+\gamma^{0}\partial_{x_{0}}+\mathrm{\mathtt{D}}_{00}}{z^{2}+\partial_{x_{0}}^{2}-\mathrm{\mathtt{D}}_{00}^{2}}\right)\mathtt{P}_{>M}
\end{eqnarray*}
and being holomorphic for $\textrm{Re}z\in\left(-\sqrt{2\nu M},\sqrt{2\nu M}\right).$
For the rest of the terms in $\mathtt{C}_{j}^{z}$, we use the commutation
relations 
\begin{eqnarray*}
\left[\gamma^{0},\mathtt{P}_{0}\right] & = & \left[\gamma^{0},\mathtt{P}_{\varLambda}\right]=\left[\gamma^{0},\mathtt{P}_{>M}\right]=0\\
\left[\partial_{x_{0}},\mathtt{P}_{0}\right] & = & \left[\partial_{x_{0}},\mathtt{P}_{\varLambda}\right]=\left[\partial_{x_{0}},\mathtt{P}_{>M}\right]=0\\
\left[\partial_{x_{0}},\mathrm{\mathtt{D}}_{00}\right] & = & 0\\
\left[\left(z^{2}+\partial_{x_{0}}^{2}-2\nu\varLambda\right)^{-1},x_{j}\right] & = & \delta_{0j}\left(z^{2}+\partial_{x_{0}}^{2}-2\nu\varLambda\right)^{-2}\partial_{x_{0}}\\
\left[\left(z^{2}+\partial_{x_{0}}^{2}-2\nu\varLambda\right)^{-1},\partial_{x_{J}}\right] & = & 0
\end{eqnarray*}
as well as the Clifford relations \prettyref{eq: Clifford relations}.
This now gives a finite sum of terms of the form 
\begin{align}
T_{0}^{z}\left[\Pi_{k=1}^{K}S_{k}T_{k}^{z}\right]\times\left[\Pi_{\varLambda\in\mu.\mathbb{N}_{0}^{m}\cap\left(0,M\right)}\frac{1}{\left(z^{2}+\partial_{x_{0}}^{2}-2\nu\varLambda\right)^{a_{\varLambda}}}\right]\left(z-\gamma^{0}\partial_{x_{0}}\right)^{-a_{0}}z^{b_{1}}x_{0}^{b_{2}}\partial_{x_{0}}^{b_{3}},\label{eq: term in kernel expansion after commutations}
\end{align}
$a_{0}+\Sigma a_{\varLambda}\leq2j+2$; $b_{1},b_{2},b_{3}\leq j+1,$
where each $S_{k}$ is a differential operator in $\left(x'x''\right)$
(i.e. independent of $x_{0}$) with polynomial coefficients and each
\begin{equation}
T_{k}^{z}=\begin{cases}
\mathtt{P}_{0} & \textrm{ or}\\
\mathtt{P}_{\varLambda}, & \textrm{ or}\\
\mathtt{P}_{\varLambda}\mathrm{\mathtt{D}}_{00}\mathtt{P}_{\varLambda}, & \textrm{ or}\\
\mathtt{P}_{>M}\left(\frac{1}{z^{2}+\partial_{x_{0}}^{2}-\mathrm{\mathtt{D}}_{00}^{2}}\right)\mathtt{P}_{>M}, & \textrm{ or}\\
\mathtt{P}_{>M}\left(\frac{\mathrm{\mathtt{D}}_{00}}{z^{2}+\partial_{x_{0}}^{2}-\mathrm{\mathtt{D}}_{00}^{2}}\right)\mathtt{P}_{>M}
\end{cases}\label{eq: Tk z}
\end{equation}
with at least one occurrence of $\mathtt{P}_{0},\mathtt{P}_{\varLambda}$
or $\mathtt{P}_{\varLambda}\mathrm{\mathtt{D}}_{00}\mathtt{P}_{\varLambda}$
in \prettyref{eq: term in kernel expansion after commutations}. Now
using partial fractions, \prettyref{eq: term in kernel expansion after commutations}
may be written as a sum of terms of the form 
\begin{align}
T_{0}^{z}\left[\Pi_{k=1}^{K}S_{k}T_{k}^{z}\right]\times\left(z-\gamma^{0}\partial_{x_{0}}\right)^{-a_{0}}z^{b_{1}}x_{0}^{b_{2}}\partial_{x_{0}}^{b_{3}},\nonumber \\
T_{0}^{z}\left[\Pi_{k=1}^{K}S_{k}T_{k}^{z}\right]\times\left(z^{2}+\partial_{x_{0}}^{2}-2\nu\varLambda\right)^{-a_{\varLambda}}z^{b_{1}}x_{0}^{b_{2}}\partial_{x_{0}}^{b_{3}}; & \;\varLambda\in\mu.\mathbb{N}_{0}^{m}\cap\left(0,M\right),\label{eq: resolvent expansion simple term}
\end{align}
$a_{0},a_{\varLambda}\leq2j+2$; $b_{1},b_{2},b_{3}\leq j+1$. Next,
we plug \prettyref{eq: resolvent expansion simple term} into the
Helffer-Sjostrand formula and use the holomorphicity of $\mathtt{P}_{>M}\left(\frac{1}{z^{2}+\partial_{x_{0}}^{2}-\mathrm{\mathtt{D}}_{00}^{2}}\right)\mathtt{P}_{>M}$
and $\mathtt{P}_{>M}\left(\frac{\mathrm{\mathtt{D}}_{00}}{z^{2}+\partial_{x_{0}}^{2}-\mathrm{\mathtt{D}}_{00}^{2}}\right)\mathtt{P}_{>M}$
for $\textrm{Re}z\in\left(-\sqrt{2\nu M},\sqrt{2\nu M}\right)$. This
gives 
\[
U_{j,p}\left(\phi\right)=-\frac{1}{\pi}\int_{\mathbb{C}}\bar{\partial}\tilde{\phi}\left(z\right)\mathtt{C}_{j}^{z}\left(0,0\right)dzd\bar{z},
\]
for $\phi\in C_{c}^{\infty}\left(-\sqrt{2\nu M},\sqrt{2\nu M}\right)$,
as a sum of terms of the form 
\begin{align}
\left(T_{0}^{0}\left[\Pi_{k=1}^{K}S_{k}T_{k}^{0}\right]\times x_{0}^{b_{2}}\partial_{x_{0}}^{b_{3}}\phi_{0}\left(\gamma^{0}\partial_{x_{0}}\right)\right)\left(0,0\right),\nonumber \\
\left(T_{0}^{0}\left[\Pi_{k=1}^{K}S_{k}T_{k}^{0}\right]\times x_{0}^{b_{2}}\partial_{x_{0}}^{b_{3}}\phi_{\varLambda}\left(-\partial_{x_{0}}^{2}+2\nu\varLambda\right)\right)\left(0,0\right),\; & \varLambda\in\mu.\mathbb{N}_{0}^{m}\cap\left(0,M\right);\label{eq: jth coefficient}
\end{align}
where
\[
T_{k}^{0}=\begin{cases}
\mathtt{P}_{0}, & \textrm{ or}\\
\mathtt{P}_{\varLambda}, & \textrm{ or}\\
\mathtt{P}_{\varLambda}\mathrm{\mathtt{D}}_{00}\mathtt{P}_{\varLambda}, & \textrm{ or}\\
\mathtt{P}_{>M}\left(\frac{1}{2\nu\varLambda-\mathrm{\mathtt{D}}_{00}^{2}}\right)\mathtt{P}_{>M}, & \textrm{ or}\\
\mathtt{P}_{>M}\left(\frac{\mathrm{\mathtt{D}}_{00}}{2\nu\varLambda-\mathrm{\mathtt{D}}_{00}^{2}}\right)\mathtt{P}_{>M};
\end{cases}
\]
and 
\begin{eqnarray*}
\phi_{0}\left(s\right) & = & \frac{\left(-1\right)^{a_{0}-1}}{\left(a_{0}-1\right)!}x^{b_{1}}\phi\left(s\right)\\
\phi_{\varLambda}\left(s^{2}\right) & = & \frac{\left(-1\right)^{a_{\varLambda}-1}}{\left(a_{\varLambda}-1\right)!}\left\{ \left.\left[\partial_{r}^{a_{\varLambda}-1}\left(\frac{r^{b_{1}}\phi\left(r\right)}{\left(r-s\right)^{a_{\varLambda}}}\right)\right]\right|_{r=-s}\right.\\
 &  & \quad\quad-\left.\left.\left[\partial_{r}^{a_{\varLambda}-1}\left(\frac{r^{b_{1}}\phi\left(r\right)}{\left(r+s\right)^{a_{\varLambda}}}\right)\right]\right|_{r=s}\right\} .
\end{eqnarray*}
At least one occurrence of $\mathtt{P}_{0},$$\mathtt{P}_{\varLambda}$
and $\mathtt{P}_{\varLambda}\mathrm{\mathtt{D}}_{00}\mathtt{P}_{\varLambda}$
in \prettyref{eq: jth coefficient} gives the smoothness of the kernel. 

Finally, an elementary computation involving Laplace transforms using
the knowledge of the heat kernel $e^{t\partial_{x_{0}}^{2}}\left(x_{0},y_{0}\right)=\frac{1}{\sqrt{4\pi t}}e^{-\left|x_{0}-y_{0}\right|^{2}/4t}$
gives 
\begin{align*}
x_{0}^{b_{2}}\partial_{x_{0}}^{b_{3}}\phi_{0}\left(\gamma^{0}\partial_{x_{0}}\right)\left(0,0\right) & =\frac{\left(-\frac{1}{2}\right)^{\left[\frac{b_{3}+1}{2}\right]}}{\sqrt{\pi}\Gamma\left(\left[\frac{b_{3}+1}{2}\right]+\frac{1}{2}\right)}\delta_{0b_{2}}v_{b_{3};p}\left(\phi_{0}\right)\\
x_{0}^{b_{2}}\partial_{x_{0}}^{b_{3}}\phi_{\varLambda}\left(-\partial_{x_{0}}^{2}+2\nu\varLambda\right)\left(0,0\right) & =\begin{cases}
\frac{\left(-\frac{1}{2}\right)^{\frac{b_{3}}{2}}}{4\pi\Gamma\left(\frac{b_{3}}{2}-\frac{1}{2}\right)}\delta_{0b_{2}}v_{0,0,\frac{b_{3}}{2},\varLambda;p}\left(\phi_{\varLambda}\left(s^{2}\right)\right); & b_{3}\textrm{ even}\\
0; & b_{3}\textrm{ odd},
\end{cases}
\end{align*}
completing the proof. 
\end{proof}
As an immediate corollary of the above proposition \prettyref{prop: structure trace distributions}
we have that the distributions $u_{j}$ are smooth near $0$.
\begin{cor}
\label{cor: uj smooth near 0}For each $j$, 
\[
\textrm{singspt}\left(u_{j}\right)\subset\mathbb{R}\setminus\left(-\sqrt{2\nu_{0}},\sqrt{2\nu_{0}}\right).
\]
\end{cor}
\begin{proof}
This follows immediately from \prettyref{eq: pointwise trace distribution},
\prettyref{eq: elementary distribution 1}, \prettyref{eq: elementary distribution 2}
and \prettyref{eq: trace distribution structure} on noting that the
distributions $v_{a;p}$ are smooth while $v_{a,b,c,\varLambda;p}=0$
on $\mathbb{R}\setminus\left(-\sqrt{2\nu_{0}},\sqrt{2\nu_{0}}\right)$
for each $p\in X$. 
\end{proof}
We next give the exact computation for the first coefficient $u_{0}$
of \prettyref{prop: local trace expansion}. In the computation below
, recall that $Z_{\tau}=\left|I_{\tau}\right|$ \prettyref{eq: Z_r}
denotes the number of non-zero components of $\tau\in\mathbb{N}_{0}^{m}\setminus0$. 
\begin{prop}
The first coefficient $u_{0}$ of \prettyref{eq: local trace expansion}
is given by 
\begin{eqnarray}
u_{0,p} & = & c_{0;0}+\sum_{\begin{subarray}{l}
\varLambda\in\mu.\left(\mathbb{N}_{0}^{m}\setminus0\right)\end{subarray}}c_{0;0,0,0,\varLambda}\left(p\right)v_{0,0,0,\varLambda;p}\left(s\right),\;\textrm{ where}\label{eq: computation u0}\\
c_{0;0} & = & \frac{\nu_{p}^{m}\left(\Pi_{j=1}^{m}\mu_{j}\right)}{\left(4\pi\right)^{m}}\;\textrm{ and}\nonumber \\
c_{0;0,0,0,\varLambda}\left(p\right) & = & \frac{\nu_{p}^{m}\left(\Pi_{j=1}^{m}\mu_{j}\right)}{\left(4\pi\right)^{m}}\textrm{ dim }\left(E_{\varLambda}\right)\nonumber \\
 & = & \frac{\nu_{p}^{m}\left(\Pi_{j=1}^{m}\mu_{j}\right)}{\left(4\pi\right)^{m}}\left(\sum_{\begin{subarray}{l}
\tau\in\mathbb{N}_{0}^{m}\setminus0\\
\mu.\tau=\varLambda
\end{subarray}}2^{Z_{\tau}}\right).
\end{eqnarray}
\end{prop}
\begin{proof}
First note that the square of \prettyref{eq: leading term rescaled Dirac}
gives the harmonic oscillator 
\[
\mathrm{\mathtt{D}}_{0}^{2}=-\delta^{jk}\partial_{x_{j}}\partial_{x_{k}}-i\left(da\right){}_{k}^{j}\left(0\right)x^{k}\partial_{x_{j}}+\frac{1}{4}x^{k}x_{l}\left(da\right){}_{k}^{j}\left(0\right)\left(da\right){}_{j}^{l}\left(0\right)+\frac{i}{2}\gamma^{j}\gamma^{k}\left(da\right){}_{jk}\left(0\right).
\]
The heat kernel $e^{-t\mathrm{\mathtt{D}}_{0}^{2}}$ of the above
is given by Mehler's formula (cf. \cite{Berline-Getzler-Vergne} section
4.2) 
\begin{align}
e^{-t\mathrm{\mathtt{D}}_{0}^{2}}\left(x,y\right) & =\frac{1}{\left(4\pi t\right){}^{m}}\det{}^{\frac{1}{2}}\left(\frac{itda\left(0\right)}{\sinh itda\left(0\right)}\right)\label{eq: Mehler's formula}\\
 & \qquad\qquad\times\exp\left\{ -\frac{1}{4t}\left\langle \left(x-y\right),itda\left(0\right)\coth\left(itda\left(0\right)\right)\left(x-y\right)\right\rangle \right\} e^{-tc\left(ida\left(0\right)\right)}\nonumber 
\end{align}
Next, using \prettyref{eq: da diagonal form} we compute 
\begin{equation}
e^{-tc\left(ida\left(0\right)\right)}=\Pi_{j=1}^{m}\left[\cosh\left(t\lambda_{j}\right)-ic\left(e_{j}\right)c\left(e_{j+m}\right)\sinh\left(t\lambda_{j}\right)\right].\label{eq: Exponential Clifford mult.}
\end{equation}
For $I\subset\left\{ 2,\ldots,m\right\} $ and $\omega_{I}=\bigwedge_{j\in I}\left(e_{j}\wedge e_{j+m}\right)$,
the commutation 
\[
c\left(e_{1}\right)c\left(e_{m+1}\right)c\left(\omega_{I}\right)=\frac{1}{2}\left[c\left(e_{1}\right),c\left(e_{m+1}\right)c\left(\omega_{I}\right)\right]
\]
shows that the only traceless terms in \prettyref{eq: Exponential Clifford mult.}
are the constants. Hence, Mehler's formula \prettyref{eq: Mehler's formula}
gives 
\begin{eqnarray}
\textrm{tr }e^{-t\mathrm{\mathtt{D}}_{0}^{2}}\left(0,0\right) & = & \frac{1}{\left(4\pi t\right){}^{m}}\det{}^{\frac{1}{2}}\left(\frac{itda\left(0\right)}{\tanh itda\left(0\right)}\right)\nonumber \\
 & = & \frac{t^{-\frac{1}{2}}}{\left(4\pi\right)^{m}}\left(\Pi_{j=1}^{m}\frac{\lambda_{j}}{\tanh t\lambda_{j}}\right)\nonumber \\
 & = & \frac{t^{-\frac{1}{2}}}{\left(4\pi\right)^{m}}\left[\Pi_{j=1}^{m}\lambda_{j}\left(1+2e^{-2t\lambda_{j}}+2e^{-4t\lambda_{j}}+\ldots\right)\right]\nonumber \\
 & = & \frac{t^{-\frac{1}{2}}}{\left(4\pi\right)^{m}}\left(\Pi_{j=1}^{m}\lambda_{j}\right)\left(\sum_{\tau\in\mathbb{N}_{0}^{m}}2^{Z_{\tau}}e^{-2t\tau.\lambda}\right)\nonumber \\
 & = & \frac{\nu_{p}^{m}\left(\Pi_{j=1}^{m}\mu_{j}\right)}{\left(4\pi\right)^{m}}\left(t^{-\frac{1}{2}}\sum_{\tau\in\mathbb{N}_{0}^{m}}2^{Z_{\tau}}e^{-2t\tau.\lambda}\right)\nonumber \\
 & = & u_{0,p}\left(e^{-ts^{2}}\right)\label{eq: evaluation gaussian}
\end{eqnarray}
with $u_{0,p}$ as in \prettyref{eq: computation u0} and the last
line above following from an easy computation of Laplace transforms
(see \cite{Savale-Asmptotics} section 4). Furthermore, differentiating
Mehler's formula using \prettyref{eq: leading term rescaled Dirac}
gives 
\begin{equation}
\textrm{tr}\mathrm{\mathtt{D}}_{0}e^{-t\mathrm{\mathtt{D}}_{0}^{2}}\left(0,0\right)=0=u_{0,p}\left(se^{-ts^{2}}\right)\label{eq: evaluation odd gaussian}
\end{equation}
since the right hand side of \prettyref{eq: computation u0} is an
even distribution. From \prettyref{eq: evaluation gaussian} and \prettyref{eq: evaluation odd gaussian}
we have that the evaluations of both sides of \prettyref{eq: computation u0}
on $e^{-ts^{2}},se^{-ts^{2}}$ are equal. Differentiating with respect
to $t$ and setting $t=1$ gives that the two sides of \prettyref{eq: computation u0}
evaluate equally on $s^{k}e^{-s^{2}}$, $\forall k\in\mathbb{N}_{0}$.
The proposition now follows from the density of this collection in
$\mathcal{S}\left(\mathbb{R}_{s}\right)$.
\end{proof}
We now complete the proof of \prettyref{lem: Easy trace expansion lemma}. 
\begin{proof}[Proof of \prettyref{lem: Easy trace expansion lemma}]
 We begin by writing 
\begin{eqnarray}
 &  & \textrm{tr}\left[f\left(\frac{D}{\sqrt{h}}\right)\frac{1}{h^{1-\epsilon}}\check{\theta}\left(\frac{\lambda\sqrt{h}-D}{h^{1-\epsilon}}\right)\right]\nonumber \\
 & = & \frac{h^{-\frac{1}{2}}}{2\pi}\int\,dt\textrm{tr}\left[f\left(\frac{D}{\sqrt{h}}\right)e^{it\left(\lambda-\frac{D}{\sqrt{h}}\right)}\right]\theta\left(th^{\frac{1}{2}-\epsilon}\right).\label{eq: local trace in terms of wave kernel}
\end{eqnarray}
Next, the expansion \prettyref{prop: local trace expansion}, with
$\phi\left(x\right)=f\left(x\right)e^{it\left(\lambda-x\right)}$,
combined with the smoothness of $u_{j}$ on $\textrm{spt}\left(f\right)\subset\left(-\sqrt{2\nu_{0}},\sqrt{2\nu_{0}}\right)$
\prettyref{cor: uj smooth near 0} gives 
\begin{eqnarray}
\textrm{tr}\left[f\left(\frac{D}{\sqrt{h}}\right)e^{it\left(\lambda-\frac{D}{\sqrt{h}}\right)}\right] & = & e^{it\lambda}h^{-n/2}\left(\sum_{j=0}^{N}h^{j/2}\widehat{fu_{j}}\left(t\right)\right)\nonumber \\
 &  & +h^{\left(N+1-n\right)/2}\underbrace{O\left(\sum_{k=0}^{n+1}\left\Vert \left\langle \xi\right\rangle ^{N}\hat{\phi}^{\left(k\right)}\left(\xi-t\right)\right\Vert _{L^{1}}\right)}_{=O\left(\left\langle t\right\rangle ^{N}\right)}.\label{eq:trace expansion to substitute}
\end{eqnarray}
Finally, plugging \prettyref{eq:trace expansion to substitute} into
\prettyref{eq: local trace in terms of wave kernel} and using $\theta\left(th^{\frac{1}{2}-\epsilon}\right)=1+O\left(h^{\infty}\right)$
gives via Fourier inversion 
\begin{eqnarray*}
 &  & \frac{h^{-\frac{1}{2}}}{2\pi}\int\,dt\textrm{tr}\left[f\left(\frac{D}{\sqrt{h}}\right)e^{it\left(\lambda-\frac{D}{\sqrt{h}}\right)}\right]\theta\left(th^{\frac{1}{2}-\epsilon}\right)\\
 & = & h^{-m-1}\left(\sum_{j=0}^{N}h^{j/2}f\left(\lambda\right)u_{j}\left(\lambda\right)\right)+O\left(h^{\epsilon\left(N+1\right)-m-1}\right)
\end{eqnarray*}
as required.
\end{proof}

\section{\label{sec:Asymptotics-of-spectral invariants}Asymptotics of spectral
invariants}

In this section we prove theorem \prettyref{thm: asmptotics spectral invariants}
on the asymptotics of the spectral invariants.
\begin{proof}[Proof of \prettyref{thm: asmptotics spectral invariants}]
 To prove the local Weyl law \prettyref{eq:local Weyl counting function est},
we choose\\
 $\theta\in C_{c}^{\infty}\left(\left(-T,T\right);\left[0,1\right]\right)$
such that $\theta\left(x\right)=1$ on $\left(-T',T'\right)$, $T'<T$,
$\check{\theta}\left(\xi\right)\geq0$ and $\check{\theta}\left(\xi\right)\geq1$
for $\left|\xi\right|\leq c$ in \prettyref{thm:main trace expansion}.
Choosing $f\left(x\right)\geq0$ with $f\left(0\right)=1$, the trace
expansion \prettyref{eq: Main trace expansion} with $\lambda=0$
now gives 
\[
\frac{1}{h}N\left(-ch,ch\right)\left(1+O\left(\sqrt{h}\right)\right)\leq\textrm{tr}\left[f\left(\frac{D}{\sqrt{h}}\right)\frac{1}{h}\check{\theta}\left(\frac{-D}{h}\right)\right]=O\left(h^{-m-1}\right)
\]
 proving \prettyref{eq:local Weyl counting function est}. 

To prove the estimate \prettyref{eq: eta estimate} on the eta invariant,
we first use its invariance under positive scaling \prettyref{eq: eta scale invariant}
and the formula \prettyref{eq: eta integral} to write 
\begin{eqnarray}
\eta_{h}=\eta\left(\frac{D}{\sqrt{h}}\right) & = & \int_{0}^{\infty}dt\frac{1}{\sqrt{\pi t}}\textrm{ tr}\left[\frac{D}{\sqrt{h}}e^{-\frac{t}{h}D^{2}}\right]\nonumber \\
 & = & \int_{0}^{1}dt\frac{1}{\sqrt{\pi t}}\textrm{ tr}\left[\frac{D}{\sqrt{h}}e^{-\frac{t}{h}D^{2}}\right]+\int_{1}^{\infty}dt\frac{1}{\sqrt{\pi t}}\textrm{ tr}\left[\frac{D}{\sqrt{h}}e^{-\frac{t}{h}D^{2}}\right].\label{eq: eta integral break up}
\end{eqnarray}
Next, the equation 4.5 pg. 859 of \cite{Savale-Asmptotics} with $r=\frac{1}{h}$
translates to the estimate 
\begin{equation}
\textrm{ tr}\left[\frac{D}{\sqrt{h}}e^{-\frac{t}{h}D^{2}}\right]=O\left(h^{-m}e^{ct}\right).\label{eq: estimate on odd trace}
\end{equation}
Plugging, \prettyref{eq: estimate on odd trace} into the first integral
of \prettyref{eq: eta integral break up} gives 
\begin{equation}
\eta_{h}=O\left(h^{-m}\right)+\textrm{tr }E\left(\frac{D}{\sqrt{h}}\right)\label{eq: eta =00003D tr E + h-m}
\end{equation}
where 
\[
E(x)=\text{sign}(x)\text{erfc}(|x|)=\text{sign}(x)\cdot\frac{2}{\sqrt{\pi}}\int_{|x|}^{\infty}e^{-s^{2}}ds
\]
with the convention $\text{sign}(0)=0$. The function $E\left(x\right)$
above is rapidly decaying with all derivatives, odd and smooth on
$\mathbb{R}_{x}\setminus0$. We may hence choose functions $f\in C_{c}^{\infty}\left(-\sqrt{2\nu_{0}},\sqrt{2\nu_{0}}\right),$
$g\in C_{c}^{\infty}\left(\mathbb{R}_{<0}\right)$ such that 
\begin{eqnarray*}
f\left(x\right)+g\left(x\right) & = & E\left(x\right)\textrm{ for }x\leq0.
\end{eqnarray*}
Define the spectral measure $\mathfrak{M}_{f}\left(\lambda'\right)\coloneqq\sum_{\lambda\in\textrm{Spec}\left(\frac{D}{\sqrt{h}}\right)}f\left(\lambda\right)\delta\left(\lambda-\lambda'\right)$.
It is clear that the expansion \prettyref{eq: Main trace expansion}
to its first term may be written as 
\[
\mathfrak{M}_{f}\ast\left(\mathcal{F}_{h}^{-1}\theta_{\frac{1}{2}}\right)\left(\lambda\right)=h^{-m-\frac{1}{2}}\left(f\left(\lambda\right)u_{0}\left(\lambda\right)+O\left(h^{1/2}\right)\right)
\]
where $\theta_{\frac{1}{2}}\left(x\right)=\theta\left(\frac{x}{\sqrt{h}}\right)$
as before. Both sides above involving Schwartz functions in $\lambda$,
the remainder maybe replaced by $O\left(\frac{h^{1/2}}{\left\langle \lambda\right\rangle ^{2}}\right)$.
One may then integrate the equation to obtain 
\begin{align}
 & \int_{-\infty}^{0}d\lambda\int d\lambda'\left(\mathcal{F}_{h}^{-1}\theta_{\frac{1}{2}}\right)\left(\lambda-\lambda'\right)\mathfrak{M}_{f}\left(\lambda'\right)\label{eq: integrated trace expansion}\\
= & h^{-m-\frac{1}{2}}\left(\int_{-\infty}^{0}d\lambda f\left(\lambda\right)u_{0}\left(\lambda\right)+O\left(h^{1/2}\right)\right).\nonumber 
\end{align}
Next we observe 
\begin{eqnarray}
\int_{-\infty}^{0}d\lambda\left(\mathcal{F}_{h}^{-1}\theta_{\frac{1}{2}}\right)\left(\lambda-\lambda'\right) & = & \int_{-\infty}^{0}dt\check{\theta}\left(t-\frac{\lambda'}{\sqrt{h}}\right)\nonumber \\
 & = & 1_{\left(-\infty,0\right]}\left(\lambda'\right)+O\left(\left\langle \frac{\lambda'}{\sqrt{h}}\right\rangle ^{-\infty}\right).\label{eq: integral theta -infty 0}
\end{eqnarray}
While the local Weyl law yields 
\begin{equation}
\int d\lambda'\mathfrak{M}_{f}\left(\lambda'\right)O\left(\left\langle \frac{\lambda'}{\sqrt{h}}\right\rangle ^{-\infty}\right)=O\left(h^{-m}\right).\label{eq: local Weyl law est}
\end{equation}
Substituting \prettyref{eq: integral theta -infty 0} and \prettyref{eq: local Weyl law est}
into \prettyref{eq: integrated trace expansion} gives 
\[
\sum_{\begin{subarray}{l}
\quad\:\lambda\leq0\\
\lambda\in\textrm{Spec}\left(\frac{D}{\sqrt{h}}\right)
\end{subarray}}f\left(\lambda\right)=h^{-m-\frac{1}{2}}\left(\int_{-\infty}^{0}d\lambda f\left(\lambda\right)u_{0}\left(\lambda\right)\right)+O\left(h^{-m}\right).
\]
This combined with 
\[
\textrm{tr }g\left(\frac{D}{\sqrt{h}}\right)=h^{-m-\frac{1}{2}}u_{0}\left(g\right)+O\left(h^{-m}\right)
\]
then gives 
\[
\sum_{\begin{subarray}{l}
\quad\:\lambda\leq0\\
\lambda\in\textrm{Spec}\left(\frac{D}{\sqrt{h}}\right)
\end{subarray}}E\left(\lambda\right)=h^{-m-\frac{1}{2}}\left(\int_{-\infty}^{0}d\lambda E\left(\lambda\right)u_{0}\left(\lambda\right)\right)+O\left(h^{-m}\right)
\]
where the integral makes sense from the formula \prettyref{eq: computation u0}
for $u_{0}$. A similar formula for
\[
\sum_{\begin{subarray}{l}
\quad\:\lambda\geq0\\
\lambda\in\textrm{Spec}\left(\frac{D}{\sqrt{h}}\right)
\end{subarray}}E\left(\lambda\right)
\]
now gives 
\[
\textrm{tr }E\left(\frac{D}{\sqrt{h}}\right)=h^{-m-\frac{1}{2}}\left(\int_{-\infty}^{\infty}d\lambda E\left(\lambda\right)u_{0}\left(\lambda\right)\right)+O\left(h^{-m}\right).
\]
Since $E$ is odd and $u_{0}$ is even from \prettyref{eq: computation u0},
the integral above is zero and hence $\eta_{h}=\textrm{tr }E\left(\frac{D}{\sqrt{h}}\right)=O\left(h^{-m}\right)$
from \prettyref{eq: eta =00003D tr E + h-m} as required.
\end{proof}

\subsection{Sharpness of the result}

Here, we finally show that the result \prettyref{thm: asmptotics spectral invariants}
is sharp. The worst case example was already noted in \cite{Savale-Asmptotics}
Section 5 for $\eta_{h}$. To recall, we let $Y$ be a complex manifold
of dimension $2m$ with complex structure $J$ and a Riemannian metric
$g^{TY}$. Fix a positive, holomorphic, Hermitian line bundle $\mathcal{L}\rightarrow Y$.
The curvature $F^{\mathcal{L}}$ of the Chern connection is thus a
positive $\left(1,1\right)$ form. Let $X$ be the total space of
the unit circle bundle $S^{1}\rightarrow X\xrightarrow{\pi}Y$ of
$\mathcal{L}$. The Chern connection gives a splitting of the tangent
bundle 
\begin{equation}
TX=TS^{1}\oplus\pi^{*}TY\label{eq: connection splitting}
\end{equation}
where $TS^{1}$ is the vertical tangent space spanned by the generator
$e$ of the $S^{1}$ action. Define a metric $g^{TS^{1}}$ on $TS^{1}$
via $\left\Vert e\right\Vert _{g^{TS^{1}}}=1$. A metric on $X$ can
now be given using the splitting \prettyref{eq: connection splitting}
via 
\[
g^{TX}=g^{TS^{1}}\oplus\varepsilon^{-1}\pi^{*}g^{TY},
\]
for any $\varepsilon>0$. A spin structure on $Y$ corresponds to
a holomorphic, Hermitian square root $\mathcal{K}$ of the canonical
line bundle $K_{Y}=\mathcal{K}^{\otimes2}$. Fixing such a spin structure
as well as the trivial spin structure on $TS^{1}$ gives a spin structure
on $X$. Finally the one form $a=e^{*}\in\Omega^{1}\left(X\right)$
while the auxiliary is chosen to be trivial $L=\mathbb{C}$ with the
family of connections $\nabla^{h}=d+\frac{i}{h}a$. We now have the
required family of Dirac operators $D_{h}$ \prettyref{eq:Semiclassical Magnetic Dirac}.
One may check that $\left(X^{2m+1},a,g^{TX},J\right)$ here gives
a metric contact structure \prettyref{eq: metric contact structure}
and hence the assumption \prettyref{def: Diagonalizability assumption}
is satisfied.

Denote by $\Delta_{\bar{\partial}_{k}}^{p}:\Omega^{0,p}\left(X;\mathcal{K}\otimes\mathcal{L}^{\otimes k}\right)\longrightarrow\Omega^{0,p}\left(X;\mathcal{K}\otimes\mathcal{L}^{\otimes k}\right)$
the Hodge Laplacian acting on $\left(0,p\right)$ forms on $X$. Its
null-space is given by the cohomology $H^{p}\left(X;\mathcal{K}\otimes\mathcal{L}^{\otimes k}\right)$
of the tensor product via Hodge theory. Let $e_{\mu}^{p,k}$ denote
the dimension of a each positive eigenspace with eigenvalue $\frac{1}{2}\mu^{2}\in\textrm{Spec}^{+}\left(\Delta_{\bar{\partial}_{k}}^{p}\right)$.
The spectrum of $D_{h}$ was now computed in Proposition 5.2 of \cite{Savale-Asmptotics}.
\begin{prop}
The spectrum of $D_{h}$ is given by 
\begin{enumerate}
\item Type 1: 
\begin{equation}
\lambda=\left(-1\right)^{p}h\left(k+\left(\varepsilon-\frac{m}{2}\right)-\frac{1}{h}\right),\label{eq: type 1}
\end{equation}
$0\leq p\leq m,k\in\mathbb{Z},$ with multiplicity $\textrm{dim}H^{p}\left(X;\mathcal{K}\otimes\mathcal{L}^{\otimes k}\right)$.
\item Type 2:
\begin{equation}
\lambda=h\left[\frac{(-1)^{p+1}\varepsilon\pm\sqrt{(2k+\varepsilon(2p-m)-\frac{2}{h}+1)^{2}+4\mu^{2}\varepsilon}}{2}\right],\label{eq: type 2}
\end{equation}
$0\leq p\leq m,k\in\mathbb{Z},$ $\frac{1}{2}\mu^{2}\in\textrm{Spec}^{+}\left(\Delta_{\bar{\partial}_{k}}^{p}\right)$
with multiplicity $d_{\mu}^{p,k}\coloneqq e_{\mu}^{p,k}-e_{\mu}^{p-1,k}+\ldots+(-1)^{p}e_{\mu}^{0,k}$.
\end{enumerate}
\end{prop}
$\quad$ As observed in \cite{Savale-Asmptotics} on choosing 
\[
\varepsilon<\inf_{k,p}\left\{ \frac{1}{2}\mu^{2}\in\textrm{Spec}^{+}\left(\Delta_{\bar{\partial_{k}}}^{p}\right)\right\} 
\]
the eigenvalues of Type 2 are either positive or negative depending
on the sign appearing in \prettyref{eq: type 2}. Hence the dimension
of the kernel $k_{h}$ of $D_{h}$ is now given by the type 1 eigenvalues
\begin{equation}
k_{h}=\begin{cases}
\textrm{dim }H^{*}\left(X;\mathcal{K}\otimes\mathcal{L}^{\otimes k}\right); & \frac{1}{h}=k+\left(\varepsilon-\frac{m}{2}\right)\\
0 & \textrm{otherwise}.
\end{cases}\label{kernel computation}
\end{equation}
Now by a combination of Kodaira vanishing and Hirzebruch-Riemann-Roch
\begin{eqnarray}
\textrm{dim }H^{*}\left(X;\mathcal{K}\otimes\mathcal{L}^{\otimes k}\right) & = & \textrm{dim }H^{0}\left(X;\mathcal{K}\otimes\mathcal{L}^{\otimes k}\right)\nonumber \\
 & = & \chi(X,\mathcal{K}\otimes\mathcal{L}^{\otimes k})\nonumber \\
 & = & \int_{X}\textrm{ch}(\mathcal{K}\otimes\mathcal{L}^{\otimes k})\textrm{td}(X)\label{eq: kernel asymptotics}
\end{eqnarray}
for $k\gg0$, where $\chi(X,\mathcal{K}\otimes\mathcal{L}^{\otimes k}),\textrm{ch}(\mathcal{K}\otimes\mathcal{L}^{\otimes k})$
and $\textrm{td}(X)$ denote Euler characteristic, Chern character
and Todd genus respectively. Hence \ref{kernel computation}, \prettyref{eq: kernel asymptotics}
show that the kernel and hence the counting function are discontinuous
of order $O\left(h^{-m}\right)=k_{h}\leq N\left(-ch,ch\right)$ in
this example. A similar discontinuity of the eta invariant of $O\left(h^{-m}\right)$
was proved in Theorem 5.3 of \cite{Savale-Asmptotics}. 

\appendix

\section{\label{sec:Appendix A}Some spectral estimates}

In this appendix we prove some important spectral estimates used in
\prettyref{sec:Reduction to R^n} and \prettyref{sec: Birkhoff normal form}. 

Let $H$ be a separable Hilbert space. Let $A:H\rightarrow H$ be
a bounded self-adjoint operator. The resolvent set and the spectrum
of $A$ are defined to be 
\begin{eqnarray*}
R\left(A\right) & = & \left\{ \lambda\in\mathbb{C}|A-\lambda I\mbox{ is invertible}\right\} \\
\mbox{Spec}\left(A\right) & = & \mathbb{C}\setminus R\left(A\right).
\end{eqnarray*}
Since $A$ is self-adjoint, $\mbox{Spec}\left(A\right)\subset\mathbb{R}$.
We may now define the following subsets of the spectrum 
\begin{eqnarray*}
\mbox{EssSpec}\left(A\right) & = & \left\{ \lambda\in\mathbb{C}|A-\lambda I\mbox{ is not Fredholm}\right\} \\
\mbox{DiscSpec}\left(A\right) & = & \mbox{Spec}\left(A\right)\setminus\mbox{EssSpec}\left(A\right).
\end{eqnarray*}
We shall consider $\mbox{DiscSpec}\left(A\right)$ above as a multiset
with the multiplicity function $m^{A}:\mbox{DiscSpec}\left(A\right)\rightarrow\mathbb{N}_{0}$
defined by $m^{A}\left(\lambda\right)=\mbox{dim ker}\left(A\right)$.
We may then find a countable set of orthonormal eigenvectors $v_{1}^{A},v_{2}^{A},v_{3}^{A},\ldots$,
with eigenvalues $\lambda_{1}^{A}\leq\lambda_{2}^{A}\leq\lambda_{3}^{A}\leq\ldots$
such that $\mbox{DiscSpec}\left(A\right)=\left\{ \lambda_{1}^{A},\lambda_{2}^{A},\ldots\right\} $
as multisets. Now let $\left[a,b\right]\subset\mathbb{R}$ be a finite
closed interval such that $\mbox{EssSpec}\left(A\right)\cap\left[a,b\right]=\emptyset$
(i.e. $A$ has discrete spectrum  in $\left[a,b\right]$). Then
\[
H_{\left[a,b\right]}^{A}=\bigoplus_{\lambda\in\textrm{Spec}\left(A\right)\cap\left[a,b\right]}\mbox{ker}\left(A-\lambda\right)
\]
is a finite dimensional vector subspace of $H$. We let $\Pi_{\left[a,b\right]}^{A}:H\rightarrow H_{\left[a,b\right]}^{A}\subset H$
denote the orthogonal projection onto $H_{\left[a,b\right]}^{A}$.
We denote by $N_{\left[a,b\right]}^{A}$ the dimension of $H_{\left[a,b\right]}^{A}$.
The operator $\rho\left(A\right):H\rightarrow H$ may now be defined
for any function $\rho\in C_{c}^{0}\left(\left[a,b\right]\right)$
by functional calculus.
\begin{lem}
\label{lem:projection close lem.}Let $v\in H$ and $\lambda\in\left[a,b\right]$.
Assume there exists $\varepsilon>0$ such that $A$ has discrete spectrum
 in $\left[a-\sqrt{\varepsilon},b+\sqrt{\varepsilon}\right]$ and
$\left\Vert \left(A-\lambda\right)v\right\Vert \leq\varepsilon\left\Vert v\right\Vert $.
Then 
\begin{eqnarray}
\left\Vert \Pi_{\left[a-\sqrt{\varepsilon},b+\sqrt{\varepsilon}\right]}^{A}v-v\right\Vert  & \leq & \sqrt{\varepsilon}\left\Vert v\right\Vert \qquad\textrm{and}\label{eq:proj approximates}\\
\left\Vert \left(\rho\left(A\right)-\rho\left(\lambda\right)\right)v\right\Vert  & \leq & 3\sqrt{\varepsilon}\left\Vert \rho\right\Vert _{C^{0,1}}\left\Vert v\right\Vert \label{eq:almost ev of f(A)}
\end{eqnarray}
for any Holder continuous function $\rho\in C_{c}^{0,1}\left(\left[a,b\right]\right)$.\end{lem}
\begin{proof}
We abbreviate $\Pi=\Pi_{\left[a-\sqrt{\varepsilon},b+\sqrt{\varepsilon}\right]}^{A}$.
Let $H_{0}:=H_{\left[a-\sqrt{\varepsilon},b+\sqrt{\varepsilon}\right]}^{A}=\Pi H$
which by assumption is a finite dimensional vector space. Let $H_{0}^{\perp}$
be the orthogonal complement of $H_{0}.$ By assumption, $\mbox{Spec}\left(\left.\left(A-\lambda\right)^{2}\right|_{H_{0}^{\perp}}\right)\cap\left[-\varepsilon,\varepsilon\right]=\emptyset$.
Hence by the mini-max principle for self-adjoint operators bounded
from below (cf. Lemma 4.21 in \cite{Dimassi-Sjostrand}), we have
$\varepsilon\leq\left.\left(A-\lambda\right)^{2}\right|_{H_{0}^{\perp}}$.
Hence 
\begin{eqnarray*}
\left\Vert \Pi v-v\right\Vert ^{2}\varepsilon & \leq & \left\Vert \left(A-\lambda\right)\left(\Pi v-v\right)\right\Vert ^{2}\\
 & \leq & \left\Vert \left(A-\lambda\right)\left(\Pi v-v\right)\right\Vert ^{2}+\left\Vert \left(A-\lambda\right)\Pi v\right\Vert ^{2}=\left\Vert \left(A-\lambda\right)v\right\Vert ^{2}\leq\varepsilon^{2}\left\Vert v\right\Vert ^{2}
\end{eqnarray*}
since $\left(A-\lambda\right)\left(\Pi v-v\right)$ and $\left(A-\lambda\right)\Pi v$
are orthogonal. This gives 
\begin{equation}
\left\Vert \Pi v-v\right\Vert <\sqrt{\varepsilon}\left\Vert v\right\Vert .\label{eq:projection small norm-1-1}
\end{equation}
To prove \prettyref{eq:almost ev of f(A)} first note that $\left\Vert \Pi'v-v\right\Vert <\sqrt{\varepsilon}\left\Vert v\right\Vert $,
for $\Pi'=\Pi_{\left[\lambda-\sqrt{\varepsilon},\lambda+\sqrt{\varepsilon}\right]}^{A}$,-
by the same argument. We now have 
\begin{eqnarray*}
\left\Vert \left(\rho\left(A\right)-\rho\left(\lambda\right)\right)v\right\Vert  & \leq & \left\Vert \left(\rho\left(A\right)-\rho\left(\lambda\right)\right)\left(\Pi'v-v\right)\right\Vert +\left\Vert \left(\rho\left(A\right)-\rho\left(\lambda\right)\right)\Pi'v\right\Vert \\
 & \leq & 2\sqrt{\varepsilon}\left\Vert \rho\right\Vert _{C^{0,1}}\left\Vert v\right\Vert +\sqrt{\varepsilon}\left\Vert \rho\right\Vert _{C^{0,1}}\left\Vert v\right\Vert .
\end{eqnarray*}

\end{proof}
Before stating the next lemma we need the following definition.
\begin{defn}
\label{Def AOSE}Given $0<\varepsilon<1$, a set of vectors $w_{1},w_{2},\ldots,w_{N}\in H$
is called an $\varepsilon$-almost orthonormal set of eigenvectors
($\varepsilon$-AOSE for short) of $A$ if 
\begin{enumerate}
\item $\left|\left\Vert w_{j}\right\Vert ^{2}-1\right|<\varepsilon$ for
all $j$
\item $\left|\left\langle w_{j},w_{k}\right\rangle \right|<\varepsilon$
for all $j\neq k$
\item $\left\Vert \left(A-\mu_{j}\right)w_{j}\right\Vert <\varepsilon$
for some $\mu_{j}\in\mathbb{R}$, for all $j$.
\end{enumerate}
\end{defn}
Now we have another lemma. 
\begin{lem}
\label{lem:AOSE for complement}Assume $H_{0}\subset H$ has finite
dimension $M$ and is mapped onto itself by $A$. Let $w_{1},w_{2},\ldots,w_{N}\in H_{0}$
be an $\varepsilon$-AOSE of $A$ for some $\varepsilon<\frac{1}{2\left(M+1\right)}$.
Then there exist orthonormal $w_{1}',w_{2}',\ldots,w_{M-N}'\in H_{0}$
such that $\left\Vert \left(A-\mu_{j}'\right)w_{j}'\right\Vert <4M\varepsilon$
for some $\mu_{j}'\in\mathbb{R}$, for all $j$. Furthermore $\left\langle w_{j},w_{k}'\right\rangle =0$
for each $j,k$. \end{lem}
\begin{proof}
It follows from $\varepsilon<\frac{1}{2\left(M+1\right)}$ that $w_{1},w_{2},\ldots,w_{N}$
are linearly independent. Let $W$ denote their span and $W^{\perp}\subset H_{0}$
its orthogonal complement. Let $\Pi,\Pi^{\perp}$ be the orthogonal
projections onto $W,W^{\perp}$ and consider the operator $A_{0}:=\Pi^{\perp}A\Pi^{\perp}:W^{\perp}\rightarrow W^{\perp}$.
Let $w_{1}',w_{2}',\ldots,w_{M-N}'\in W^{\perp}$ be an orthogonal
basis of eigenvectors of $A_{0}$. Hence 
\[
\Pi^{\perp}Aw_{j}'=\mu_{j}'w_{j}'
\]
for some $\mu_{j}'\in\mathbb{R}$, for all $j$. Also 
\[
\left|\left\langle Aw_{j}',w_{k}\right\rangle \right|=\left|\left\langle w_{j}',\left(A-\mu_{k}\right)w_{k}\right\rangle \right|<\varepsilon.
\]
It then follows that $\left\Vert \Pi Aw_{j}'\right\Vert \leq2M\varepsilon\sqrt{1+\varepsilon}<4M\varepsilon$
giving the result.
\end{proof}
Now we prove another lemma.
\begin{lem}
\label{lem:almost orth. vec. lemma} Given $N\in\mathbb{N}$, let
$0<\varepsilon<\left(\frac{1}{\left\Vert A\right\Vert +\left|a\right|+\left|b\right|+N+1}\right)^{4}$.
Let $w_{1},w_{2},\ldots,w_{N}\in H$ be an $\varepsilon$-AOSE for
$A$. Assume that $A$ has discrete spectrum  in $\left[a-\varepsilon^{\frac{1}{8}},b+\varepsilon^{\frac{1}{8}}\right]$.
Then there exist orthonormal vectors $\overline{w}_{1},\overline{w}_{2},\ldots,\overline{w}_{N}\in H$,
which span the same subspace of $H$ as $w_{1},w_{2},\ldots,w_{N}$.
Moreover $\left\Vert w_{j}-\overline{w}_{j}\right\Vert <\sqrt{\varepsilon}$
and $\left\Vert \left(\rho\left(A\right)-\rho\left(\mu_{j}\right)\right)\overline{w}_{j}\right\Vert \leq3\varepsilon^{\frac{1}{8}}\left\Vert \rho\right\Vert _{C^{0,1}}$
for $1\leq j\leq N,$ and any Holder continuous function $\rho\in C_{c}^{0,1}\left(\left[a,b\right]\right)$. \end{lem}
\begin{proof}
Again it follows easily that the vectors $w_{j}$,$1\leq j\leq N$,
are linearly independent. Let $W\subset H$ be their span and choose
an orthonormal basis $e_{i}$, $1\leq j\leq N$, for $W$. We write
\[
w_{j}=\sum_{k=1}^{N}m_{jk}e_{k}.
\]
If we consider the matrix $M=\left[m_{jk}\right]$, then assumptions
1 and 2 of Definition \ref{Def AOSE} are equivalent to $\left|M^{*}M-I\right|<\varepsilon$.
Consider the polar decomposition $M=UP$ where $U$ is unitary and
$P$ is a positive semi-definite Hermitian matrix. We have $\left|P^{*}P-I\right|<\varepsilon$
and hence $\left\Vert P^{*}P-I\right\Vert <N\varepsilon$. Thus any
eigenvalue $\lambda^{P}$ of $P$, being nonnegative, satisfies $\left|\lambda^{P}-1\right|<\varepsilon$
and we have $\left\Vert P-I\right\Vert <N\varepsilon$. Thus $\left\Vert M-U\right\Vert =\left\Vert UP-U\right\Vert <N\varepsilon$.
If we now let $U=\left[u_{jk}\right]$ and $\overline{w}_{j}=\sum_{k=1}^{N}u_{jk}e_{k}$,
then the $\overline{w}_{j}$ are clearly orthonormal and satisfy $\left\Vert w_{j}-\overline{w}_{j}\right\Vert <\sqrt{\varepsilon}$.
This last inequality along with assumption 3 of Definition \ref{Def AOSE}
easily gives 
\[
\left\Vert \left(A-\mu_{j}\right)\overline{w}_{j}\right\Vert <\varepsilon^{\frac{1}{4}}.
\]
Now \prettyref{lem:projection close lem.} gives 
\begin{eqnarray}
\left\Vert \Pi\overline{w}_{j}-\overline{w}_{j}\right\Vert  & < & \varepsilon^{\frac{1}{8}}\quad\mbox{and}\label{eq:projection small norm-1}\\
\left\Vert \left(\rho\left(A\right)-\rho\left(\mu_{j}\right)\right)\overline{w}_{j}\right\Vert  & < & 3\varepsilon^{\frac{1}{8}}\left\Vert \rho\right\Vert _{C^{0,1}}.
\end{eqnarray}

\end{proof}
Next, let $H'$ be another separable Hilbert space. Let $U:H\rightarrow H'$
be a bounded operator. Let $B,D:H'\rightarrow H'$ and $C:H\rightarrow H$
be bounded self-adjoint operators. Define $A'=UAU^{*}:H'\rightarrow H'$,
$B'=U^{*}BU:H\rightarrow H$ , $C'=UCU^{*}:H'\rightarrow H'$ and
$D'=U^{*}DU:H\rightarrow H$. In the next proposition we assume that
there exists $\delta>0$ such that $A,A',B$ and $B'$ have discrete
spectrum  in $\left[a-\delta,b+\delta\right]$. We also abbreviate
$N^{A}=N_{\left[a-\delta,b+\delta\right]}^{A}$ and $\Pi^{A}=\Pi_{\left[a-\delta,b+\delta\right]}^{A}$
and similarly define $N^{A'},N^{B},N^{B'},\Pi^{A'},\Pi^{B},\Pi^{B'}$.
\begin{prop}
\label{prop:apdx prop for R^n red.}Suppose there exists $0<\varepsilon<L^{-2048}$,
with
\begin{multline*}
L=25\Bigl\{\left\Vert A\right\Vert +\left\Vert A'\right\Vert +\left\Vert B\right\Vert +\left\Vert B'\right\Vert +\left\Vert C\right\Vert +\left\Vert D\right\Vert \\
+N^{A}+N^{A'}+N^{B}+N^{B'}+\left|a\right|+\left|b\right|+\delta^{-1}+1\Bigr\},
\end{multline*}
 such that 
\begin{enumerate}
\item $\left\Vert \left(U^{*}U-I\right)\Pi^{A}\right\Vert \left(\left\Vert A\right\Vert \left\Vert U\right\Vert +1\right)<\varepsilon$
and \\
$\left\Vert \left(UU^{*}-I\right)\Pi^{B}\right\Vert \left(\left\Vert B\right\Vert \left\Vert U^{*}\right\Vert +1\right)<\varepsilon$
\item $\left\Vert \left(A'-B\right)\Pi^{A'}\right\Vert <\varepsilon$ and
$\left\Vert \left(A-B'\right)\Pi^{B'}\right\Vert <\varepsilon$
\item $\left\Vert \left(C'-D\right)\Pi^{A}\right\Vert <\varepsilon$ and
$\left\Vert \left(C-D'\right)\Pi^{B}\right\Vert <\varepsilon$.
\end{enumerate}

Then we have 
\[
\left|\mbox{tr}\left[C\rho\left(A\right)\right]-\mbox{tr}\left[D\rho\left(B\right)\right]\right|\leq\varepsilon^{\frac{1}{2048}}\left\Vert \rho\right\Vert _{C^{1}}
\]
for any $\rho\in C_{c}^{1}\left(\left[a,b\right]\right)$.

\end{prop}
\begin{proof}
Let $\left(\mbox{DiscSpec}\left(A\right),m^{A}\right)\cap\left[a,b\right]=\left\{ \lambda_{a_{1}}^{A},\lambda_{a_{2}}^{A},\ldots,\lambda_{a_{N}}^{A}\right\} $,
with $N=N_{\left[a,b\right]}^{A}$, as multisets. Let $\rho^{+}\left(x\right)=\frac{\rho\left(x\right)+\left|\rho\left(x\right)\right|}{2}$
and $\rho^{-}\left(x\right)=\frac{\rho\left(x\right)-\left|\rho\left(x\right)\right|}{2}$.
We then have $\rho^{+},\rho^{-}\in C_{c}^{0,1}\left(\left[a,b\right]\right)$
with $\left\Vert \rho^{+}\right\Vert _{C^{0,1}}\leq\left\Vert \rho\right\Vert _{C^{1}},\left\Vert \rho^{-}\right\Vert _{C^{0,1}}\leq\left\Vert \rho\right\Vert _{C^{1}}$.
We further decompose $C=C^{+}+C^{-}$, $D=D^{+}+D^{-}$ into their
positive and non-positive parts. Clearly 
\[
\textrm{tr}\left[C^{+}\rho^{+}\left(A\right)\right]=\sum_{j=1}^{N}\rho^{+}\left(\lambda_{a_{j}}\right)\left\langle v_{a_{j}},C^{+}v_{a_{j}}\right\rangle .
\]
Next we consider $w_{j}=Uv_{a_{j}}\in H'$. From assumption 1 we have
\[
\left\Vert \left(A'-\lambda_{a_{j}}\right)w_{j}\right\Vert =\left\Vert \left(UAU^{*}-\lambda_{a_{j}}\right)Uv_{a_{j}}\right\Vert \leq\left\Vert \left(U^{*}U-I\right)\Pi_{\left[a,b\right]}^{A}\right\Vert \left\Vert A\right\Vert \left\Vert U\right\Vert <\varepsilon.
\]
Similar estimates give $\left|\left\Vert w_{j}\right\Vert ^{2}-1\right|<\varepsilon$,
and $\left|\left\langle w_{j},w_{k}\right\rangle \right|<\varepsilon$
for $j\neq k$. Now by \prettyref{lem:projection close lem.} we have
$\left\Vert \Pi w_{j}-w_{j}\right\Vert <\left(2\varepsilon\right)^{\frac{1}{2}}$
with $\Pi=\Pi_{\left[a-\sqrt{2\varepsilon},b+\sqrt{2\varepsilon}\right]}^{A'}$.
Following this and using assumption 3 we have 
\begin{eqnarray*}
\left\Vert \left(B-\lambda_{a_{j}}\right)w_{j}\right\Vert  & \leq & \left\Vert \left(A'-\lambda_{a_{j}}\right)w_{j}\right\Vert +\left\Vert \left(B-A'\right)\Pi w_{j}\right\Vert +\left\Vert \left(B-A'\right)\left(\Pi w_{j}-w_{j}\right)\right\Vert \\
 & \leq & \varepsilon+\varepsilon\sqrt{1+\varepsilon}+\left(2\varepsilon\right)^{\frac{1}{2}}\left(\left\Vert A'\right\Vert +\left\Vert B\right\Vert \right)\\
 & < & \varepsilon^{\frac{1}{4}}\leq\varepsilon^{\frac{1}{8}}\left\Vert w_{j}\right\Vert .
\end{eqnarray*}
Next define $w_{j}^{0}:=\Pi_{\left[a-\varepsilon^{\frac{1}{16}},b+\varepsilon^{\frac{1}{16}}\right]}^{B}w_{j}$.
By \prettyref{lem:projection close lem.}
\begin{equation}
\left\Vert w_{j}^{0}-w_{j}\right\Vert \leq\varepsilon^{\frac{1}{16}}\left\Vert w_{j}\right\Vert .\label{eq:diff w_j^0-w_j}
\end{equation}
From here it follows immediately that $w_{1}^{0},w_{2}^{0},\ldots,w_{N}^{0}$
form an $\varepsilon^{\frac{1}{64}}$-ASOE of $B$. If we let $H_{0}=H_{\left[a-\varepsilon^{\frac{1}{16}},b+\varepsilon^{\frac{1}{16}}\right]}^{B}$,
then by \prettyref{lem:almost orth. vec. lemma} there exist orthonormal
$\overline{w}_{1},\overline{w}_{2},\ldots,\overline{w}_{N}\in H_{0}$
which span the same subspace of $H_{0}$ as the $w_{j}^{0}$'s. Furthermore
\begin{equation}
\left\Vert w_{j}^{0}-\overline{w}_{j}\right\Vert <\varepsilon^{\frac{1}{128}}\label{eq:diff w_j^0-bar=00007Bw_j=00007D}
\end{equation}
and $\left\Vert \left(\rho^{+}\left(B\right)-\rho^{+}\left(\lambda_{a_{j}}\right)\right)\overline{w}_{j}\right\Vert \leq3\left\Vert \rho\right\Vert _{C^{1}}\varepsilon^{\frac{1}{512}}$.
From \prettyref{eq:diff w_j^0-w_j} and \prettyref{eq:diff w_j^0-bar=00007Bw_j=00007D}
we also have $\left\Vert w_{j}-\overline{w}_{j}\right\Vert <\varepsilon^{\frac{1}{256}}$.
From \prettyref{lem:AOSE for complement} there exist orthonormal
$w_{1}',w_{2}',\ldots,w_{M-N}'$ with $M=N_{\left[a-\varepsilon^{\frac{1}{16}},b+\varepsilon^{\frac{1}{16}}\right]}^{B}$
such that $\left\langle w_{i}',\overline{w}_{j}\right\rangle =0$
and $\left\Vert \left(B-\mu_{j}'\right)w_{j}'\right\Vert <4M\varepsilon^{\frac{1}{64}}<\varepsilon^{\frac{1}{128}}$.
Hence \prettyref{lem:projection close lem.} $\left\Vert \left(\rho^{+}\left(B\right)-\rho^{+}\left(\mu_{j}'\right)\right)w_{j}'\right\Vert \leq3\left\Vert \rho\right\Vert _{C^{1}}\varepsilon^{\frac{1}{256}}$.
We now have
\begin{eqnarray*}
\textrm{tr}\left[D^{+}\rho^{+}\left(B\right)\right] & = & \sum_{j=1}^{N}\left\langle \overline{w}_{j},D^{+}\rho^{+}\left(B\right)\overline{w}_{j}\right\rangle +\sum_{j=1}^{M-N}\left\langle w_{j}',D^{+}\rho^{+}\left(B\right)w_{j}'\right\rangle \\
 & \geq & \sum_{j=1}^{N}\rho^{+}\left(\lambda_{a_{j}}\right)\left\langle \overline{w}_{j},D^{+}\overline{w}_{j}\right\rangle +\sum_{j=1}^{M-N}\rho^{+}\left(\mu_{j}'\right)\left\langle w_{j}',D^{+}w_{j}'\right\rangle \\
 &  & \qquad\qquad\qquad\qquad\qquad\qquad-3\varepsilon^{\frac{1}{512}}M\left\Vert D\right\Vert \left\Vert \rho\right\Vert _{C^{1}}\\
 & \geq & \sum_{j=1}^{N}\rho^{+}\left(\lambda_{a_{j}}\right)\left\langle \overline{w}_{j},D^{+}\overline{w}_{j}\right\rangle -3\varepsilon^{\frac{1}{512}}M\left\Vert D\right\Vert \left\Vert \rho\right\Vert _{C^{1}}\\
 & \geq & \sum_{j=1}^{N}\rho^{+}\left(\lambda_{a_{j}}\right)\left\langle w_{j},D^{+}w_{j}\right\rangle -6\varepsilon^{\frac{1}{512}}M\left\Vert D\right\Vert \left\Vert \rho\right\Vert _{C^{1}}\\
 & \geq & \sum_{j=1}^{N}\rho^{+}\left(\lambda_{a_{j}}\right)\left\langle v_{a_{j}},C^{+}v_{a_{j}}\right\rangle -6\varepsilon^{\frac{1}{512}}M\left(\left\Vert D\right\Vert +1\right)\left\Vert \rho\right\Vert _{C^{1}}\\
 & \geq & \textrm{tr}\left[C^{+}\rho^{+}\left(A\right)\right]-\varepsilon^{\frac{1}{1024}}\left\Vert \rho\right\Vert _{C^{1}}.
\end{eqnarray*}
Reversing the roles of $H$ and $H'$ gives 
\[
\left|\textrm{tr}\left[D^{+}\rho^{+}\left(B\right)\right]-\textrm{tr}\left[C^{+}\rho^{+}\left(A\right)\right]\right|\leq\varepsilon^{\frac{1}{1024}}\left\Vert \rho\right\Vert _{C^{1}}.
\]
Similar estimates with $C^{+}\rho^{-}\left(A\right)$,$C^{-}\rho^{+}\left(A\right)$
and $C^{-}\rho^{-}\left(A\right)$ give the result.
\end{proof}
Finally, we now give a criterion implying the discreteness of spectrum
for pseudodifferential operators required by the preceding propositions
in this appendix.
\begin{prop}
\label{prop:Dicrete spectrum criterion}Let $A\in\Psi_{\textrm{cl}}^{m}\left(\mathbb{R}^{n};\mathbb{C}^{l}\right)$
and $I=\left[a,b\right]\subset\mathbb{R}$ a closed interval such
that the $I$ energy band 
\[
\Sigma_{I}^{A}\coloneqq\bigcup_{\lambda\in I}\Sigma_{\lambda}^{A}
\]
 is bounded. Then for $h<h_{0}$ sufficiently small 
\[
\mbox{EssSpec}\left(A\right)\cap I=\emptyset.
\]
\end{prop}
\begin{proof}
Let $\sigma\left(A\right)=a\left(x,\xi\right)\in C^{\infty}\left(\mathbb{R}^{2n}\right)$
and $\Sigma_{I}\left(a\right)\subset B_{R}$ some open ball of finite
radius $R$ around the origin. For $\lambda\in I$ and $\left(x,\xi\right)\notin B_{R}$,
we hence have that $a_{-1}:=\left(a\left(x,\xi\right)-\lambda\right)^{-1}$
exists. Let $\chi\in C_{c}^{\infty}\left(-4R,4R\right)$ such that
$\chi\left(x\right)=1$ for $x<2R$. Set $\phi\left(x\right)=1-\chi\left(x\right)$
and define 
\[
A_{-1}=\left[\phi\left(\left|\left(x,\xi\right)\right|\right)a_{-1}\left(x,\xi\right)\right]^{W}\in\Psi_{\textrm{cl}}^{0}\left(\mathbb{R}^{n};\mathbb{C}^{l}\right).
\]
 Then since it has vanishing symbol, we have 
\[
\left(A-\lambda\right)A_{-1}-\left(I-\chi\left(\left|\left(x,\xi\right)\right|\right)^{W}\right)=hR\in h\Psi_{\textrm{cl}}^{0}\left(\mathbb{R}^{n};\mathbb{C}^{l}\right).
\]
Next, we clearly have $I+hR$ is invertible for $h<h_{0}$ sufficiently
small. Also, $\chi\left(\left|\left(x,\xi\right)\right|\right)^{W}$
is trace class by \cite{HormanderIII} Lemma 19.3.2. Hence if $S:=A_{-1}\left(I+hR\right)^{-1}$,
then $\left(A-\lambda\right)S-I$ is trace class. By a similar argument,
$S\left(A-\lambda\right)-I$ is trace class. Hence by Proposition
19.1.14 of \cite{HormanderIII}, $A-\lambda$ is Fredholm. 
\end{proof}
\bibliographystyle{siam}
\addcontentsline{toc}{section}{\refname}\bibliography{biblio}

\end{document}